\documentclass{amsart}
\usepackage{amsmath}
\usepackage{amsthm}
\usepackage{amssymb}
\usepackage{enumerate}

\theoremstyle{plain}

\theoremstyle{change}
\newtheorem{thm}{Theorem}[subsection]
\newtheorem{prop}[thm]{Proposition}
\newtheorem{lem}[thm]{Lemma}
\newtheorem{cor}[thm]{Corollary}

\theoremstyle{remark}
\newtheorem{rem}[thm]{\bf{Remark}}
\newtheorem{rems}[thm]{\bf{Remarks}}

\theoremstyle{definition}
\newtheorem{defn}[thm]{Definition}
\numberwithin{equation}{subsection}

\renewcommand{\labelenumi}{(\theenumi)}

\def\Vec#1{\mbox{\boldmath $#1$}}

\usepackage{amscd}

\title{On types for unramified $p$-adic unitary groups}
\author{Kazutoshi Kariyama}

\begin{document}
\maketitle

\begin{abstract}
Let $F$ be a non-archimedean local field of residue characteristic neither 2 nor 3 equipped with a galois involution with fixed field $F_0$, and let $G$ be a symplectic group over $F$ or an unramified unitary group over $F_0$. Following the methods of Bushnell-Kutzko for $GL(N,F)$, we define an analogue of a simple type attached to a certain skew simple stratum, and realize a type in $G$.  In particular, we obtain an irreducible supercuspidal representation of $G$ like $GL(N,F)$.
\end{abstract}

\vspace{2mm}
{\footnotesize

}
\vspace{2mm}

                                                   % Introduction
\begin{center}
I{\scriptsize NTRODUCTION}
\end{center}

Let $N$ be an integer $\ge 2$, and $V$ an $N$-dimensional vector space over a non-archimedean local field $F$. Put $A = \mathrm{End}_F(V)$ and $G = \mathrm{Aut}_F(V) \simeq GL(N,F)$. 

From Bushnell-Kutzko \cite{BK1}, in which a complete classification of the irreducible smooth representations of $G$ is given, we obtain the following results: 
A stratum in $A$ is a 4-tuple $[\mathfrak{A},n,0,\beta]$ which consists of a hereditary $\mathfrak{o}_F$-order $\mathfrak{A}$ in $A$, an integer $n > 0$, and an element $\beta \in \mathfrak{P}^{-n}$, where $\mathfrak{o}_F$ is the maximal order of $F$, and $\mathfrak{P}$ is the Jacobson radical of $\mathfrak{A}$.
We define a compact open subgroup $J = J(\beta,\mathfrak{A})$ of $G$ and its normal subgroups $H^1(\beta,\mathfrak{A}),\ J^1(\beta,\mathfrak{A})$ (\cite[(3.1)]{BK1}), associated with a simple stratum $[\mathfrak{A},n,0,\beta]$ (\cite[(1.5)]{BK1}).
Let $\theta$ be a simple character, which is an abelian character of $H^1 = H^1(\beta,\mathfrak{A})$ (\cite[(3.2)]{BK1}). Then there is a unique irreducible representation $\eta$ of $J^1 = J^1(\beta,\mathfrak{A})$ such that $\eta \vert H^1$ contains $\theta$ (\cite[(5.1)]{BK1}), 
and is an irreducible representation $\kappa$ of $J$, called a $\beta$-extension of $\eta$, which is an extension of $\eta$ and has the $G$-intertwining $J B^\times J$ (\cite[(5.2)]{BK1}), where $B$ is the $G$-centralizer of $\beta$.

Suppose that $\mathfrak{A}$ is principal. The group $J/J^1$ is isomorphic to a Levi subgroup of $GL(R,k_E)$, where $R = \dim_E(V)$ and $k_E$ denotes the residue class field of $E$. A certain irreducible cuspidal representation of $J/J^1$ is chosen and is inflated to the representation, $\sigma$, of $J$.
Then an irreducible representation $\lambda$ of $J$ is defined by $\lambda = \kappa \otimes \sigma$, which is called a simple type (of positive level) (\cite[(5.5)]{BK1}). If $\mathfrak{A} \cap B^\times$ is a maximal compact subgroup of $B^\times$, then the representation $(J,\lambda)$ is a $[G,\pi]_G$-type in $G$, for some irreducible supercuspidal representation $\pi$ of $G$ (\cite[(6.2)]{BK1}, \cite{BK2}). Such a simple type $(J,\lambda)$ is called maximal.

Associated with a simple stratum $[\mathfrak{A},n,0,\beta]$, there is a choice of a parabolic subgroup $P = MN$ of $G$ with a Levi component $M$  (\cite[(7.1)]{BK1}). From a simple type $(J,\lambda)$, we can define a certain pair of a compact open subgroup $J_P$ of $G$ and an irreducible representation $\lambda_P$ of $J_P$ (\cite[(7.2)]{BK1}).
Then there is an irreducible supercuspidal representation $\pi$ of $M$ such that $(J_P \cap M,\lambda_P \vert J_P \cap M)$ is an $[M,\pi]_M$-type in $M$ (\cite[(7.2)]{BK1}, \cite{BK2}), and $(J_P,\lambda_P)$ is a $G$-cover of $(J_P \cap M,\lambda_P \vert J_P \cap M)$ (\cite[(7.3)]{BK1}, \cite{BK2}). Hence $(J_P,\lambda_P)$ is an $[M,\pi]_G$-type in $G$ (\cite[(8.3)]{BK2}).
Moreover, the Hecke algebra of $(J_P,\lambda_P)$ is isomorphic to an affine Hecke algebra (\cite[(5.6)]{BK1}).

Let $F$ be a non-archimedean local field of residual characteristic not 2 equipped with a galois involution with fixed field $F_0$, and $V$ a finite dimensional $F$-vector space equipped with a non-degenerate hermitian form $h$. Let $G$ be the unitary group of $(V,h)$ over $F_0$. Put $A = \mathrm{End}_F(V)$ and $\widetilde{G} = \mathrm{Aut}_F(V)$ here.
From Stevens \cite{St1}, \cite{St2} and \cite{St3}, we obtain the following results: A skew semi-simple stratum $[\mathfrak{A},n,0,\beta]$ in $A$ is defined, and we obtain the subgroups $H^1(\beta,\mathfrak{A}), J^1(\beta,\mathfrak{A})$ and $J(\beta,\mathfrak{A})$ of $\widetilde{G}$ as above. Restricting them to $G$, we obtain the subgroups $H_-^1 = H_-^1(\beta,\mathfrak{A}), J_-^1 = J_-^1(\beta,\mathfrak{A})$, and $J_- = J_-(\beta,\mathfrak{A})$ of $G$, respectively.
A skew semi-simple character $\theta_-$ of $H_-^1$ is defined as well, and we can similarly give a unique irreducible representation $\eta_-$ of $J_-^1$ such that $\eta_- \vert H_-^1$ contains $\theta_-$.
In particular, if the $A$-centralizer of $\beta$ is a maximal commutative semisimple algebra of $A$, there is an irreducible representation $\kappa_-$ of $J_-$ such that $\kappa_- \vert J_-^1 = \eta_-$, which is a $\beta$-extension of $\eta_-$ in a sense. The representation $(J_-,\kappa_-)$ induces an irreducible supercuspidal representation of $G$, and so it is a type in $G$ (cf. \cite{Ad2}, \cite{Ka}, \cite{Yu}).
In general, it is very difficult to prove the existence of a $\beta$-extension of $\eta_-$ even for a skew simple stratum $[\mathfrak{A},n,0,\beta]$ in $A$.

Now suppose that $h$ is a non-degenerate alternating form on a $2n$-demensional $F$-vector space $V$. Then $G$ is a symplectic group $Sp_{2n}(F)$. Recently, by Blondel \cite{Bl}, the following results for $G = Sp_{2n}(F)$ are obtained:  
Let $\pi$ be a self-contragradient supercuspidal irreducible representation of $GL(n,F)$ (cf. \cite{GK}, \cite{Ad1}), and $(J_0,\lambda_0)$ a maximal simple type in $GL(n,F)$ for the inertial class $[GL(n,F),\pi]_{GL(n,F)}$.
We can take a special simple stratum $[\mathfrak{A},n,0,\beta]$ in $A = \mathrm{End}_F(V)$ such that the associated parabolic subgroup $P = MN$ of $GL(2n,F)$ satisfies $M \simeq GL(n,F) \times GL(n,F)$ and leads to a Siegel parabolic subgroup $P_0 = M_0N_0$  of $G$ with $M_0 \simeq GL(n,F)$.
Then there is a simple type $(J,\lambda)$ in $GL(2n,F)$ attached to $[\mathfrak{A},n,0,\beta]$ such that $J \cap M \simeq J_0 \times J_0$ and $\lambda \vert (J \cap M) \simeq \lambda_0 \otimes \lambda_0$.
Thus we can construct an irreducible representation $(J_{P},\lambda_{P})$ in $GL(2n,F)$ from $(J,\lambda)$ as above, and restrict $(J_{P},\lambda_{P})$ to $G$ so as to obtain an $[M_0,\pi]_G$-type in $G$ as a $G$-cover of $(J_0,\lambda_0)$. The methods of \cite{Bl} construct a type in $G$ without using a simple type for $G$.

Recently, the constructing problem of (simple) types for $GL(N,D)$, with $D$ a central division $F$-algebra, is solved by S\'echerre \cite{Sc1, Sc2, Sc3}.

In this paper, let $F$ be a non-archimedean local field of residual characteristic neither 2 nor 3 equipped with a galois involution with fixed field $F_0$. We assume that $F/F_0$ is an unramified field extension, and let $h$ be a non-degenerate $F/F_0$-skew-hermitian form on a vector space $V$ of dimension $2n$ over $F$ such that the anisotropic part is zero. Put $G = U(V,h)$. 
Following the methods of Bushnell-Kutzko \cite{BK1}, we define a simple type for $G$ attached to a certain skew simple stratum in $A = \mathrm{End}_F(V)$, which is called good (see Definition 2.1.1), and realize a type in $G$.
 A simple type in $Sp_{2n}(F)$, attached to a good skew simple stratum $[\mathfrak{A},n,0,\beta]$ with $\mathfrak{A}$ principal and with $e(\mathfrak{B} \vert \mathfrak{o}_{F[\beta]}) = 2$, gives the one constructed in Blondel \cite{Bl}, where $e(\mathfrak{B} \vert \mathfrak{o}_{F[\beta]})$ denotes the $\mathfrak{o}_{F[\beta]}$-period of the lattice chain in $V$ defining the $\mathfrak{A}$-centralizer $\mathfrak{B}$ of $\beta$.

The contents of this paper are as follows: In sections 1 and 2, from \cite{BK1} and \cite{St2}, we recall the definitions of the skew simple stratum $[\mathfrak{A},n,0,\beta]$, the compact open subgroups $H^t(\beta,\mathfrak{A}), J^t(\beta,\mathfrak{A})$ of $G$, for $t = 0, 1$, and the skew simple character $\theta_- \in \mathcal{C}_-(\mathfrak{A},0,\beta)$.
We define a good skew simple stratum $[\mathfrak{A},n,0,\beta]$, which implies that there are hereditary $\mathfrak{o}_F$-orders $\mathfrak{A}_m \subset \mathfrak{A} \subset \mathfrak{A}_M$ in $A = \mathrm{End}_F(V)$ such that $\Vec{U}(\mathfrak{B}_m) = \mathfrak{A}_m \cap B \cap G$ is an Iwahori subgroup of $B \cap G$ and $\Vec{U}(\mathfrak{B}_M) = \mathfrak{A}_M \cap B \cap G$ is a special (good) maximal compact subgroup of $B \cap G$, where $B$ is the $A$-centralizer of $\beta$. This property is used to prove the existence of a $\beta$-extension.

In section 3, let $[\mathfrak{A},n,0,\beta]$ be a good skew simple stratum in $A$. From \cite{St3}, there is a unique irreducible representation $\eta_-$ of $J_-^1(\beta,\mathfrak{A})$ associated with a skew simple character $\theta_-$. Modulo some claim, we can prove that there is a $\beta$-extension $\kappa_-$ of $\eta_-$, which is, by definition, a representation of $J_- = J_-(\beta,\mathfrak{A})$ satisfying (1) $\kappa_- \vert J_- = \eta_-$, (2) the $G$-intertwining of $\kappa_-$ contains $J_-.B \cap G.J_-$.

In section 4, we have a parabolic subgroup $P = MN_u$ of $G$, with Levi component $M$ and unipotent radical $N_u$, associated with a good skew simple stratum $[\mathfrak{A},n,0,\beta]$ in $A$. We see that $H_-^t(\beta,\mathfrak{A}), J_-^t(\beta,\mathfrak{A})$, $t = 0, 1$, have Iwahori decompositions relative to $P = MN_u$, and prove the claim in section 3.

In section 5, let $[\mathfrak{A},n,0,\beta]$ be a good skew simple stratum in $A$ with $\mathfrak{A}$ principal. We choose a certain irreducible cuspidal representation $\sigma_-$ of $J_-(\beta,\mathfrak{A})/J_-^1(\beta,\mathfrak{A})$. 
From this $\sigma_-$, together with a $\beta$-extension $\kappa_-$, we define an irreducible representation $\lambda_- = \kappa_- \otimes \sigma_-$ of $J_-(\beta,\mathfrak{A})$, which is an analogue of a simple type of positive level for $GL(N,F)$ of \cite[(5.5.10)]{BK1}.
Let $\Vec{W}$ be an affine Weyl group of $B \cap G$ with $B \cap G = \Vec{U}(\mathfrak{B}_m)\Vec{W}\Vec{U}(\mathfrak{B}_m)$, and put $\Vec{W}(\mathfrak{B}) = \{w \in \Vec{W} \mid \text{$w$ normalizes $\mathfrak{A} \cap M \cap B$} \}$.  
We prove that the $G$-intertwining of the simple type $(J_-,\lambda_-)$ is contained in $J_-\Vec{W}(\mathfrak{B})J_-$. It follows that if $\mathfrak{A} \cap B$ is a maximal compact subgroup of $G \cap B$, $(J_-,\lambda_-)$ induces an irreducible supercuspidal representation of $G$.
Moreover, we construct an irreducible representation $(J_{P,-},\lambda_{P,-})$, in the same way as \cite{BK1}, such that $(J_{P,-} \cap M,\lambda_{P,-} \vert J_{P,-} \cap M)$ is an $[M,\pi]_M$-type in $M$, for some irreducible supercuspidal representation $\pi$ of $M$.

In section 6, we study the Hecke algebra $\mathcal{H}(G,\lambda_{P,-})$ of $(J_{P,-},\lambda_{P,-})$, and then we prove that $(J_{P,-},\lambda_{P,-})$ is an $[M,\pi]_G$-type in $G$, and so is $(J_-,\lambda_-)$.

\section{Preliminaries}                 % Sec. 1

\subsection{}                 % Sec. 1.1

Let $F$ be a non-archimedean local field equipped with a galois involution\ $^-$, with the fixed field $F_0$. Let $\mathfrak{o}_F$ and $\mathfrak{p}_F$ be its maximal order and the maximal ideal of $\mathfrak{o}_F$, respectively, and $k_F = \mathfrak{o}_F/\mathfrak{p}_F$ the residue class field. Let $\varpi_F$ be a uniformizer of $F$.
We assume that the residual characteristic $p$ is not 2 and that $F/F_0$ is unramified (possibly $F = F_0$).

Let $N$ be an integer $\ge 4$. Let $V$ be an $N$-dimensional vector space over $F$, and put $A = \mathrm{End}_F(V) \simeq \mathbb{M}(N,F)$. Let $h$ be a non-degenerate anti-hermitian form on $V$ over $F/F_0$. 
We furthermore assume that the anisotropic part of $V$ is zero. 
Then $N$ must be even. Let\ $^-$\ be the adjoint (anti-)involution on $A$ defined by the form $h$. Put $\widetilde{G} = \mathrm{Aut}_F(V) \simeq GL(N,F)$, and define $\gamma$ to be the involution $x \mapsto \bar{x}\ ^{-1}$ on $\widetilde{G}$. Put $\Gamma = \{1, \gamma\}$.

We put
\begin{equation*}
G = \widetilde{G}^\Gamma = \{ g \in \widetilde{G} \mid h(gv,gw) = h(v,w),\ \text{for all}\ v, w \in V \}.
\end{equation*}
By the assumption, $G$ is a symplectic group over $F$ if $F = F_0$, and is an unramified unitary group over $F_0$ if $F \ne F_0$. We write $G = U(V,h)$.  We also put
\begin{equation*}
\mathfrak{G} = \{ a \in A \mid a + \bar{a} = 0 \}.
\end{equation*}
This is isomorphic to $\mathrm{Lie}\ G$.

Let $\mathbb{Z}$ and $\mathbb{C}$ denote the ring of rational integers and the field of complex numbers, respectively. For a ring $R$, let $R^\times$ denote the multiplicative group of invertible elements in $R$. For a finite field extension $E/F$, we denote by $\mathfrak{o}_E, \mathfrak{p}_E, k_E$ the objects for $E$ analogous to those above for $F$.

\subsection{}                        % Sec. 1.2

We recall notation in \cite{BK1} and \cite{Mo2}.

For an $\mathfrak{o}_F$-lattice in $V$, we define the dual lattice $L^\#$ by
\begin{equation*}
L^\# = \{ v \in V \mid h(v,L) \subset \mathfrak{o}_F \}
\end{equation*}
(cf. \cite[1.1]{Mo2}). An $\mathfrak{o}_F$-{\it lattice chain} in $V$ is a set $\mathcal{L} = \{L_i \vert i \in \mathbb{Z}\}$ of $\mathfrak{o}_F$-lattices in $V$ which satisfies
\begin{enumerate}
 \item $L_i \supsetneq L_{i+1}$, for all $i \in \mathbb{Z}$,
 \item there is a positive integer $e$ such that $L_{i+e} = \mathfrak{p}_FL_i$, for all $i \in \mathbb{Z}$.
\end{enumerate}
This integer $e = e(\mathcal{L})$ is unique and is called the $\mathfrak{o}_F$-{\it period} of $\mathcal{L}$. 

A $\mathfrak{o}_F$-lattice chain $\mathcal{L}$ in $V$ is called {\it self-dual} (with respect to the form $h$) if $L \in \mathcal{L}$ implies $L^\# \in \mathcal{L}$. If $\mathcal{L}$ is self-dual, from \cite[Proposition 1.4]{Mo2}, there is a unique slice of the form:
\begin{equation*}
L^\#_{r-1} \supsetneq \cdots \supsetneq L^\#_0 \supset L_0 \supsetneq \cdots \supsetneq L_{r-1} \supset \varpi_FL^\#_{r-1},
\end{equation*}
for some integer $r \ge 1$, where possibly $L^\#_0 = L_0$ and/or $L_{r-1} = \varpi_FL^\#_{r-1}$. This slice is called a {\it self-dual slice} of $\mathcal{L}$.

Associated with an $\mathfrak{o}_F$-lattice chain $\mathcal{L}$ in $V$, a filtration on $A$ is given by
\begin{equation*}
\mathfrak{P}^n = \{x \in A \mid xL_i \subset L_{i+n},\ \text{for all}\ i \in \mathbb{Z} \},
\end{equation*}
for $n \in \mathbb{Z}$. In particular, $\mathfrak{A} = \mathfrak{A}(\mathcal{L}) = \mathfrak{P}^0$ is a hereditary $\mathfrak{o}_F$-order in $A$, and $\mathfrak{P}$ is its Jacobson radical. An $\mathfrak{o}_F$-lattice chain $\mathcal{L}$ in $V$ determines a valuation map $\nu_\mathfrak{A}: A \to \mathbb{Z}$ by
\begin{equation*}
\nu_\mathfrak{A}(x) = \max \{n \in \mathbb{Z} \mid x \in \mathfrak{P}^n \},\ \text{for}\ x \in A,
\end{equation*}
with $\nu_\mathfrak{A}(0) = \infty$. 

We obtain a family of compact open subgroups $\mathfrak{A} \cap \widetilde{G} = \mathfrak{A}^\times$ and $1 + \mathfrak{P}^n$ for integers $n \ge 1$, of $\widetilde{G}$.
If $\mathcal{L}$ is self-dual, $\mathfrak{A}^\times$ and $1 + \mathfrak{P}^n, n \ge 1$, are fixed by $\gamma$. So we obtain a family of compact open subgroups of $G$:
\begin{center}
$\Vec{U}(\mathfrak{A}) = (\mathfrak{A}^\times)^\Gamma = \mathfrak{A} \cap G$, \\
\vspace{1mm}
$\Vec{U}^n(\mathfrak{A}) = (1 + \mathfrak{P}^n)^\Gamma =  (1 + \mathfrak{P}^n) \cap G$,
\end{center}
for integers $n \ge 1$. Then $\{\Vec{U}^n(\mathfrak{A}) \mid n \ge 1\}$ is a filtration on $G$ by normal subgroups of $\Vec{U(}\mathfrak{A})$.

For an $\mathfrak{o}_F$-order $\mathfrak{A} = \mathfrak{A}(\mathcal{L})$ in $A$, we put
\begin{equation*}
\mathfrak{K}(\mathfrak{A}) = \{x \in \widetilde{G} \mid xL \in \mathcal{L},\ \text{for all}\ L \in \mathcal{L}\}.
\end{equation*}
Then we have $\mathfrak{K}(\mathfrak{A}) = \{x \in \widetilde{G} \mid x\mathfrak{A}x^{-1} = \mathfrak{A}\}$.

\subsection{}                       % Sec. 1.3

Suppose that $\beta$ is an element in the Lie algebra $\mathfrak{G}$ such that the algebra $E = F[\beta]$ is a subfield of $A$. Then the involution $^-$ on $A$ fixes $E$. Put $E_0 = \{x \in E \mid \overline{x} = x\}$. We choose an $F$-linear form $\ell_0 : E_0 \to F$ which satisfies
\begin{equation*}
\ell_0(\mathfrak{o}_{E_0}) = \mathfrak{o}_{F_0},\ \ell_0(\mathfrak{p}_{E_0}^{-1}) = \mathfrak{p}_{F_0}^{-1}
\end{equation*}
as in \cite[2.3]{Bl}. We define an $F$-linear form $\ell: E \to F$ as follows: If $F = F_0$, put
\begin{equation*}
\ell = \ell_0 \circ \mathrm{tr}_{E/E_0}.
\end{equation*}
Otherwise, we extend $\ell_0$ to $E$ linearly. In fact, since $F/F_0$ is unramified and the residual characteristic $p$ of $F$ is not $2$, there is an element $\xi \in \mathfrak{o}_F^\times$ such that $F = F_0[\xi],\ E = E_0[\xi]$, and $\xi^2 \in F_0$. 
We note that $E/E_0$ is also unramified. Thus we have $\mathfrak{o}_F = \mathfrak{o}_{F_0} + \mathfrak{o}_{F_0}\xi,\ \mathfrak{o}_E = \mathfrak{o}_{E_0} + \mathfrak{o}_{E_0}\xi$. Hence $\ell: E \to F$ is given by
\begin{equation}
\ell(x + y\xi) = \ell_0(x) + \ell_0(y)\xi
\end{equation}
for all $x, y \in E_0$. Hereafter we fix this $F$-linear form $\ell: E \to F$.

From the $F$-linear form $\ell$ on $E = F[\beta]$ and the form $h$ on $V$, we can define an $E$-anti-hermitian form $\widetilde{h}_\beta$ on $V$ by
\begin{equation}
h(av,w) = \ell(a\widetilde{h}_\beta(v,w))
\end{equation}
for all $v, w \in V$ and all $a \in E$ (cf. \cite{SS}). Then $\widetilde{h}_\beta$ is non-degenerate. Let $B = B_\beta$ be the $A$-centralizer of $\beta$. Then we may identify $B$ with $\mathrm{End}_E(V)$.

By definition, we have
\begin{equation}
\ell_0^{-1}(\mathfrak{o}_{F_0}) = \mathfrak{o}_{E_0}.
\end{equation}

\begin{prop}             % Prop. 1.3.1
The form $\widetilde{h}_\beta$ is a non-degenerate $E/E_0$-anti-hermitian form on $V$, and there is a canonical isomorphism
\begin{equation*}
B^\times \cap\ G = \{x \in B^\times \mid \gamma(x) = x\} \simeq U(V,\widetilde{h}_\beta).
\end{equation*}
\end{prop}
\begin{proof}
In the case of $F = F_0$, this follows easily (cf. \cite[2.3]{Bl}). Suppose that $F \ne F_0$. By the assumption, $E/E_0$ is unramified, as was noted above. It follows from the definition of the $F$-linear form $\ell$ in 1.3 that $\ell(\overline{z}) = \overline{\ell(z)}$ for $z \in E$, whence this shows that $\overline{h}_\beta$, defined by (1.3.2), is a non-degenerate $E$-anti-hermitian form. The proof is completed.
\end{proof}

\begin{lem}          % Lem. 1.3.2
We have
\begin{equation*}
\ell^{-1}(\mathfrak{o}_F) = \mathfrak{p}_E^{1-e(E \mid E_0)},
\end{equation*}
where $e(E \mid E_0)$ denotes the ramification index of $E/E_0$.
\end{lem}
\begin{proof}
We again note that if $G$ is an unramified unitary group over $F_0$ (with $F \ne F_0$), $E/E_0$ must be unramified.

Write $e_0 = e(E \mid E_0)$. Since $p \ne 2$, $E/E_0$ is tamely ramified. Thus, by \cite[VIII, Section 1, Proposition 4]{We}, we have
\begin{equation}
\mathrm{tr}_{E/E_0}^{-1}(\mathfrak{o}_{E_0}) = \mathfrak{p}_E^{1-e_0}.
\end{equation}
Suppose first that $e_0 = 1$, i.e., $E/E_0$ is unramified. If $F = F_0$, the assertion follows directly from (1.3.3) and (1.3.4). Suppose that $F \ne F_0$. Then $\mathfrak{o}_E \subset \ell^{-1}(\mathfrak{o}_F)$ follows immediately. Conversely, let $z = x + y\xi \in \ell^{-1}(\mathfrak{o}_F)$, for $x, y \in E_0$. Then from (1.3.1), $\ell(z) = \ell_0(x) + \ell_0(y)\xi \in \mathfrak{o}_F$, and so $\ell_0(x), \ell_0(y) \in \mathfrak{o}_{F_0}$. Hence from (1.3.3) $x, y \in \mathfrak{o}_{E_0}$, that is, $z = x + y\xi \in \mathfrak{o}_E$.

Suppose that $e_0 = 2$, i.e., $E/E_0$ is ramified. Then we must have $F = F_0$. For, since $F/F_0$ is assumed to be unramified, it follows from (1.3.4) that $\mathrm{tr}_{E/E_0}^{-1}(\mathfrak{o}_{E_0}) = \mathfrak{p}_E^{-1}$. Thus from (1.3.3),
\begin{equation*}
\mathrm{tr}_{E/E_0}^{-1}(\ell_0^{-1}(\mathfrak{o}_F)) = \mathrm{tr}_{E/E_0}^{-1}(\mathfrak{o}_{E_0}) = \mathfrak{p}_E^{-1} = \mathfrak{p}_E^{1-e_0}.
\end{equation*}
The proof is completed.
\end{proof}

\subsection{}                     % Sec. 1.4

Suppose that $\beta$ is an element in the Lie algebra $\mathfrak{G}$ such that the algebra $E = F[\beta]$ is a subfield of $A$, as in 1.3. Let $L$ be an $\mathfrak{o}_E$-lattice in $V$. Then $L$ is also an $\mathfrak{o}_F$-lattice in $V$. We define the $\mathfrak{o}_E$-dual $L^\natural$ of $L$, with respect to $\widetilde{h}_\beta$, by
\begin{equation*}
L^\natural = \{v \in V \mid \widetilde{h}_\beta(v,L) \subset \mathfrak{o}_E\}.
\end{equation*}
There is a close relationship between $L^\#$ and $L^\natural$ as follows:

\begin{prop}                 % Prop. 1.4.1
For an $\mathfrak{o}_E$-lattice $L$ in $V$, we have
\begin{equation*}
L^\# = \varpi_E^{1-e(E \mid E_0)}L^\natural,
\end{equation*}
where $\varpi_E$ is a uniformizer of $E$.
\end{prop}
\begin{proof}
From (1.3.2), we have an equivalence: $v \in L^\# \Leftrightarrow \mathfrak{o}_F \supset h(v,L) = \ell(\widetilde{h}_\beta(v,L))$. From Lemma 1.3.2, the latter is equivalent to
\begin{eqnarray*}
\mathfrak{p}_E^{1-e_0} \supset \widetilde{h}_\beta(v,L) &\Longleftrightarrow& \mathfrak{o}_E \supset \widetilde{h}_\beta(\varpi_E^{e_0-1}v,L)\\
                                                        &\Longleftrightarrow& v \in \varpi_E^{1-e_0}L^\natural,
\end{eqnarray*}
where $e_0 = e(E \mid E_0)$. Hence the proof is completed.
\end{proof}

Let $\mathcal{L}$ be an $\mathfrak{o}_F$-lattice chain in $V$ such that $E^\times \subset \mathfrak{K}(\mathfrak{A})$, with $\mathfrak{A} = \mathfrak{A}(\mathcal{L})$. Then it follows from \cite[(1.2.1)]{BK1} that $\mathcal{L}$ is also an $\mathfrak{o}_E$-lattice chain in $V$, which is denoted by $\mathcal{L}_{\mathfrak{o}_E}$. Thus, as in 1.2, $\mathcal{L}$ has a unique self-dual slice of the form:
\begin{equation}
L^\natural_{r-1} \supsetneq \cdots \supsetneq L^\natural_0 \supset L_0 \supsetneq \cdots \supsetneq L_{r-1} \supset \varpi_EL^\natural_{r-1}
\end{equation}
for some integer $r \ge 1$, with respect to the form $\widetilde{h}_\beta$.

\begin{prop}                 % Prop. 1.4.2
Let $\mathcal{L}$ be a self-dual $\mathfrak{o}_E$-lattice chain in $V$ with respect to $\widetilde{h}_\beta$. Then it is also a self-dual $\mathfrak{o}_F$-lattice chain in $V$ with respect to $h$. Moreover
\begin{enumerate}
  \item Suppose that $E/E_0$ is unramified. If the self-dual slice of $\mathcal{L}$ of the form (1.4.1) satisfies $L^\natural_0 = L_0$, then $L^\#_0 = L_0$ as an $\mathfrak{o}_F$-lattice.\\
  \item Suppose that $E/E_0$ is ramified. If the self-dual slice of $\mathcal{L}$ satisfies $\varpi_EL^\natural_{r-1} = L_{r-1}$, then it contains an $\mathfrak{o}_E$-lattice $M$ in $V$ such that $M^\# = M$ as an $\mathfrak{o}_F$-lattice.
\end{enumerate}
\end{prop}
\begin{proof}
The first assertion and (1) follows immediately from Proposition 1.4.1. We show (2). Write $e = e(\mathcal{L}_{\mathfrak{o}_E})$ for the $\mathfrak{o}_E$-period of $\mathcal{L}$. From Lemma 1.3.2, it follows that $M = \varpi_E^{-1}L_{r-1}$ is the desired lattice. For, we have
\begin{eqnarray*}
(\varpi_E^{-1}L_{r-1})^\# &=& (L_{-e+r-1})^\# = \varpi_E^{-1}L_{-e+r-1}^\natural\\
                          &=& (\varpi_EL_{-e+r-1})^\natural = L_{r-1}^\natural = \varpi_E^{-1}L_{r-1}.
\end{eqnarray*}
The proof is completed.
\end{proof}

\section{Skew simple strata}              % Sec. 2

\subsection{}                      % Sec. 2.1

We recall the definition of a skew simple stratum in \cite{BK1} and \cite{St2}, and define a good skew simple stratum in $A$.

A stratum in $A$ is a 4-tuple $[\mathfrak{A},n,r,b]$, which consists of a hereditary $\mathfrak{o}_F$-order $\mathfrak{A}$ in $A$, integers $n > r$, and an element $b \in A$ such that $\nu_{\mathfrak{A}}(b) \ge -n$.

\begin{defn} [(1.7) of \cite{St2}]         % Def. 2.1.1
A stratum $[\mathfrak{A},n,r,b]$ in $A$ is called {\it skew}, if the lattice chain $\mathcal{L}$, with $\mathfrak{A} = \mathfrak{A}(\mathcal{L})$, is self-dual and $b \in \mathfrak{G} \simeq \mathrm{Lie}(G)$.
\end{defn}

\begin{defn}[(1.5.5) of \cite{BK1}]          % Def. 2.1.2
A stratum $[\mathfrak{A},n,r,\beta]$ in $A$ is {\it pure}, if it satisfies
\begin{enumerate}
  \item the algebra $E = F[\beta]$ is a field;\
  \item $E^\times \subset \mathfrak{K}(\mathfrak{A})$;\
  \item $\nu_{\mathfrak{A}}(\beta) = -n$.
\end{enumerate}
\end{defn}

For a pure stratum $[\mathfrak{A},n,r,\beta]$ in $A$, the integer $k_0(\beta,\mathfrak{A})$ of \cite[(1.4.5)]{BK1} is defined.

\begin{defn}[(1.5.5) of \cite{BK1}]         % Def. 2.1.3
A pure stratum $[\mathfrak{A},n,r,\beta]$ in $A$ is {\it simple}, if it satisfies $r < -k_0(\beta,\mathfrak{A})$.
\end{defn}

Let $[\mathfrak{A},n,r,\beta]$ be a pure stratum in $A$. Then the rings $\mathfrak{H}(\beta,\mathfrak{A}),\ \mathfrak{J}(\beta,\mathfrak{A})$ of \cite[(3.1)]{BK1} are defined. We define
\begin{center}
$H(\beta,\mathfrak{A}) = \mathfrak{H}(\beta,\mathfrak{A})^\times$\\
\vspace{1mm}
$J(\beta,\mathfrak{A}) = \mathfrak{J}(\beta,\mathfrak{A})^\times$
\end{center}
subgroups of $G$, and for an integer $m \ge 1$, 
\begin{center}
$H^m(\beta,\mathfrak{A}) = \mathfrak{H}(\beta,\mathfrak{A}) \cap (1 + \mathfrak{P}^m),$\\
\vspace{1mm}
$J^m(\beta,\mathfrak{A}) = \mathfrak{J}(\beta,\mathfrak{A}) \cap (1 + \mathfrak{P}^m)$
\end{center}
normal subgroups of $H(\beta,\mathfrak{A})$ and $J(\beta,\mathfrak{A})$, respectively.
A simple character set $\mathcal{C}(\mathfrak{A},m,\beta)$, for an integer $m \ge 0$, of \cite[(3.2)]{BK1} is defined. An element of $\mathcal{C}(\mathfrak{A},m,\beta)$ is a certain abelian character of the group $H^{m+1}(\beta,\mathfrak{A})$.

Let $[\mathfrak{A},n,0,\beta]$ be a skew simple stratum in $A$, with $r = -k_0(\beta,\mathfrak{A})$. Then $\mathfrak{H}(\beta,\mathfrak{A})$ and $\mathfrak{J}(\beta,\mathfrak{A})$ are fixed by $\Gamma$.
 For $0 \le m \le r-1$, the subset $\mathcal{C}^\Gamma(\mathfrak{A},m,\beta)$ of $\mathcal{C}(\mathfrak{A},m,\beta)$ is defined in \cite[3.2]{St1} by
\begin{equation*}
\mathcal{C}^\Gamma(\mathfrak{A},m,\beta) = \{\theta \in \mathcal{C}(\mathfrak{A},m,\beta) \mid \theta^\gamma = \theta\},
\end{equation*}
where $\theta^\gamma(x) = \theta(\gamma(x))$, for $x \in H^{m+1}(\beta,\mathfrak{A})$.

We define two families of compact open subgroups of $G$ as follows:
\begin{center}
$H^m_-(\beta,\mathfrak{A}) = H^m(\beta,\mathfrak{A})^\Gamma = H^m(\beta,\mathfrak{A}) \cap G,$\\
\vspace{1mm}
$J^m_-(\beta,\mathfrak{A}) = J^m(\beta,\mathfrak{A})^\Gamma = J^m(\beta,\mathfrak{A}) \cap G,$
\end{center}
for integers $m \ge 0$. From \cite[(2.1)]{St1}, there is a correspondence $\Vec{\rm{g}}$, which is called {\it Glauberman's correspondence}, between the set of equivalence classes of irreducible representations of $H^{m+1}(\beta,\mathfrak{A})$ fixed by $\Gamma$ and the set of equivalence classes of irreducible representations of $H^{m+1}_-(\beta,\mathfrak{A})$. 
In particular, for $\theta \in \mathcal{C}^\Gamma(\mathfrak{A},m,\beta)$, we have $\Vec{\rm{g}}(\theta) = \theta \vert {H^{m+1}_-(\beta,\mathfrak{A})}$. We put 
\begin{equation*}
\mathcal{C}_-(\mathfrak{A},m,\beta) = \{ \Vec{\rm{g}}(\theta) \mid \theta \in \mathcal{C}^\Gamma(\mathfrak{A},m,\beta)\}.
\end{equation*}

An element of $\mathcal{C}_-(\mathfrak{A},m,\beta)$ is called a {\it skew simple character}.

\subsection{}                     % Sec. 2.2

Suppose that $[\mathfrak{A},n,0,\beta]$ is a skew simple stratum in $A$, with $\mathfrak{A} = \mathfrak{A}(\mathcal{L})$. Let $E = F[\beta]$ and $B = B_\beta$ the $A$-centralizer of $\beta$. 
Let $E_0$ be the fixed field of $E$ under the involution $-$ on $A$.
From Proposition 1.4.1, $\mathcal{L}$ is a self-dual $\mathfrak{o}_E$-lattice chain in $V$ with respect to the form $\widetilde{h}_\beta$. Thus $\mathcal{L}_{\mathfrak{o}_E}$ has a self-dual slice of the form (1.4.1).

\begin{defn}          % Def. 2.2.1
A skew simple stratum $[\mathfrak{A},n,0,\beta]$ in $A$, with $\mathfrak{A} = \mathfrak{A}(\mathcal{L})$, is called {\it good}, if it satisfies
\begin{enumerate}
  \item $E/E_0$ is unramified;
  \item $R = \dim_E(V)$ is even;
  \item The self-dual slice of $\mathcal{L}_{\mathfrak{o}_E}$ of the form (1.4.1) contains the $L_0$ satisfying $L^\natural_0 = L_0$.
\end{enumerate}
\end{defn}

\begin{prop}         % Prop. 2.2.2
If the conditions (1), (2) and (3) in Definition 2.2.1 are satisfied, the anisotropic part of $(V,\widetilde{h}_\beta)$ is zero.
\end{prop}
\begin{proof}
A proof is found in \cite[2.3]{Bl}.
\end{proof}

If $[\mathfrak{A},n,0,\beta]$ is a good skew simple stratum in $A$, from \cite[(5.5.2), (7.1.2)(ii)]{BK1}, we have an $E$-decomposition of $V$ subordinated to $\mathcal{L}_{\mathfrak{o}_E}$, with $e = e(\mathcal{L}_{\mathfrak{o}_E})$:
\begin{equation}
V = \displaystyle \bigoplus_{i=1}^e V^i
\end{equation}
such that
\begin{enumerate}
  \item $L_k = \coprod_{i=1}^e L_k^i$, where $L_k^i = L_k \cap V^i$, for $1 \le i \le e,\ k \in \mathbb{Z}$;\\
  \item $L_{i+me}^i = L_{i+me+1}^i = \cdots = L_{i+(m+1)e-1}^i \ne L_{i+(m+1)e}^i$, for $1 \le i \le e,\ m \in \mathbb{Z}$.
\end{enumerate}

% Since $\mathfrak{A}$ is assumed to be principal, there is a positive integer $f$ satisfying $R = ef$.

\begin{lem}                  % Lem. 2.2.3
Let $[\mathfrak{A},n,0,\beta]$ be a good skew simple stratum in $A$, with $\mathfrak{A} = \mathfrak{A}(\mathcal{L})$, $E = F[\beta]$ and $ e = e(\mathcal{L}_{\mathfrak{o}_E})$. For the self-dual slice of $\mathcal{L}_{\mathfrak{o}_E}$ of the form (1.4.1), there is a Witt basis for $L_0$
\begin{equation}
\mathcal{V} = \{v_1, v_2, \cdots, v_R\}
\end{equation}
such that $L_0 = \mathfrak{o}_Ev_1 \oplus \mathfrak{o}_Ev_2 \oplus \cdots \oplus \mathfrak{o}_Ev_R$, and that each pair $\{v_j, v_{R-j+1}\}$ generates a hyperbolic $E$-subspace of $V$ relative to $\widetilde{h}_\beta$. Write $L_0 = \mathfrak{o}_E<\mathcal{V}>$. 
For the $E$-decomposition (2.2.1) of $V$, each $V^i$ is spanned by 
\begin{center}
$\mathcal{V}^i = \mathcal{V} \cap V^i = \{v_{j_{i-1}+1}, v_{j_{i-1}+2}, \cdots, v_{j_i}\}$ 
\end{center}
over $E$, and $L_k = \coprod_i L_k^i,\ 0 \le k \le [e/2]$, satisfies
\begin{equation*}
L_k^i = 
\begin{cases}
\mathfrak{o}_E<\mathcal{V}^i>  & \text{for $i \le e-k$},\\
\mathfrak{p}_E<\mathcal{V}^i>  & \text{for $i \ge e-k+1$},
\end{cases}
\end{equation*}
where $j_0, j_1, \cdots, j_e$ are integers with $0 = j_0 < j_1 < \cdots < j_e = R$ and for a real number $r$, $[r]$ denotes the largest integer $\le r$. 
\end{lem}
\begin{proof}
This follows directly from Proposition 1.3.1 and \cite[Proposition 1.7]{Mo2}.
\end{proof}

\begin{prop}                  % Prop. 2.2.4
Suppose that $[\mathfrak{A},n,0,\beta]$ is a good skew simple stratum in $A$, with $\mathfrak{A} = \mathfrak{A}(\mathcal{L})$. Let $E = F[\beta]$ and $B = B_\beta$ the $A$-centralizer of $\beta$, and $e = e(\mathcal{L}_{\mathfrak{o}_E})$. Put $t = [(e+1)/2]$. Then the $E$-vector space $V$ is decomposed into an orthogonal decomposition as follows:
\begin{equation*}
V = \bot_{i=1}^t V_i,\ \widetilde{h}_\beta = \bot_{i=1}^t \widetilde{h}_i
\end{equation*}
such that for $1 \le i \le [e/2]$, $(V_i,\widetilde{h}_i)$ is a hyperbolic space, where $V^i$ and $V^{e-i+1}$ are totally isotropic subspaces of $V_i$.
\end{prop}
\begin{proof}
From (2.2.1), for $1 \le i \le [e/2]$, put 
\begin{equation*}
V_i = V^i \oplus V^{e-i+1},\ \widetilde{h}_i = \widetilde{h}_\beta \vert {V_i},
\end{equation*}
and if $t = (e+1)/2$ is an integer, put 
\begin{equation*}
V_t = V^t,\ \widetilde{h}_t = \widetilde{h}_\beta \vert {V_t}.
\end{equation*}
Then the assertion follows directly from \cite[Propositions 1.7, 1.12]{Mo2}.
\end{proof}

Let $\mathfrak{A}$, $E = F[\beta]$ be as above, and $B = B_\beta$ be the $A$-centralizer of $\beta$. Put $\mathfrak{B} = B \cap \mathfrak{A}$. We define a compact open subgroup of $G$ by
\begin{equation*}
\Vec{U}(\mathfrak{B}) = \mathfrak{A} \cap B^\times \cap G,
\end{equation*}
and a family of normal subgroups of $\Vec{U}(\mathfrak{B})$ by
\begin{equation*}
\Vec{U}^m(\mathfrak{B}) = (1 + \mathfrak{P}^m) \cap B^\times \cap G = (1 + \mathfrak{Q}^m) \cap G,
\end{equation*}
for integers $m \ge 1$, where $\mathfrak{Q} = \mathfrak{P} \cap B$.

\begin{prop}                      % Prop. 2.2.5
Suppose that $[\mathfrak{A},n,0,\beta]$ is a good skew simple stratum in $A$, with $\mathfrak{A} = \mathfrak{A}(\mathcal{L})$. Let $E = F[\beta]$ and $B = B_\beta$ the $A$-centralizer of $\beta$, and $e = e(\mathcal{L}_{\mathfrak{o}_E})$.  
Put $t = [(e+1)/2]$. Suppose moreover that the lattice chain $\mathcal{L}_{\mathfrak{o}_E}$ has the self-dual slice of the form (1.4.1). Then there is a canonical isomorphism:
\begin{equation*}
\Vec{U}(\mathfrak{B})/\Vec{U}^1(\mathfrak{B}) \simeq
\begin{cases}
\displaystyle \prod_{i=1}^{e/2} \mathrm{Aut}_{k_E}(\overline{V}^i)   & \text{if $e$ is even},\\
\displaystyle \prod_{i=1}^{(e-1)/2} \mathrm{Aut}_{k_E}(\overline{V}^i) \times \Vec{U}(\overline{V}_t,\overline{h}_t)   & \text{if $e$ is odd},
\end{cases}
\end{equation*}
where $\overline{V}^i = L_{i-1}/L_i,$ for $1 \le i \le [e/2]$, and if $t = (e+1)/2$ is an integer, $\overline{V}_t = L_{t-1}/\varpi_EL_{t-1}^\natural$ and $\overline{h}_t$ is a non-degenerate form, induced naturally from $\widetilde{h}_\beta$. Moreover, $(\overline{V}_t, \overline{h}_t)$ is a $k_E/k_{E_0}$-anti-hermitian space whose anisotropic part is zero.
\end{prop}
\begin{proof}
This follows at once from Proposition 2.2.4 and \cite[1.10 and Proposition 1.12]{Mo2}. In particular, the last assertion follows from Proposition 2.2.2 and \cite[1.10]{Mo2}.
\end{proof}

\section{Beta extensions}                 % Sec. 3

\subsection{}              % Sec. 3.1
Following the methods of \cite{BK1} and \cite{St3}, we prove the existence of a beta extension for our classical group $G$. Hereafter, we assume that the residual characteristic $p$ of $F$ is neither 2 nor 3.

If $\rho$ is a representation of a compact open subgroup $K$ of $G$, and $g \in G$, we write
\begin{equation*}
I_g(\rho) = \mathrm{Hom}_{K^g \cap K}(\rho,\rho^g),
\end{equation*}
where $K^g = g^{-1}Kg$ and $\rho^g(x) = \rho(gxg^{-1})$, for $x \in K^g \cap K$.

\begin{prop}[(5.1.1) of \cite{BK1}]          % Prop. 3.1.1
Let $[\mathfrak{A},n,0,\beta]$ be a skew simple stratum in $A$, and $\theta_- \in \mathcal{C}_-(\mathfrak{A},0,\beta)$. 
Then there is a unique irreducible representation $\eta_- = \eta(\theta_-)$ of $J_-^1(\beta,\mathfrak{A})$ such that $\eta_- \vert H^1_-(\beta,\mathfrak{A})$ contains $\theta_-$. We have
\begin{equation*}
\dim(\eta_-) = (J_-^1(\beta,\mathfrak{A}):H_-^1(\beta,\mathfrak{A}))^{\frac{1}{2}},
\end{equation*}
and for $g \in G$,
\begin{equation*}
\dim(I_g(\eta_-)) =
\begin{cases}
1   & \text{if $g \in J_-^1(B^\times \cap G)J_-^1$,}\\
0   & \text{otherwise}.
\end{cases}
\end{equation*}
\end{prop}
\begin{proof}
This is a special case of \cite[(3.29) and (3.31)]{St3}.
\end{proof}

\begin{prop}[(5.1.2) of \cite{BK1}]             % Prop. 3.1.2
For $i = 1, 2$, suppose that $[\mathfrak{A}_i,n_i,0,\beta]$ is a skew simple stratum in $A$, and let $\theta_-^i \in \mathcal{C}_-(\mathfrak{A}_i,0,\beta)$. 
Let $\eta_-^i$ be the unique irreducible representation of $J_-^1(\beta,\mathfrak{A}_i)$ which contains $\theta_-^i$. Then we have
\begin{equation*}
\dim(\eta_-^1)(\Vec{U}^1(\mathfrak{B}_1):\Vec{U}^1(\mathfrak{B}_2)) = \dim(\eta_-^2)(J_-^1(\beta,\mathfrak{A}_1):J_-^1(\beta,\mathfrak{A}_2)),
\end{equation*}
where $\mathfrak{B}_i$ denotes the $\mathfrak{A}$-centralizer of $\beta$, for $i = 1, 2$.
\end{prop}
\begin{proof}
Using the exact sequence of \cite[(3.17)]{St3} and the Cayley map $C(x) = (1+\frac{1}{2}x)(1-\frac{1}{2}x)^{-1}$, we can prove the assertion in the same way as the proof of \cite[(5.1.2)]{BK1} (cf. \cite[4.2]{Bl}).
\end{proof}

Suppose that $[\mathfrak{A},n,0,\beta]$ is a good skew simple stratum in $A$, with $\mathfrak{A} = \mathfrak{A}(\mathcal{L})$. Let $E = F[\beta],$ and $B = B_\beta$ be the $A$-centralizer of $\beta$. Then $\mathcal{L} = \mathcal{L}_{\mathfrak{o}_E}$ is a self-dual $\mathfrak{o}_E$-lattice chain in $V$, with $e = e(\mathcal{L}_{\mathfrak{o}_E})$. 
From Definition 2.2.1, its self-dual slice of the form (1.4.1) contains the $\mathfrak{o}_E$-lattice $L_0$ in $V$ such that $L_0^\natural = L_0$. Thus we can put
\begin{equation}
\mathcal{L}_M = \{ \varpi_E^iL_0 \mid i \in \mathbb{Z}\}.
\end{equation}
This is a self-dual $\mathfrak{o}_E$-lattice chain in $V$ satisfying (1) $\mathcal{L}_M \subset \mathcal{L}$, (2) the $\mathfrak{o}_E$-period of $\mathcal{L}_M$ is equal to one. We can choose a (maximal) self-dual $\mathfrak{o}_E$-lattice chain $\mathcal{L}_m$ in $V$ satisfying (1) $\mathcal{L} \subset \mathcal{L}_m$, (2) the $\mathfrak{o}_E$-period of $\mathcal{L}_m$ is equal to $R = \dim_E(V)$. From $\mathcal{L}_M$ and $\mathcal{L}_m$, we obtain $\mathfrak{o}_E$-orders $\mathfrak{B}_M$ and $\mathfrak{B}_m$ in $B = B_\beta$ as follows:
\begin{equation*}
\mathfrak{B}_M = \mathrm{End}^0_{\mathfrak{o}_E}(\mathcal{L}_M) = \{x \in B \mid xL \subset L,\ \text{for all}\ L \in \mathcal{L}_M\}
\end{equation*}
and similarly
\begin{equation*}
\mathfrak{B}_m = \mathrm{End}^0_{\mathfrak{o}_E}(\mathcal{L}_m).
\end{equation*}
Then $\mathfrak{B}_M$ (resp. $\mathfrak{B}_m$) is a maximal (resp. minimal) herediatry $\mathfrak{o}_E$-order of $B$. Moreover $\mathfrak{B} = B \cap \mathfrak{A}$ satisfies $\mathfrak{B}_m \subset \mathfrak{B} \subset \mathfrak{B}_M$. From Proposition 1.4.1, $\mathcal{L}_M$ and $\mathcal{L}_m$ are also self-dual $\mathfrak{o}_F$-lattice chains in $V$. Write 
\begin{equation*}
\mathfrak{A}_M = \mathrm{End}^0_{\mathfrak{o}_F}(\mathcal{L}_M),\ \mathfrak{A}_m = \mathrm{End}^0_{\mathfrak{o}_F}(\mathcal{L}_m).
\end{equation*}
Then we have $\mathfrak{B}_M = \mathfrak{A}_M \cap B,\ \mathfrak{B}_m = \mathfrak{A}_m \cap B$. 

We denote by $\nu_E(\beta)$ the normalized valuation of $\beta$ in $E$. Then, since we have $\nu_{\mathfrak{A}_M}(\beta) = -\nu_E(\beta)$ and $\nu_{\mathfrak{A}_m}(\beta) = -\nu_E(\beta)R$, strata $[\mathfrak{A}_M,-\nu_E(\beta),0,\beta]$ and $[\mathfrak{A}_m,-\nu_E(\beta)R,0,\beta]$ in $A$ are both (good) skew simple. From \cite[(3.26)]{St3}, there is a transfer
\begin{equation*}
\tau_{\mathfrak{A}_m,\mathfrak{A}_M,\beta,0}:\mathcal{C}_-(\mathfrak{A}_m,0,\beta) \to \mathcal{C}_-(\mathfrak{A}_M,0,\beta),
\end{equation*}
(cf. \cite[(3.6.2)]{BK1}). Similarly, there is a transfer $\tau_{\mathfrak{A}_m,\mathfrak{A},\beta,0}$.

Let $\theta_{M,-} \in \mathcal{C}_-(\mathfrak{A}_M,0,\beta),\ \theta_{m,-} \in \mathcal{C}_-(\mathfrak{A}_m,0,\beta)$, and  $\theta_- \in \mathcal{C}_-(\mathfrak{A},0,\beta)$. Assume that these characters are related as follows:
\begin{equation*}
\theta_{M,-} = \tau_{\mathfrak{A}_m,\mathfrak{A}_M,\beta,0}(\theta_{m,-}),\ \theta_- = \tau_{\mathfrak{A}_m,\mathfrak{A},\beta,0}(\theta_{m,-}),
\end{equation*}
as in \cite[(5.1.13)]{BK1}.

For an integer $t \ge 1$, write simply $J_-^t = J_-^t(\beta,\mathfrak{A}),\ J_{m,-}^t = J_-^t(\beta,\mathfrak{A}_m),\ J_{M,-}^t = J_-^t(\beta,\mathfrak{A}_M),\ J_- = J_-(\beta,\mathfrak{A})$, and so on, with similar conventions for the group $H_-$. 
Let $\eta_-$ (resp. $\eta_{m,-},$ resp. $\eta_{M\-}$) be the unique irreducible representation in Proposition 3.1.1 which contains $\theta_-$ (resp. $\theta_{m,-}$, resp. $\theta_{M,-}$). Analogous results to Propositions (5.1.14) $-$ (5.1.19) for $GL(N,F)$ in \cite{BK1} can be proved for $G$ in a quite similar way.

\begin{prop}[(5.1.14)-(5.1.18) of \cite{BK1}]       % Prop. 3.1.3
Let notation and assumptions be as above. 
\begin{enumerate}
  \item There is a unique irreducible representation $\widetilde{\eta}_{M,-}$ of $\Vec{U}^1(\mathfrak{B}_m)J_{M.-}^1$ such that
\begin{enumerate}
  \item $\widetilde{\eta}_{M,-} \vert {J_{M,-}^1} = \eta_{M,-}$, and
  \item the representations $\widetilde{\eta}_{M,-}$ and $\eta_{m,-}$ induce equivalent irreducible representations of $\Vec{U}^1(\mathfrak{A}_m)$.
 \end{enumerate}
 
  \item There is a unique irreducible representation $\widetilde{\eta}_-$ of $\Vec{U}^1(\mathfrak{B}_m)J_-^1$ such that
\begin{enumerate}
  \item $\widetilde{\eta}_- \vert {J_-^1} = \eta_-$, and
  \item the representations $\widetilde{\eta}_-$ and $\eta_{m,-}$ induce equivalent irreducible representations of $\Vec{U}^1(\mathfrak{A}_m)$.
 \end{enumerate}

  \item There is a unique irreducible representation $\hat{\eta}_{M,-}$ of $\Vec{U}^1(\mathfrak{B})J_{M,-}^1$ such that
\begin{enumerate}
  \item $\hat{\eta}_{M,-} \vert {J_{M,-}^1} = \eta_{M,-}$, and
  \item the representations $\hat{\eta}_{M,-}$ and $\eta_-$ induce equivalent irreducible representations of $\Vec{U}^1(\mathfrak{A})$.
 \end{enumerate}
 \end{enumerate}
\end{prop}

If $\rho$ is a representation of a compact open subgroup $K$ of $G$, put
\begin{equation*}
I_G(\rho) = \{ g \in G \mid I_g(\rho) \ne (0)\}.
\end{equation*}
We say that an element $g$ of $G$ {\it intertwines} $\rho$, if $g \in I_G(\rho)$.

\begin{prop}[(5.1.19) of \cite{BK1}]       % Prop. 3.1.4
Let notation and assumptions be as in Proposition 3.1.3. Then we have
\begin{center}
$I_G(\widetilde{\eta}_{M,-}) = J_{M,-}^1(B^\times \cap G)J_{M,-}^1$,\\
\vspace{1mm}
$I_G(\eta_-) = J_-^1(B^\times \cap G)J_-^1$.
\end{center}
\end{prop}
\begin{proof}
By using \cite[Theorem 2.2]{St2}, we can prove the assertion in the same way as the proof of \cite[(5.1.19)]{BK1}.
\end{proof}

\subsection{}                   % Sec. 3.2

Let $[\mathfrak{A},n,0,\beta]$ be a skew simple stratum in $A$, and $\theta_- \in \mathcal{C}_-(\mathfrak{A},0,\beta)$. Let $E = F[\beta]$ and $B = B_\beta$ be the $A$-centralizer of $\beta$. Let $\eta_-$ be the unique irreducible representation of $J_-^1(\beta,\mathfrak{A})$ which contains $\theta_-$.

\begin{defn}[(5.2.1) of \cite{BK1}]      % Def. 3.2.1
A representation $\kappa_-$ of $J_-(\beta,\mathfrak{A})$ is called a $\beta$-{\it extension} of $\eta_-$, if it satisfies 
\begin{enumerate}
  \item $\kappa_- \vert J_-^1(\beta,\mathfrak{A}) = \eta_-$;
  \item $B^\times \cap G \subset I_G(\kappa_-)$.
\end{enumerate}
\end{defn}

We show that if a skew simple stratum $[\mathfrak{A},n,0,\beta]$ in $A$ is good, there is a $\beta$-extension of $\eta_-$.

\begin{lem}            % Lem. 3.2.2
Let $U, V$ be subgroups of $\widetilde{G}$ fixed by $\Gamma$. Suppose that $U$ normalizes $V$, and that $U \cap V$ is a pro $p$-group. Then we have
\begin{equation*}
(UV)^\Gamma = U^\Gamma V^\Gamma.
\end{equation*}
\end{lem}
\begin{proof}
The groups $UV, U \cap V$ are both $\Gamma$-sets. Then we obtain a short sequence
\[1 \xrightarrow{} U \cap V \xrightarrow{\delta} U \times V \xrightarrow{\pi} UV \xrightarrow{} 1, \]
where $\delta(x) = (x,x)$, for $x \in U \cap V$, and $\pi(x,y) = xy^{-1}$, for $x \in U, y \in V$. This is an exact sequence of $\Gamma$-sets. For, we have
\begin{center}
$\delta(\gamma(x)) = (\gamma(x),\gamma(x)) = \gamma(x,x),$\\
\vspace{1mm}
$\pi(\gamma(x),\gamma(y)) = \gamma(x)\gamma(y)^{-1} = \gamma(xy^{-1}) = \gamma(\pi(x,y)),$
\end{center}
for $x \in U, y \in V$. From \cite[Proposition 3.6]{Se}, we thus obtain an exact sequence
\begin{equation*}
 1 \longrightarrow (U \cap V)^\Gamma \longrightarrow (U \times V)^\Gamma \longrightarrow (UV)^\Gamma \longrightarrow H^1(\Gamma,U \cap V) \longrightarrow H^1(\Gamma,U \times V) 
\end{equation*}
Since $U \cap V$ is pro $p$-group and $p$ is not 2, we hence have $H^1(\Gamma,U \cap V) = 1$, whence $(UV)^\Gamma = U^\Gamma V^\Gamma$. The proof is completed.
\end{proof}

\begin{prop}[cf. (5.2.4) of \cite{BK1}]          % Prop. 3.2.3
Let $[\mathfrak{A},n,0,\beta]$ be a good skew simple stratum in $A$, and $\widetilde{\eta}_{M,-}$ the representation of $\Vec{U}^1(\mathfrak{B}_M)J_{M,-}^1$, as in Proposition 3.1.3. 
Then there is a representation $\kappa_{M,-}$ of $J_{M,-}$ such that $\kappa_{M,-} \vert \Vec{U}^1(\mathfrak{B}_m)J_{M,-}^1 = \widetilde{\eta}_{M,-}$.
\end{prop}
\begin{proof}
Following the methods of the proof of \cite[(5.2.4)]{BK1}, we prove the assertion. We sketch the proof. 

Put $r = -k_0(\beta,\mathfrak{A})$. From Lemma 3.2.2 and \cite[(3.12)]{St3}, we get
\begin{equation*}
J_{M,-}^1 = \Vec{U}^1(\mathfrak{B}_M)J_-^{[(r+1)/2]}(\beta,\mathfrak{A}_M),\ J_{M,-} = \Vec{U}(\mathfrak{B}_M)J_{M,-}^1.
\end{equation*}
From the case where $e = e(\mathcal{L}_{\mathfrak{o}_E}) = 1$ in Proposition 2.2.5, we have
\begin{equation*}
J_{M,-}/J_{M,-}^1 \simeq \Vec{U}(\mathfrak{B}_M)/\Vec{U}^1(\mathfrak{B}_M) \simeq U(\overline{V},\overline{h}),
\end{equation*}
where $\overline{V} = L_0/\varpi L_0^\natural$ for $L_0 \in \mathcal{L}_{\mathfrak{o}_E}$ in (1.4.1) and $\overline{h}$ is a non-degenerate $k_E/k_{E_0}$-anti-hermitian form, which is naturally induced from the form $\widetilde{h}_\beta$. It follows from Proposition 2.2.5 that $\mathcal{G} = U(\overline{V},\overline{h})$ is a unitary group over $k_{E_0}$ of type $\mathrm{A}^2_{R-1}$.
The canonical image of $\Vec{U}^1(\mathfrak{B}_m)/\Vec{U}^1(\mathfrak{B}_M)$ into $\mathcal{G}$ is the unipotent radical, $\mathcal{N}$, of a Borel subgroup of $\mathcal{G}$. Thus $\Vec{U}^1(\mathfrak{B}_m)J_{M,-}^1$ is a Sylow pro $p$-subgroup of $J_{M,-}$.
Since, from \cite[(3.31)]{St3}, $J_{M,-}$ normalizes $\eta_{M,-}$, we obtain a projective representation of $J_{M,-}$ which is an extension of $\eta_{M,-}$.
We can adjust this projective representation to be a linear representation $\lambda$ of $J_{M,-}$. Then we have
\begin{equation*}
\lambda \vert \Vec{U}^1(\mathfrak{B}_m)J_{M,-}^1 = \widetilde{\eta}_{M,-} \otimes \phi,
\end{equation*}
where $\phi$ is a character of $\Vec{U}^1(\mathfrak{B}_m)$ which is trivial on $\Vec{U}^1(\mathfrak{B}_M)$.
This $\phi$ is a character of $\mathcal{N}$ which is intertwined by all the elements of $\mathcal{G}$. Let $\Phi$ be a root system of $\mathcal{G}$ and $\Delta$ the set of simple roots in $\Phi$, associated with $\mathcal{N}$. We denote by $U_a$ the root subgroup of $\mathcal{G}$ associated with $a \in \Phi$, and by [$\mathcal{N},\mathcal{N}$] the commutator group of $\mathcal{N}$. Let $ht$ be the height function on $\Phi$ with respect to the basis $\Delta$. 
Then, under the assumption $p \ne 2, 3$, by using the commutator relations in the twisted group $\mathcal{G}$ of $GL(R,k_E)$, we can easily see that $[\mathcal{N},\mathcal{N}] = \prod_{a} U_a$, where $a$ runs through roots in $\Phi$ with $ht(a)\ge 2$, (cf. \cite[Section 11]{St}, \cite[Section 13]{Ca2}) and see that there is a canonical isomorphism
\begin{equation*}
\mathcal{N}/[\mathcal{N},\mathcal{N}] \simeq \displaystyle \prod_{a \in \Delta} U_a.
\end{equation*}
As in \cite[8.1]{Ca2}, this fact holds for any finite group of Lie type.
Thus $\phi$ is trivial on $\mathcal{N}$ and can be extended to a character $\phi'$ of $\mathcal{G}$, like the proof of \cite[(5.2.4)]{BK1} for $GL(N,F)$.
We regard $\phi'$ as a character of $J_{M,-}$, and put
\begin{equation*}
\kappa_{M,-} = \lambda \otimes \phi^{\prime -1}.
\end{equation*}
It easily seen that the representation $\kappa_{M,-}$ is the desired.
\end{proof}

\begin{prop}[cf. (5.2.5) of \cite{BK1}]        % Prop. 3.2.4
Let $\kappa_{M,-}$ be the representation as in 3.2.3. Then there is a representation $\kappa_-$ of $J_-$ which is uniquely determined by the following properties:
\begin{enumerate}
  \item $\kappa_- \vert J_-^1 = \eta_-$;
  \item $\kappa_-$ and $\kappa_{M,-} \vert \Vec{U}(\mathfrak{B})J_{M,-}^1$ induce equivalent irreducible representations of $\Vec{U}(\mathfrak{A})$.\
  \item $\mathrm{Ind}(\kappa_-:J_-,\Vec{U}(\mathfrak{B})\Vec{U}^1(\mathfrak{A}))$ is equivalent to
  \begin{equation*}
  \mathrm{Ind}(\kappa_{M,-} \vert \Vec{U}(\mathfrak{B})J_{M,-}^1:\Vec{U}(\mathfrak{B})J_{M,-}^1, \Vec{U}(\mathfrak{B})\Vec{U}^1(\mathfrak{A})).
  \end{equation*}
\end{enumerate}
\end{prop}
\begin{proof}
Using Proposition 3.1.2, we can prove the assertion in the same way as the proof of \cite[(5.2.5)]{BK1}. 
\end{proof}

We show that the representation $\kappa_-$ in Proposition 3.2.4 is a $\beta$-extension.

\begin{prop}[cf. (5.2.7) of \cite{BK1}]         % Prop. 3.2.5
Let $\kappa_-$ be the representation of $J_-$ constructed as in Proposition 3.2.4. Then we have
\begin{equation*}
I_G(\kappa_-) = J_-(B^\times \cap G)J_- = J_-^1(B^\times \cap G)J_-^1.
\end{equation*}
\end{prop}
\begin{proof}
The proof of \cite[(5.2.7)]{BK1} for $GL(N,F)$ remains valid for our classical $G$, as well. We also sketch the proof.

By the Witt basis $\mathcal{V}$ of (2.2.2), we express elements of $B^\times \cap G$ in matrix form, that is, $B^\times \cap G$ is embedded in $GL(R,E)$, where $R = \dim_E(V)$. 
Moreover, $\Vec{U}(\mathfrak{B}_M)$ is embedded in $GL(R,\mathfrak{o}_E)$, and it is a special maximal compact subgroup of $B^\times \cap G$. Thus $B^\times \cap G$ has a Cartan decomposition relative to $\Vec{U}(\mathfrak{B})$. 

From \cite[(3.13)]{St3}, $I_G(\kappa_-) \subset I_G(\eta_-) = J_-(B^\times \cap G)J_-$. So it is enough to prove that any element $y$ of $B^\times \cap G$ intertwines $\kappa_-$. 
Moreover, by Proposition 3.2.4 (2), it is enough to treat the case where $\mathcal{L} = \mathcal{L}_M$ and $\kappa_- = \kappa_{M,-}$. Since $\Vec{U}(\mathfrak{B}_M) \subset J_- \cap B^\times \cap G$, we can choose $y$ in a $(\Vec{U}(\mathfrak{B}_M),\Vec{U}(\mathfrak{B}_M))$-double coset, and reduce it to a diagonal element 
\begin{equation*}
\mathrm{Diag}(\varpi_E^{n_1}, \cdots, \varpi_E^{n_r}, \varpi_E ^{-n_r}, \cdots, \varpi_E ^{-n_1})
\end{equation*}
where $r = R/2$ and $n_1, n_2, \cdots, n_r$ are integers with $n_1 \ge n_2 \ge \cdots \ge n_r$. Here we recall that $E/E_0$ is unramified.
As in the proof of \cite[(5.2.7)]{BK1}, we can choose a self-dual $\mathfrak{o}_E$-lattice chain $\mathcal{L}'$ in $V$, with $e(\mathcal{L}'_{\mathfrak{o}_E}) = e'$, for some integer $e' \ge 1$, which satisfies the following properties:
\begin{enumerate}
  \item the self-dual slice of $\mathcal{L}'$ of the form (1.4.1) satisfies $L_0^\natural = L_0$,\
  \item This lattice $L_0$ is the same as that of $\mathcal{L}$,\
  \item for the $E$-decomposition $V =  \bigoplus_{i=1}^{e'} V^i$ subordinated to $\mathcal{L}'$, the element $y$ has a diagonal block form $(y_i)$, and each $y_i $ in  $\mathrm{End}_E(V^i)$ is central, for $1 \le i \le e'$.
\end{enumerate}

From Proposition 1.4.2, $\mathcal{L}'$ is also a self-dual $\mathfrak{o}_F$-lattice chain in $V$. Put $\mathfrak{B}' = \mathrm{End}_{\mathfrak{o}_F}^0(\mathcal{L}') \cap B$. 
From part (2) of the above properties, elements of $\mathfrak{B}'$ are written in the following block form: $(x_{jk}),\ 1 \le j, k \le e'$, such that coefficients of $n_j \times n_k$-matrix $x_{jk}$ are all in $\mathfrak{o}_E$ if $j \le k$, and all in $\mathfrak{p}_E$ otherwise, where $R = n_1 + n_2 + \cdots + n_{e'}$ is the partition of $R$ associated with $\mathcal{L}'$.
Put
\begin{equation*}
\widetilde{\Vec{\mathfrak{M}}}(\mathfrak{B}') = \{ (x_{jk}) \in \mathfrak{B}' \mid x_{jk} = 0, \text{for all}\ j \ne k \}.
\end{equation*}
Then it follows from Proposition 2.2.4 that the involution $^-$ fixes $\widetilde{\Vec{\mathfrak{M}}}(\mathfrak{B}')$. Thus we have
\begin{equation*}
\Vec{\mathfrak{M}}(\mathfrak{B}')^\times = (\widetilde{\Vec{\mathfrak{M}}}(\mathfrak{B}')^\times)^\Gamma = \widetilde{\Vec{\mathfrak{M}}}(\mathfrak{B}') \cap G.
\end{equation*}
From the proof of \cite[(5.2.7)]{BK1}, we have
\begin{enumerate}
  \item $y$ centralizes $\widetilde{\Vec{\mathfrak{M}}}(\mathfrak{B}')$;\
  \item $\mathfrak{B}_M \cap \mathfrak{B}_M^y \subset \mathfrak{p}_F\mathfrak{B}_M + (\mathfrak{B}' \cap (\mathfrak{B}')^y),$
\end{enumerate}
where $L^y = y^{-1}Ly$. We denote by $^t\mathfrak{B}'$ the transpose of $\mathfrak{B}'$. Then we also have
\begin{enumerate}
  \item $y^{-1}$ centralizes $^t\widetilde{\Vec{\mathfrak{M}}}(\mathfrak{B}')$;\
 \item $\mathfrak{B}_M \cap\ ^y\mathfrak{B}_M \subset \mathfrak{p}_F\mathfrak{B}_M +\ ^y(^t \mathfrak{B}' \cap (^t \mathfrak{B}')),$
\end{enumerate}
where  $^yL = yLy^{-1}$.

If $\mathfrak{B}' = \mathfrak{B}_M$, clearly $y = 1$. We note that this fact never occurs for the case of $GL(N,F)$. Thus $y = 1$ trivially intertwines $\kappa_{M,-}$.

From \cite[p.173]{BK1} together with Lemma 3.2.2, we obtain
\begin{eqnarray*}
(\Vec{\mathfrak{M}}(\mathfrak{B}')^\times\Vec{U}^1(\mathfrak{B}')J_M^1) &\cap& (\Vec{\mathfrak{M}}(\mathfrak{B}')^\times \Vec{U}^1(\mathfrak{B}')J_M^1)^y \\
                                                                        &=& \Vec{\mathfrak{M}}(\mathfrak{B}')^\times(\Vec{U}^1(\mathfrak{B}')J_M^1 \cap (\Vec{U}^1(\mathfrak{B}')J_M^1)^y)
\end{eqnarray*}
in $\widetilde{G}$. 
It follows from Lemma 3.2.2 and \cite[(5.2.11)]{BK1} that the element $y$ intertwines $\kappa_{M,-} \vert \Vec{U}(\mathfrak{B}')J_{M,-}^1$ with $\kappa_{M,-} \vert \Vec{U}(\mathfrak{B}')J_{M,-}^1 \otimes \phi$, 
where $\phi$ is an abelian character of $\Vec{\mathfrak{M}}(\mathfrak{B}')^\times/(\Vec{\mathfrak{M}}(\mathfrak{B}')^\times \cap \Vec{U}^1(\mathfrak{B}')J_{M,-}^1).$
For the lattice chain $\mathcal{L}'$ in $V$, we can choose the minimal self-dual $\mathfrak{o}_E$-lattice chain $\mathcal{L}'_M = \mathcal{L}_M$, given in 3.1, and a maximal self-dual  $\mathfrak{o}_E$-lattice chain $\mathcal{L}'_m$ in $V$, such that $\mathcal{L}'_m \subset \mathcal{L}' \subset \mathcal{L}'_M$. 
Then we can see that $\phi$ is factored through the determinant, in a suitable sense (cf.  \cite[ p.173]{BK1}). 
Let $\kappa_-$ be the representation of $J_-(\beta,\mathfrak{A}')$ given by Proposition 3.2.4, where $\mathfrak{A}' = \mathrm{End}_{\mathfrak{o}_F}^0(\mathcal{L}')$. 
We can form the representation $\kappa_- \otimes \phi$, and by using Propositions 3.2.4 and 3.1.1, we can prove that $y$ intertwines $\kappa_-$ with $\kappa_- \otimes \phi$.

\vspace{3mm}
{\bf Claim} There is an extension $\mu_{-}$ of $\eta_{-}$ intertwined by $y$.
\vspace{3mm}

We shall prove the claim in 4.2 below. We now assume that the claim is true. We also apply $H = J_{-}^1,\ N = \Vec{\mathfrak{M}}(\mathfrak{B}')^\times,\ g = y,\ \rho = \eta_{-}$ to \cite[(5.2.11)]{BK1}. Then these satisfy those hypotheses. 
In particular, we apply $\kappa_{-}$ to $\widetilde{\rho}$ there. We now apply $\mu_{-}$ to $\rho'$ in part (a) of \cite[(5.2.11)]{BK1} so that $y$ intertwines $\mu_{-}$ with $\mu_{-} \otimes \phi$. 
Thus the uniqueness of $\phi$ shows that $\phi$ is trivial. Hence we have seen that $y$ intertwines $\kappa_{M,-} \vert \Vec{U}(\mathfrak{B}')J_{M,-}^1$.

From the proof of \cite[(5.2.7)]{BK1} and Lemma 3.2.2, we obtain
\begin{equation*}
J_{M,-} \cap J_{M,-}^y = (\Vec{U}^1(\mathfrak{B}_M) \cap \Vec{U}(\mathfrak{B}_M)^y)(\Vec{U}(\mathfrak{B}')J_{M,-}^1 \cap (\Vec{U}(\mathfrak{B}')J_{M,-}^1)^y).
\end{equation*}
Similarly,
\begin{equation*}
(\Vec{U}^1(\mathfrak{B}_M) \cap \Vec{U}(\mathfrak{B}_M)^y) \subset (\Vec{U}(\mathfrak{B}_M) \cap \Vec{U}^1(\mathfrak{B}_M)^y)(\Vec{U}(\mathfrak{B}^\prime) \cap \Vec{U}(\mathfrak{B}^\prime)^y).
\end{equation*}
Hence we can prove that $y$ intertwines $\kappa_{M,-}$ in the same way as the proof of \cite[(5.2.7)]{BK1}. The proof is completed modulo the claim.
\end{proof}

\begin{thm}         % Thm. 3.2.6
Let $[\mathfrak{A},n,0,\beta]$ be a good skew simple stratum in $A$, and $\theta_{-} \in \mathcal{C}_{-}(\mathfrak{A},0,\beta)$. Let $\eta_{-}$ be the unique irreducible representation of $J_{-}^1(\beta,\mathfrak{A})$ which contains $\theta_{-}$. Then there is a $\beta$-extension of $\eta_{-}$.
\end{thm}
\begin{proof}
The assertion follows directly from Propositions 3.2.4 and 3.2.5 (modulo the claim).
\end{proof}

To prove the claim, the following lemma will be used in next section.

\begin{lem}                 % Lem. 3.2.7
Let $\mathcal{L}'$ be the self-dual $\mathfrak{o}_E$-lattice chain in $V$ associated with $y \in B^\times \cap G$ in the proof of Proposition 3.2.5. Let $\mathfrak{A}' = \mathrm{End}_{\mathfrak{o}_F}^0(\mathcal{L}')$ and $n' = -\nu_{\mathfrak{A}'}(\beta)$. Then $[\mathfrak{A}',n',0,\beta]$ is a good skew simple stratum in $A$.
\end{lem}
\begin{proof}
Straightforward.
\end{proof}

\section{Iwahori decompositions}         % Sec. 4

\subsection{}
We prove the claim in the proof of Proposition 3.2.5.

Suppose that $[\mathfrak{A},n,0,\beta]$ is a skew simple stratum in $A$, with $\mathfrak{A} = \mathfrak{A}(\mathcal{L})$. Let $E = F[\beta]$, and $B = B_\beta$ be the $A$-centralizer of $\beta$. Put $e = e(\mathcal{L}_{\mathfrak{o}_E})$. For the $E$-decomposition $V = \bigoplus_{i=1}^e V^i$ of (2.2.1) subordinated to $\mathcal{L}_{\mathfrak{o}_E}$, put
\begin{equation*}
A^{ij} = \mathrm{Hom}_F(V^j, V^i),\ A^i = A^{ii},\ \text{for}\ 1 \le i, j \le e.
\end{equation*}
We define subgroups of $\widetilde{G}$ as follows:
\begin{enumerate}
  \item $\widetilde{P} = \widetilde{G} \cap (\displaystyle \prod_{1 \le i < j \le e} A^{ij})$;\\
  \item $\widetilde{M} = \widetilde{G} \cap (\displaystyle \prod_{1 \le i \le e} A^i)$;\\
  \item $\mathbb{N}_u = \displaystyle \prod_{1 \le i < j \le e} A^{ij},\ \widetilde{N}_u = 1 + \mathbb{N}_u$;\\
  \item $\mathbb{N}_\ell = \displaystyle \prod_{1 \le j < i \le e} A^{ij},\ \widetilde{N}_\ell = 1 + \mathbb{N}_\ell$.
\end{enumerate}
Each $\mathfrak{o}_E$-lattice $L_k$ in $\mathcal{L}_{\mathfrak{o}_E}$ has a decomposition $L_k = \coprod_{1 \le i \le e} L_k^i$, with $L_k^i = L_k \cap V^i$, for $k \in \mathbb{Z}$.
From \cite[(7.1.12)]{BK1}, there is a canonical isomorphism
\begin{equation*}
H^1(\beta,\mathfrak{A}) \cap \widetilde{M} \simeq \displaystyle \prod_{i=1}^e H^i(\beta,\mathfrak{A}^{(i)}),
\end{equation*}
where $\mathfrak{A}^{(i)} = \mathrm{End}_{\mathfrak{o}_F}^0(\{L_k^i \mid k \in \mathbb{Z}\})$, for $1 \le i \le e$

\begin{prop}[(7.1.19) of \cite{BK1}]            % Prop. 4.1.1
Let $[\mathfrak{A},n,0,\beta]$ be a simple stratum in $A$, with $\mathfrak{A} = \mathfrak{A}(\mathcal{L})$ and $e = e(\mathcal{L}_{F[\beta]})$, and $\theta \in \mathcal(\mathfrak{A},0,\beta)$. 
Then $\theta$ is trivial on $H^1(\beta,\mathfrak{A}) \cap \mathrm{Hom}_F(V^j,V^i)$,\ for $i \ne j$. 
Under the identification $H^1(\beta,\mathfrak{A}) \cap \widetilde{M} = \prod_i H^1(\beta,\mathfrak{A}^{(i)})$, we have
\begin{equation*}
\theta \vert (H^1(\beta,\mathfrak{A}) \cap \widetilde{M}) = \theta^{(1)} \otimes \cdots \otimes \theta^{(e)},
\end{equation*}
where $\theta^{(i)} \in \mathcal(\mathfrak{A}^{(i)},0,\beta)$ and $\theta^{(i)} = \tau_{\mathfrak{A},\mathfrak{A}^{(i)},\beta,0}(\theta)$, for $1 \le i \le e$.
\end{prop}

Suppose that a skew simple stratum $[\mathfrak{A},n,0,\beta]$ in $A$ is good. Let $\mathfrak{A} = \mathfrak{A}(\mathcal{L})$, $E = F[\beta]$, $e = e(\mathcal{L}_{\mathfrak{o}_E})$, and $B = B_\beta$ be the $A$-centralizer of $\beta$. 
Put $t = [(e+1)/2]$. For the orthogonal decomposition $(V, \widetilde{h}_\beta) = \bot_i\ (V_i, \widetilde{h}_i)$ in Proposition 2.2.4, we define
\begin{equation*}
h_i = \ell \circ \widetilde{h}_i,
\end{equation*}
for $1 \le i \le t$, where $\ell: E \to F$ is the $F$-linear form defined in 1.3. Then, for $1 \le i \le [e/2]$, $(V_i, h_i)$ is a hyperbolic $F$-space such that $V^i,\ V^{e-i+1}$ are totally isotropic $F$-subspaces of $V_i$, and if $t = (e+1)/2$ is an integer, then $V_t = V^t$ and $h_t = h \vert V_t$.
Moreover, we have an orthogonal $F$-decomposition of $V$:
\begin{equation}
V = \bot_{i=1}^t V_i,\ h = \bot_{i=1}^t h_i,
\end{equation}
Thus the involution $^-$ on $A$, defined by $h$, induces involutions $A^i \to A^{e-i+1},\ A^{ij} \to A^{e-i+1,e-j+1}$, for $1 \le i, j,\le e$, where if $i \equiv j\ (\mathrm{mod}\ e)$, we set $i = j$. 
We denote by $x \mapsto \bar{x}$ the induced involution $A^i \to A^{e-i+1}$.
Hence the involution $^-$ on $A$ fixes $\prod_i A^i, \mathbb{N}_u$ and $\mathbb{N}_\ell$, respectively, whence the involution $\gamma$ on $\widetilde{G}$ fixes the subgroups $\widetilde{P},\ \widetilde{M},\ \widetilde{N}_u$ and $\widetilde{N}_\ell$.
Let $\widetilde{\mathcal{G}}$ be one of these subgroups. Put
\begin{equation}
\mathcal{G} = \widetilde{\mathcal{G}}^\Gamma = \widetilde{\mathcal{G}} \cap G.
\end{equation}
Then $P = MN_u$ is a parabolic subgroup of $G$, with Levi component $M$ and unipotent radical $N_u$. We also have the opposite parabolic subgroup $P_\ell = MN_\ell$ with respect to $M$. We say that the parabolic subgroup $P = MN_u$ is {\it associated with} a good skew simple stratum $[\mathfrak{A},n,0,\beta]$.

\begin{lem}          % Lem. 4.1.2
Let $[\mathfrak{A},n,0,\beta]$ be a good skew simple stratum in $A$, and $P = MN_u$ a parabolic subgroup of $G$ associated with $[\mathfrak{A},n,0,\beta]$. Let $\mathfrak{A} = \mathfrak{A}(\mathcal{L})$, $E = F[\beta]$, and $e = e(\mathcal{L}_{\mathfrak{o}_E})$. Let $V = \bigoplus_{i=1}^e V^i$ be the $E$-decomposition of (2.2.1) subordinated to $\mathcal{L}_{\mathfrak{o}_E}$. Then there is a canonical isomorphism
\begin{equation*}
M \simeq
\begin{cases}
\displaystyle \prod_{i=1}^{e/2} \mathrm{Aut}_F(V^i)     & \text{if $e$ is even},\\
(\displaystyle \prod_{i=1}^{(e-1)/2} \mathrm{Aut}_F(V^i)) \times U(V_t,h_t)      & \text{if $e$ is odd},
\end{cases}
\end{equation*}
where $t = (e+1)/2$.
\end{lem}
\begin{proof}
The assertion follows easily from the above argument (cf. Proposition 2.2.5).
\end{proof}

We write simply $H_-^m = H_-^m(\beta,\mathfrak{A})$ and $J_-^m = J_-^m(\beta,\mathfrak{A})$, for $m = 0, 1$.
From \cite[(7.1.14), (7.1.16)-(7.1.18)]{BK1}, we obtain Iwahori decompositions of  $H_-^m, J_-^m$, for $m = 0, 1$, as follows:

\begin{prop} [cf. (7.1.14) of \cite{BK1}]       % Prop. 4.1.3
Let $\mathcal{G}_{-}$ denote any of the groups $H_{-}^m,\ J_{-}^m,$ for $m = 0, 1$. Then we have the Iwahori decomposition:
\begin{center}
$\mathcal{G}_{-} = (\mathcal{G}_{-} \cap N_\ell).(\mathcal{G}_{-} \cap M).(\mathcal{G}_{-} \cap N_u),$\\
\vspace{1mm}
$\mathcal{G}_{-} \cap P = (\mathcal{G}_{-} \cap M).(\mathcal{G}_{-} \cap N_u),$
\end{center}
\end{prop}

Put $t = [(e+1)/2]$. According to the decomposition of $M$ in Lemma 4.1.2, for $m =0, 1$, we have
\begin{equation*}
J_-^m(\beta,\mathfrak{A}) \cap M \simeq \displaystyle \prod_{i=1}^t J^m(\beta,\mathfrak{A}^{(i)}),
\end{equation*}
where if $t = (e+1)/2$ is an integer, we understand $J^m(\beta,\mathfrak{A}^{(t)}) = J_-^m(\beta,\mathfrak{A}^{(t)})$. Likewise for $H_-^m(\beta,\mathfrak{A})$, for $m = 0, 1$. Moreover, we have
\begin{center}
$(J_- \cap M)H_-^1 = (H_-^1 \cap N_\ell)(J_- \cap M)(H_-^1 \cap N_u)$,\\
\vspace{1mm}
$(J_- \cap P)H_-^1 = (H_-^1 \cap N_\ell)(J_- \cap M)(J_-^1 \cap N_u)$.
\end{center}

\subsection{}                    % Sec. 4.2

We are ready to prove the claim in 3.2.

\begin{prop}           % Prop. 4.2.1
Let $[\mathfrak{A},n,0,\beta]$ be a good skew simple stratum in $A$, with $\mathfrak{A} = \mathfrak{A}(\mathcal{L})$ and $e = e(\mathcal{L}_{\mathfrak{o}_E})$, and $\theta_- \in \mathcal{C}_-(\mathfrak{A},0,\beta)$. 
Let $P = MN_u$ be a parabolic subgroup of $G$ associated with $[\mathfrak{A},n,0,\beta]$. Put $t =[(e+1)/2]$. Then $\theta_-$ is trivial on both $H^1_-(\beta,\mathfrak{A}) \cap N_\ell$ and $H^1_-(\beta,\mathfrak{A}) \cap N_u$. After the identification $H^1_-(\beta,\mathfrak{A}) \cap M= \prod_{i=1}^t H^1(\beta,\mathfrak{A}^{(i)})$, we have
\begin{equation*}
\theta_- \vert (H_-^1(\beta,\mathfrak{A}) \cap M) = \theta^{(1)} \otimes \cdots \otimes \theta^{(t)},
\end{equation*}
where $\theta^{(i)} \in \mathcal{C}(\mathfrak{A}^{(i)},0, 2\beta)$, for $1 \le i \le [e/2]$, and 
if $t = (e+1)/2$ is an integer, we understand $\theta^{(t)} = \theta^{(t)}_-$ and $\mathcal{C}(\mathfrak{A}^{(t)},0,\beta) = \mathcal{C}_-(\mathfrak{A}^{(t)},0,\beta)$. Further, $\theta^{(i)}$ is a simple character of $H^1(2\beta, \mathfrak{A}^{(i)}) = H^1(\beta, \mathfrak{A}^{(i)})$ for $1 \le i \le [e/2]$.
\end{prop}
\begin{proof}
The first assertion follows directly from Proposition 4.1.1. As in 2.1, we have $\theta_- = \Vec{g}(\theta)) = \theta \vert H_-^1(\beta,\mathfrak{A})$, for some $\theta \in \mathcal{C}(\mathfrak{A},0,\beta)$ with $\theta^\gamma = \theta$.
From Proposition 4.1.1, $\theta \vert (H^1(\beta,\mathfrak{A} \cap \widetilde{M}) = \theta^{(1) \prime} \otimes \cdots \otimes \theta^{(e) \prime}$. We restrict this character to $\widetilde{G} \cap (A^i \times A^{e-i+1})$, for $1 \le i \le [e/2]$, and so have
\begin{equation*}
(\widetilde{G} \cap (A^i \times A^{e-i+1}))^\Gamma = \{(x,\overline{x}\ ^{-1}) \mid x \in (A^i)^\times = \mathrm{Aut}_F(V^i)\},
\end{equation*}
where $x \mapsto \overline{x}$ is the involution $A^i \to A^{e-i+1}$ defined in 4.1. Since $\theta((x,1)) = \theta^\gamma((x,1))$, for $x \in H^1(\beta,\mathfrak{A}^{(i)})$, we have $\theta^{(i) \prime}(x) = \theta^{(e-i+1) \prime}(\overline{x}\ ^{-1})$. 
Thus $\theta_-$ restricted to the factor $H^1(\beta, \mathfrak{A}^{(i)})$ is equal to $(\theta^{(i) \prime})^2$. Denote this character by $\theta^{(i)}$. 
Then $\theta^{(i)}$ belongs to $\mathcal{C}(\mathfrak{A}^{(i)},0,2\beta)$. Since it follows from \cite[Lemma 1 in 4.3]{Bl} that $H^1(2\beta, \mathfrak{A}^{(i)}) = H^1(\beta, \mathfrak{A}^{(i)})$, $\theta^{(i)}$ is a simple character of $H^1(\beta, \mathfrak{A}^{(i)})$ as in the assertion.
Moreover, if $t = (e+1)/2$ is an integer, clearly $\theta^{(t)} = \theta_-^{(t)} \in \mathcal{C}_-^1(\mathfrak{A}^{(t)},0,\beta)$. Hence the proof is completed.
\end{proof}

Suppose that $[\mathfrak{A},n,0,\beta],\ \theta_- \in \mathcal{C}_-(\mathfrak{A},0,\beta)$, and $P = MN_u$ be as in Proposition 4.2.1. From \cite[(5.1.1)]{BK1} and Proposition 3.1.1,
we obtain the unique irreducible representation $\eta_-$ (resp. $\eta^{(i)}$, resp. $\eta_-^{(t)}$) of $J_-^1(\beta,\mathfrak{A})$ (resp. $J^1(\beta,\mathfrak{A}^{(i)})$, resp. $J_-^1(\beta,\mathfrak{A}^{(t)})$) which contains $\theta_-$ (resp. $\theta^{(i)}$, resp. $\theta^{(t)}_-$). We define a subgroup of $J_-$ by
\begin{equation*}
J_{P,-}^1 = (J_-^1(\beta,\mathfrak{A}) \cap P)H_-^1(\beta,\mathfrak{A}).
\end{equation*}

\begin{prop}                % Prop. 4.2.2
Let notation and assumptions be as above. Then there is an irreducible representation $\eta_{P,-}$ of $J_{P,-}^1$ which satisfies the following conditions:
\begin{enumerate}
  \item $\eta_{P,-} \vert (J_-^1(\beta,\mathfrak{A}) \cap M) \simeq \eta^{(1)} \otimes \cdots \otimes \eta^{(t)}$;\
  \item $\eta_{P,-} \vert H_-^1(\beta,\mathfrak{A})$ is a multiple of $\theta_-$;\
  \item $\eta_{P,-} \vert (J_-^1(\beta,\mathfrak{A}) \cap N_u)$ is the trivial character;\
  \item $\eta_- = \mathrm{Ind}(\eta_{P,-}:J_{P,-},J_-)$,
\end{enumerate}
where in part (1), if $t = (e+1)/2 \in \mathbb{Z}$, we understand $\eta^{(t)} = \eta^{(t)}_-$.
\end{prop}
\begin{proof}
By using Proposition 4.2.1, we can prove the proposition in the same way as the proofs of \cite[(7.2.3), (7.2.4)]{BK1}.
\end{proof}

Let $y$ be the element in the proof Proposition 3.2.5. From Lemma 3.2.7, we may replace $[\mathfrak{A}',n',0,\beta]$ in that proposition by $[\mathfrak{A},n,0,\beta]$ in this subsection 4.2. From Lemma 4.1.2, we can write $y$ in the form $y = (y_1,\cdots,y_t)$, where if $t = (e+1)/2 \in \mathbb{Z}$, $y_t = 1$.

\begin{lem}            % Lem. 4.2.3
Let notation and assumptions be as above. For $1 \le i \le [e/2]$, there is an irreducible representation $\mu^{(i)}$ of $J(\beta,\mathfrak{A}^{(i)})$ which is intertwined by $y_i$ and is an extension of $\eta^{(i)}$. Moreover, if $t = (e+1)/2$ is an integer, there is an irreducible representation $\mu^{(t)} = \mu_-^{(t)}$ of $J_-(\beta,\mathfrak{A}^{(t)})$ which is an extension of $\eta^{(t)}$.
\end{lem}
\begin{proof}
In case $1 \le i \le [e/2]$, the assertion is just \cite[(7.2.10)]{BK1}. In case $t = (e+1)/2 \in \mathbb{Z}$, since $y_t = 1$, the assertion follows from Proposition 3.2.4.
\end{proof}

The following proposition is nothing but the claim in 3.2.

\begin{prop}                     % Prop. 4.2.4
There is an irreducible representation $\mu$ of $J_-(\beta,\mathfrak{A})$ which is intertwined by $y$ and such that $\mu \vert J_-^1 = \eta_-$.
\end{prop}
\begin{proof}
For $\eta^{(i)}$ in Lemma 4.2.3, put
\begin{equation*}
\eta_{N_u,-} = \eta^{(1)} \otimes \cdots \otimes \eta^{(t)},
\end{equation*}
where if $t = (e+1)/2 \in \mathbb{Z}$, we understand $J^1(\beta,\mathfrak{A}^{(t)}) = J_-^1(\beta,\mathfrak{A}^{(t)}),\ \eta^{(t)} = \eta^{(t)}_-$.
From Lemma 4.2.3, we obtain an irreducible representation of $J_-(\beta,\mathfrak{A}) \cap M = \prod_i J(\beta,\mathfrak{A}^{(i)})$ by
\begin{equation*}
\mu_{N_u,-} = \mu^{(1)} \otimes \cdots \otimes \mu^{(t)}.
\end{equation*}
Then $y = (y_i)$ clearly intertwines $\mu_{N_u,-}$. From the Iwahori decomposition in 4.1, we can inflate $\mu_{N_u,-}$ to a representation $\mu_{P,-}$ of $(J_-(\beta,\mathfrak{A}) \cap P)H_-^1(\beta,\mathfrak{A})$ by putting
\begin{equation*}
\mu_{P,-}(hmj) = \mu_{N_u,-}(m),\ \text{for}\ h \in H_-^1 \cap N_\ell,\ m \in J_- \cap M,\ j \in J_-^1 \cap N_u.
\end{equation*}
So put
\begin{equation*}
\mu_- = \mathrm{Ind}(\mu_{P,-}:(J_- \cap P)H_-^1,J_-).
\end{equation*}
From Proposition 4.2.2, $\eta_{P,-}$ induces $\eta_-$. Hence, from the Mackey restriction formula, we get $\mu_- \vert J_-^1 = \eta_-$, and from \cite[(4.1.5)]{BK1}, we can at once see that $y$ intertwines $\mu_-$. The proof is completed.
\end{proof}

The proposition completes the proof of Proposition 3.2.5, and hence that of Theorem 3.2.6.

\section{Simple types}          % Sec. 5

\subsection{}                    % Sec. 5.1

In this section, we define an analogue of a simple type for $GL(N,F)$ defined by \cite[(5.5.10)]{BK1}.

Suppose that $[\mathfrak{A},n,0,\beta]$ is a good skew simple stratum in $A = \mathrm{End}_F(V)$. Let $E = F[\beta]$, and $B = B_\beta$ the $A$-centralizer of $\beta$. Put $R = \dim_E(V)$. Let $\mathfrak{A} = \mathfrak{A}(\mathcal{L})$, $\mathfrak{B} = \mathfrak{A} \cap B$, and put $e = e(\mathcal{L}_{\mathfrak{o}_E})$.

From Proposition 1.3.1, $B^\times \cap G$ is the unramified unitary group of the non-degenerated $E$-anti-hermitian space $(V,\widetilde{h}_\beta)$, and from Proposition 2.2.2, it is of type C in the sense of \cite[(10.1.2)]{BT}. In this paragraph, we recall the structure of the affine Weyl group of $B^\times \cap G$ by \cite[10.1]{BT} and \cite{Ti}. 
Denote by $\Vec{G}_1$ the algebraic group defined over $E_0$ such that the group of $E_0$-rational points in $\Vec{G}_1$, denoted by $G_1 = \Vec{G}_1(E_0)$, is equal to $B^\times \cap G$.

In order to quote \cite[10.1]{BT} and \cite{Ti}, we rewrite the Witt basis $\mathcal{V}$ of (2.2.2) for $(V,\widetilde{h}_\beta)$ as follows: Let $r = R/2$ and $I = \{\pm 1, \cdots, \pm r\}$. Put $\mathcal{V} = \{e_i \vert\ i \in I\}$ with $e_{-r} = v_1, e_{-r+1} = v_2, \cdots, e_{-1} = v_r; e_1 = v_{r+1}, \cdots, e_r = v_{2r} = v_R$.

 We express elements of $G_1$ in the matrix form by this basis $\mathcal{V}$.
Let $\Vec{S}$ be the maximal $E_0$-split torus of $\Vec{G}_1$ defined by
\begin{equation*}
\Vec{S}(E_0) = \{\mathrm{Diag}(d_{-r}, \cdots, d_{-1}, d_1, \cdots, d_r) \vert\ d_i \in E_0\ \text{and}\ d_{-i}d_i = 1\ (i \in I)\}.
\end{equation*}
Let $\Vec{Z}$ be the centralizer of $\Vec{S}$, and $\Vec{N}$ the normalizer of $\Vec{S}$. Then we have
\begin{equation*}
\Vec{Z}(E_0) = \{\mathrm{Diag}(d_{-r}, \cdots, d_{-1}, d_1, \cdots, d_r) \vert\ d_i \in E\ \text{and}\ \overline{d_{-i}}d_i = 1\ (i \in I)\}.
\end{equation*}
Write $H = \Vec{Z}(E_0)$ for simplicity. Then $H$ has the maximal compact open subgroup
\begin{equation*}
H_0 = \{\mathrm{Diag}(d_{-r}, \cdots, d_{-1}, d_1, \cdots, d_r) \vert\ d_i \in \mathfrak{o}_E^\times\ \text{and}\ \overline{d_{-i}}d_i = 1\ (i \in I)\},
\end{equation*}
which coincides with $Z_c$ in the notation of \cite[1.2]{Ti}. Let $\Vec{W}_0 = \Vec{N}(E_0)/H$ and $\Vec{W} = \Vec{N}(E_0)/H_0$.

For $i, j \in I$, denote by $\delta_{i,j}$ the Kronecker delta. Then the group $\Vec{N}(E_0)$ consists of all matrices of the form $n = n(\sigma; d_{-r}, \cdots, d_r) = (g_{ij})$ with $g_{ij} = \delta_{i,\sigma(j)}d_j$, where (1) $\sigma$ is a permutation of $I$ which preserves the partition of $I$ in pairs $(-i,i)$, (2) $d_i \in E$ such that $\overline{d_{-i}}d_i = 1$, and (3) $\mathrm{det}(n) = \pm \prod_{i \in I} d_i = 1$.

For an integer $i$, $1 \le i \le r$, we define a character $a_i: \Vec{S} \to \Vec{GL}_1$ by 
\begin{equation*}
a_i(\mathrm{Diag}(d_{-r}, \cdots, d_r)) = d_{-i},
\end{equation*}
where $\Vec{GL}_1$ denotes the multiplicative group defined over $E_0$. Then $(a_i)_{1 \le i \le r}$ is a $\mathbb{Z}$-basis of the character group $X^* = \mathrm{Hom}_{E_0}(\Vec{S},\Vec{GL}_1)$.
Put $a_{-i} = -a_i,\ a_{ij} = a_i + a_j$ in $X^*$. Then $\Phi = \{a_{ij} \vert\ i, j \in I, i \ne \pm j\} \cup \{2a_i \vert\ i \in I\}$ is the root system of $(\Vec{G}_1,\Vec{S})$. 
Let $\Vec{U}_a$ be the root subgroup of $\Vec{G}_1$ associated with a root $a \in \Phi$. Associated with $a_{ij}$ and $2a_i$, we define elements $u_{ij}(c)\ (c \in E)$ and $u_i(0,d)\ (d \in E_0)$ of $G_1 = \Vec{G}_1(E_0)$ respectively as follows: $u_{ij}(c) = 1 + (g_{k\ell})$ with $g_{-j,i} = \overline{c},\ g_{-i,j} = -c$ and all other $g_{k\ell} = 0$, and $u_i(0,d) = 1 + (g_{k\ell})$ with $g_{-i,i} = d$ and all other $g_{k\ell} = 0$ (cf. \cite[(10.2.1)]{BT}), where we recall that $2 \in E_0$ is invertible. 
Then $\Vec{U}_{a_{ij}}(E_0) = \{u_{ij}(c) \vert\ c \in E\}$ and $\Vec{U}_{2a_i}(E_0) = \{u_i(0,d) \vert\ d \in E_0\}$. Further, we define elements $m(u_{ij}(c))\ (c \in E^\times)$ and $m(u_i(0,d))\ (d \in E_0^\times)$ of $\Vec{N}(E_0)$ by
\begin{eqnarray*}
m(u_{ij}(c)) &=& u_{-j,-i}(-c^{-1})u_{ij}(c)u_{-j,-i}(-c^{-1})\\
          &=& n(\sigma; d_{-r}, \cdots, d_r),
\end{eqnarray*}
where $\sigma = (i,-j)(j,-i),\ d_{-i} = c^{-1},\ d_{-j} = -(\overline{c})^{-1},\ d_j = -c,\ d_i = \overline{c}$ and all other $d_k = 1$, and
\begin{eqnarray*}
m(u_i(0,d)) &=& u_{-i}(0,-d^{-1})u_i(0,d)u_{-i}(0,-d^{-1})\\
         &=& n(\sigma; d_{-r}, \cdots, d_r)
\end{eqnarray*}
where $\sigma = (i,-i),\ d_{-i} = -d^{-1},\ d_i = d$ and all other $d_k = 1$.
For each integer $i$, $1 \le i \le r$, we define an element $h_i$ of $H_0$ by
$h_i = \mathrm{Diag}(d_{-r}, \cdots, d_r)\ \text{with}\ d_{-r+i-1} = d_{r-i+1} = -1\ \text{and all other}\ d_k = 1$.
Put
\begin{equation*}
 n_{s_i} =
 \begin{cases}
 m(u_{-(r-i),r-i+1}(1))h_i\ \hspace{3mm} (1 \le i \le r-1)\\
 m(u_{-1}(0,1))h_r\ \hspace{18mm} (i = r).
 \end{cases}
\end{equation*}
Then it follows from \cite[(10.1.2), (10.1.6)]{BT} that $n_{s_r}, n_{s_{r-1}}, \cdots, n_{s_1} \in \Vec{N}(E_0)$ correspond to the roots $2a_{-1}, a_{1,-2}, \cdots, a_{r-1,-r}$, respectively, which form a basis, $\Delta$,  of $\Phi$. The root $2a_{-r}$ is the highest root with respect to $\Delta$. Associated with this $2a_{-r}$, put
\begin{equation*}
n_{s_0} = n(\sigma; d_{-r}, \cdots, d_r)
\end{equation*}
where $\sigma = (-r,r),\ d_{-r} = - \varpi_E^{-1},\ d_r = \varpi_E$ and all other $d_i = 0$.

We now denote by $N_0$ the subgroup of $\Vec{N}(E_0)$ generated by $\{n_{s_1}, \cdots, n_{s_r}\}$, and by $N_{\mathfrak{o}}$ the subgroup of $\Vec{N}(E_0)$ generated by $N_0$ and $H_0$. Then $N_{\mathfrak{o}}$ consists of $n(\sigma; d_{-r}, \cdots, d_r) \in \Vec{N}(E_0)$ with $d_i \in \mathfrak{o}_E^\times$, and $\Vec{N}(E_0)$ is generated by $N_{\mathfrak{o}}$ and $H = \Vec{Z}(E_0)$.
We define a subgroup $\Vec{D}$ of $H$ by
\begin{equation*}
\Vec{D} = \{ \mathrm{Diag}(\varpi_E^{m_r}, \cdots, \varpi_E^{m_1}, \varpi_E ^{-m_1}, \cdots, \varpi_E ^{-m_r}) \mid m_1, \cdots, m_r \in \mathbb{Z}\}.
\end{equation*}
Then, since $E^\times = \varpi_E^{\mathbb{Z}} \times \mathfrak{o}_E^\times$, we have semi-direct products $H = \Vec{D} \cdot H_0$ and 
\begin{equation*}
\Vec{N}(E_0) = \Vec{D} \rtimes N_{\mathfrak{o}}.
\end{equation*}
 
Since the derived subgroup of $\Vec{G}_1$ is semi-simple and simply-connected, $\Vec{W} = \Vec{N}(E_0)/H_0$ is an affine Weyl group (cf. \cite[1.13]{Ti}). Since $E/E_0$ is unramified, it follows from \cite[1.6, 1.8]{Ti} that 
\begin{equation*}
\Phi_{af} = \{a_{ij} + \gamma \vert\ i, j \in I, i \ne \pm j, \gamma \in \mathbb{Z}\} \cup \{2a_i + \gamma \vert\ i \in I, \gamma \in \mathbb{Z}\}
\end{equation*}
(cf. \cite[1.15]{Ti}). The set $\{2a_{-1}, a_{1,-2}, \cdots, a_{r-1,-r}, 2a_{r} + 1\}$ is a basis of $\Phi_{af}$. For each $i$, $0 \le i \le r$, denote by $s_i \in \Vec{W}$ the image of $n_{s_i} \in \Vec{N}(E_0)$ under the canonical map $\Vec{N}(E_0) \to \Vec{W} = \Vec{N}(E_0)/H_0$.
Then it follows that $s_r, s_{r-1}, \cdots, s_1, s_0$ are the affine reflections associated with $2a_{-1}, a_{1,-2}, \cdots, a_{r-1,-r}, 2a_{r} + 1$, respectively.

                   % Prop. 5.1.1
\begin{prop}
Let notation and assumptions be as above. Then $\Vec{W}$ is a Coxeter group with a set of generators $\{s_0, s_1, \cdots, s_r\}$, and there is an isomorphism
\begin{equation*}
\Vec{W} \simeq \Vec{D} \rtimes \Vec{W}_0.
\end{equation*}
Identifying $\Vec{W}$ with $\Vec{D} \rtimes \Vec{W}_0$ via this isomorphism, we can regard $\Vec{W}_0$ as a finite Coxeter group with a set of generators $\{s_1, \cdots, s_r\}$.
\end{prop}
\begin{proof}
The first assertion has been proved above. For the second, from the above arguments, we have
\begin{equation*}
\Vec{W} = (\Vec{D} \rtimes N_{\mathfrak{o}})/H_0 = \Vec{D} \rtimes (N_{\mathfrak{o}}/H_0),
\end{equation*}
(cf. \cite[2.1]{IM}). By definition, $\{s_1, \cdots, s_r\}$ is contained in $N_0$ and so in $N_{\mathfrak{o}}$. Thus from \cite[(10.1.6), (10.1.7)]{BT} there is an isomorphism $N_{\mathfrak{o}}/H_0 \simeq \Vec{W}_0$, which shows the second assertion. The last is clear.
\end{proof}

\subsection{}           % Sec. 5.2

Suppose that $[\mathfrak{A},n,0,\beta]$ is a good skew simple stratum in $A = \mathrm{End}_F(V)$ as in 5.1. Let $E = F[\beta]$, and $B = B_\beta$ the $A$-centralizer of $\beta$. Let $\mathfrak{A} = \mathfrak{A}(\mathcal{L})$, $\mathfrak{B} = \mathfrak{A} \cap B$, and put $e = e(\mathcal{L}_{\mathfrak{o}_E})$. 
{\bf Hereafter we assume that $\mathfrak{A}$ is principal}. Then, for $R = \dim_E(V)$, there is a positive integer $f$ such that
\begin{equation*}
R = fe.
\end{equation*}

We choose self-dual $\mathfrak{o}_E$-lattice chains $\mathcal{L}_M,\ \mathcal{L}_m$ in $V$ such that $e(\mathcal{L}_M \vert \mathfrak{o}_E) = 1$, $e(\mathcal{L}_m \vert \mathfrak{o}_E) = R$, and $\mathcal{L}_M \subset \mathcal{L} \subset \mathcal{L}_m$, as in 3.1. 
In $B = B_\beta$, put $\mathfrak{B}_M = \mathrm{End}_{\mathfrak{o}_E}^0(\mathcal{L}_M)$ and $\mathfrak{B}_m = \mathrm{End}_{\mathfrak{o}_E}^0(\mathcal{L}_m)$, as in 3.1. 
Then $B^\times \cap G$ contains an Iwahori subgroup $\Vec{U}(\mathfrak{B}_m) = \mathfrak{B}_m \cap G$. From Proposition 5.1.1, we have the semi-direct product
\begin{equation*}
\Vec{W} = \Vec{D} \rtimes \Vec{W}_0,
\end{equation*}
and an Iwahori-Bruhat decomposition of $B^\times \cap G$
\begin{equation}
B^\times \cap G = \Vec{U}(\mathfrak{B}_m)\Vec{W}\Vec{U}(\mathfrak{B}_m).
\end{equation}

Let $V = \bigoplus_{i=1}^e V^i$ be the $E$-decomposition of $V$ subordinated to $\mathcal{L}_{\mathfrak{o}_E}$, and write $\mathcal{V} = \{v_i\}$ again.
For each integer $i$, $1 \le i \le e$, we may set 
\begin{center}
$\mathcal{V}^i = \mathcal{V} \cap V^i = \{ v_{(i-1)f+1}, v_{(i-1)f+2}, \cdots, v_{if}\}$.
\end{center}
For each $i$, $1 \le i \le e$, define an integer $\overline{i}$, with $1 \le \overline{i} \le e$ by
\begin{equation}
\overline{i} = e - i +1.
\end{equation}
For each $i$, $1 \le i \le [(e+1)/2]$, we rewrite the basis $\mathcal{V}^i$ and $\mathcal{V}^{\overline{i}}$ as follows: $\mathcal{V}^i = \{v_1^i, v_2^i, \cdots, v_f^i\}$, $\mathcal{V}^{\overline{i}} = \{v_1^{\overline{i}}, v_2^{\overline{i}}, \cdots, v_f^{\overline{i}}\}$, and 
\begin{equation}
\begin{cases}
  v_1^i = v_{(i-1)f+1}, v_2^i = v_{(i-1)f+2}, \cdots, v_f^i = v_{if},\\
  v_1^{\overline{i}} = v_{\overline{i}f}, v_2^{\overline{i}} = v_{\overline{i}f-1}, \cdots, v_f^{\overline{i}} = v_{(\overline{i}-1)f+1}.
\end{cases}
\end{equation}
If $i \ne \overline{i}$, each $Ev_j^i + Ev_j^{\overline{i}}$ is a hyperbolic subspace of $V$ by Lemma 2.2.3. If $i = \overline{i}$, $e$ is odd and $i = (e+1)/2$. Since $R = ef$ is even, so $f$ is also even. In this case, each $Ev_j^i + Ev_{f-j+1}^i$ is a hyperbolic subspace of $V$ as well.

 Put $\widetilde{\Vec{\mathfrak{M}}}(\mathfrak{B}) = \bigoplus_{i=1}^e \mathfrak{B}^i$ as in the proof of Proposition 3.2.5, where $\mathfrak{B}^i = \mathfrak{A}^{(i)} \cap \mathrm{End}_E(V^i)$ for $\mathfrak{A}^{(i)}$, defined in 4.1. Denote by $\Vec{D}(\mathfrak{B})$ the $\Vec{D}$-centralizer of $\widetilde{\Vec{\mathfrak{M}}}(\mathfrak{B})^\times$. We define elements $n_{\Vec{s}_1}, n_{\Vec{s}_2}, \cdots, n_{\Vec{s}_{[e/2]}}$ of $N_{\mathfrak{o}}$ as follows: for
 $1 \le i \le [e/2]-1$,
\begin{equation*}
\begin{cases}
 n_{\Vec{s}_i}: v_j^i \leftrightarrow v_j^{i+1},\ v_j^{\overline{i}} \leftrightarrow v_j^{\overline{i+1}},\ \text{for}\ 1 \le j \le f,\\
 n_{\Vec{s}_i} \vert V^k \equiv I,\ \text{for}\ k \ne i,\ \overline{i},
\end{cases}
\end{equation*}
and
\begin{equation*}
\begin{cases}
 n_{\Vec{s}_{[e/2]}}:v_j^{[e/2]} \mapsto v_j^{\overline{[e/2]}},\ v_j^{\overline{[e/2]}} \mapsto -v_j^{[e/2]},\ \text{for}\ 1 \le j \le f,\\
 n_{\Vec{s}_{[e/2]}} \vert V^k \equiv I,\ \text{for}\ k \ne [e/2]
\end{cases}
\end{equation*}
Let $\Vec{s}_1, \Vec{s}_2, \cdots, \Vec{s}_{[e/2]}$ be the canonical image of $n_{\Vec{s}_1}, n_{\Vec{s}_2}, \cdots, n_{\Vec{s}_{[e/2]}}$, respectively, under the canonical map $N_{\mathfrak{o}} \to \Vec{W}_0$.
Denote by $\Vec{W}_0(\mathfrak{B})$ the subgroup of $\Vec{W}_0$ generated by $\Vec{s}_1, \Vec{s}_2, \cdots, \Vec{s}_{[e/2]}$. From Proposition 5.1.1, we can define a subgroup,  $\Vec{W}(\mathfrak{B})$, of $\Vec{W}$ by
\begin{equation*}
\Vec{W}(\mathfrak{B}) = \Vec{D}(\mathfrak{B}) \rtimes \Vec{W}_0(\mathfrak{B}).
\end{equation*}
This group is the $\Vec{W}$-normalizer of $\widetilde{\mathfrak{M}}(\mathfrak{B})^\times$.

\subsection{}                % Sec. 5.3

Suppose that $[\mathfrak{A},n,0,\beta]$ is a good skew simple stratum in $A$, with $\mathfrak{A} = \mathfrak{A}(\mathcal{L})$ principal. Let $E = F[\beta]$, $e = e(\mathcal{L}_{\mathfrak{o}_E})$, and $B = B_\beta$ be the $A$-centralizer of $\beta$. 
We have $R = \dim_E(V) = ef$, for some positive integer $f$, as in 5.2. We note that $f$ must be even if $e$ is odd, since $R$ is even. Since $J_-(\beta,\mathfrak{A})/J_-^1(\beta,\mathfrak{A}) \simeq \Vec{U}(\mathfrak{B})/\Vec{U}^1(\mathfrak{B})$, from Proposition 2.2.5, there is a canonical isomorphism:
\begin{equation*}
J_-(\beta,\mathfrak{A})/J_-^1(\beta, \mathfrak{A}) \simeq
\begin{cases}
GL(f,k_E)^{e/2}                            &  \text{if $e$ is even,}\\
GL(f,k_E)^{(e-1)/2} \times U(f, k_{E_0})    &  \text{if $e$ is odd},
\end{cases}
\end{equation*}
where $U(f, k_{E_0})$ is the unitary group of a non-degenerate $k_E/k_{E_0}$-anti-hermitian form.

Suppose that $\sigma_0$ (resp. $\sigma_1$) is an irreducible cuspidal representation of $GL(f,k_E)$ (resp. $U(f, k_{E_0}))$. If $e$ is even, we define an irreducible representation $\sigma_-$ of $GL(f, k_E)^{e/2}$ by
\begin{equation*}
\sigma_- = \sigma_0 \otimes \cdots \otimes \sigma_0 = \displaystyle \bigotimes^{e/2} \sigma_0,
\end{equation*}
and if $e$ is odd, we define an irreducible representation $\sigma_-$ of $GL(f,k_E)^{(e-1)/2} \times U(f, k_{E_0})$ by
\begin{equation*}
\sigma_- = \sigma_0 \otimes \cdots \otimes \sigma_0 \otimes \sigma_1 = \bigl(\displaystyle \bigotimes^{(e-1)/2} \sigma_0 \bigl) \otimes \sigma_1.
\end{equation*}
Via the above isomorphism, we lift $\sigma_-$ to an irreducible representation, say again $\sigma_-$, of $J_-(\beta,\mathfrak{A})$. We can also regard $\sigma_-$ as an irreducible representation of $\Vec{U}(\mathfrak{B})$.

Let $[\mathfrak{A},n,0,\beta]$ is a good skew simple stratum in $A$, with $\mathfrak{A} = \mathfrak{A}(\mathcal{L})$ principal, and $\theta_- \in \mathcal{C}_-(\mathfrak{A},0,\beta)$. 
Then there is the unique irreducible representation $\eta_-$ of $J_-^1(\beta,\mathfrak{A})$ which contains $\theta_-$, and from Theorem 3.2.6, there is an irreducible representation $\kappa_-$ of $J_-(\beta,\mathfrak{A})$ which is a $\beta$-extension of $\eta_-$.

\begin{defn}             % Def. 5.3.1
Let notation and assumptions be as above. We say that a representation $\lambda_-$ is a {\it simple type} (of positive level) in $G$, if it has the following form:
\begin{equation*}
\lambda_- = \kappa_- \otimes \sigma_-,
\end{equation*}
for a $\beta$-extension $\kappa_-$  and an irreducible representation $\sigma_-$ of $J_-(\beta,\mathfrak{A})$ as above.
\end{defn}

The representation $\lambda_-$ is an analogue of a simple type for $GL_N(F)$ defined by part (a) of \cite[(5.5.10)]{BK1}.

\begin{prop}[(5.3.2) of \cite{BK1}]           % Prop. 5.3.2
Let $\lambda_- = \kappa_- \otimes \sigma_-$ be a simple type in $G$. Let $E = F[\beta]$, $B = B_\beta$, and $\mathfrak{B} = \mathfrak{A} \cap B$. Then $\lambda_-$ is irreducible and
\begin{equation*}
I_G(\lambda_-) = J_-(\beta,\mathfrak{A})I_{B^\times \cap G}(\sigma_- \vert \Vec{U}(\mathfrak{B}))J_-(\beta,\mathfrak{A}),
\end{equation*}
\end{prop}
\begin{proof}
By using Propositions 3.1.1 and 3.2.5, we can prove the assertion in the same way as the proof of \cite[(5.3.2)]{BK1}.
\end{proof}

Let $\Vec{W}(\mathfrak{B})$ be as in 5.2, and $\sigma_-$ be an irreducible representation of $\Vec{U}(\mathfrak{B})$ defined as above. Put
\begin{equation}
\Vec{W}(\sigma_-) = \{w \in \Vec{W}(\mathfrak{B}) \mid  (\sigma_-)^w \simeq \sigma_-\},
\end{equation}
where $(\sigma_-)^w(x) = \sigma_-(wxw^{-1})$ for $x \in \Vec{U}(\mathfrak{B})/\Vec{U}^1(\mathfrak{B})$.

The involution $x \mapsto \overline{x}: A^i \to A^{e-i+1}$, defined in 4.1, induces an involution $B^i \to B^{e-i+1}$. This is also induced by the involution on $B$ which is defined by $\widetilde{h}_\beta$. 
Under the identification $B^1 = \cdots = B^e = \mathbb{M}(f,E)$ via the Witt basis $\mathcal{V}$, the involution $B^i \to B^{e-i+1}$ induces naturally the involution on the $GL(f,\mathfrak{o}_E)$, and induces ones on $GL(f,k_E)$ and $U(f,k_{E_0})$. We write again by $^-$ these involutions. In particular, we have $U(f,k_{E_0}) = \{x \in GL(f,k_E) \mid x \overline{x} = 1\}.$

\begin{defn}               % Def. 5.3.3
Let $\sigma_0$ be an irreducible cuspidal representation of $GL(f,k_E)$. We define a representation $\sigma_0^*$ by
\begin{equation*}
\sigma_0^*(x) = \sigma_0(\overline{x}\ ^{-1}),
\end{equation*}
for $x \in GL(f,k_E)$. We say that the representation $\sigma_0$ is {\it self-dual}, if $\sigma_0 \simeq \sigma_0^*$.
\end{defn}

In this definition, the definition of $\sigma_0^*$ depends on the choice of the Witt basis $\mathcal{V}$. But the definition of self-dual does not depend on it. For, another Witt basis induces an involution on each $GL(f,\mathfrak{o}_E)$ which differs by a conjugation from the above involution $x \mapsto \overline{x}$.

If the component $\sigma_0$ of $\sigma_-$ is self-dual, it is easy to see that $\Vec{W}(\sigma_-)$ is equal to $\Vec{W}(\mathfrak{B})$.

In the next paragraph, we shall show the existence of a self-dual irreducible cuspidal representation $\sigma_0$ of $GL(f,k_E)$.

\begin{rem}            % Rem. 5.3.4
Any irreducible cuspidal representation $\sigma_1$ of $U(f,k_{E_0})$ is automatically self-dual.
%  \item If $\sigma_0 \simeq \sigma_0^*$, then the irreducible supercuspidal representation $\pi_0$ of $GL(f,E)$ given by (5.2.2) and its contragradient representation $\pi_0^\vee$ belong to the same inertial class, as stated in \cite{Bl} 2.2.
\end{rem}

\subsection{}           % Sec. 5.4

Suppose that $f$ is an integer $\ge 2$. For simplicity, write $k_0 = k_{E_0}$ and $k = k_E$. Let $k_0 = \mathbb{F}_q$ be the finite field of order $q$. Then $k = \mathbb{F}_{q^2}$ is the quadratic extension of $k_0$.
Let $x \mapsto \overline{x} = x^q$ be the non-trivial Galois involution of $k/k_0$. 
Let $\Vec{G} = \Vec{GL}_f$ be the general linear group of rank $f$ defined over $k$, and $G = \Vec{G}(k)$ the group of $k$-rational points in $\Vec{G}$. We define a Frobenius map $F_0$ on $\Vec{G}$ as follows: for $g = (g_{ij}) \in \Vec{G}$,
\begin{equation*}
F_0(g) = (\overline{g}_{ij}) = (g_{ij}^q).
\end{equation*}

Let $(\sigma_0,\mathcal{V})$ be an irreducible cuspidal representation of $G = \Vec{G}(k)$. From the remark below Definition 5.3.3, we may set the representation $(\sigma_0^*,\mathcal{V})$ of $G$ to be one defined by
\begin{equation*}
\sigma_0^*(g) = \sigma_0(^t(F_0(g))^{-1}),\ g \in G,
\end{equation*}
where $^tg$ denotes the transpose of $g$. 

Put $\Vec{G}_1 = \mathrm{Res}_{k/k_0}(\Vec{G})$, where $\mathrm{Res}$ denotes the functor of restrictions of scalars. We may identify $\Vec{G}_1$ with $\Vec{G} \times \Vec{G} = \Vec{G} \times F_0(\Vec{G})$. We define a Frobenius map $F_1$ on $\Vec{G}_1$ as follows: for $(x,y) \in \Vec{G}_1 = \Vec{G} \times \Vec{G}$,
\begin{equation*}
F_1(x,y) = (F_0(y),F_0(x)).
\end{equation*}
Then we have $\Vec{G}_1(k_0) = \Vec{G}(k)$ and $\Vec{G}_1(k_0) = \Vec{G}_1^{F_1} = \{g \in \Vec{G}_1 \vert\ F_1(g) = g\}$.

We define automorphisms $\delta$ and $\tau$ of $\Vec{G}_1$ by $\delta(x,y) = (y,x)$ for $x, y \in \Vec{G}$ and so by 
\begin{equation*}
\tau(g) =\ ^t\delta(g)^{-1},
\end{equation*}
for $g \in \Vec{G}_1$, where $^t(x,y) = (^tx, ^ty)$ for $(x,y) \in \Vec{G}_1 = \Vec{G} \times \Vec{G}$. Then for $g = (g, F_0(g)) \in \Vec{G}_1(k_0) = \Vec{G}(k) = G$, we have $\delta(g) = F_0(g)$ and
\begin{equation*}
\tau(g) =\ ^t(F_0(g))^{-1}.
\end{equation*}
Let $\chi_{\sigma_0}$ be the character of $\sigma_0$, i.e., $\chi_{\sigma_0}(g) = \mathrm{Tr}(\sigma_0(g)),\ g \in G$. Then by Deligne-Lusztig theory \cite[Proposition 8.3]{DL} (cf. \cite[Chapter 7]{Ca}), it is well known that there are a minisotropic maximal $k$-torus, $\Vec{T}$, of $\Vec{G}$ and a regular (in general position) character, $\theta$, of $T = \Vec{T}(k)$ such that
\begin{equation*}
\chi_{\sigma_0} = \pm R_{T,\theta}\ \hspace{5mm} \text{(Deligne-Lusztig character)}.
\end{equation*}
Then there are an extension $k_f = \mathbb{F}_{q^{2f}}$ of $k$ of degree $f$ and  the multiplicative group $\Vec{GL}_1$ defined over $k_f$ such that $\Vec{T}$ is isomorphic to $\mathrm{Res}_{k_f/k}(\Vec{GL}_1)$. We identify $\Vec{T} = \mathrm{Res}_{k_f/k}(\Vec{GL}_1)$. Put $\Vec{T}_1 = \mathrm{Res}_{k/k_0}(\Vec{T})$. Then we have $T = \Vec{T}(k) = \Vec{T_1}(k_0)$.

We study $\chi_{\sigma_0^*}$. The automorphism $\tau$ of $\Vec{G}_1$ satisfies the following properties:

\begin{enumerate}
  \item $\tau$ is defined over $k_0$,
  \item $\tau  \circ F_1 = F_1 \circ \tau$,
  \item $\tau^2 = \mathrm{Id}$.
\end{enumerate}

Since $\sigma_0^*(g) = \sigma_0(\tau(g)),\ g \in G$, by definition, we have
\begin{equation*}
\chi_{\sigma_0^*}(g) = \chi_{\sigma_0}(\tau(g)) = \pm R_{T,\theta}(\tau(g)),\ g \in G.
\end{equation*}

We prove the following.

\begin{prop}            % Prop. 5.4.1
We have $R_{T,\theta}(\tau(g)) = R_{\tau(T),\theta \circ \tau}(g),\ g \in G$.
\end{prop}
\begin{proof}
We first note that $T = \Vec{T}_1(k_0) = \Vec{T}(k)$ and $G = \Vec{G}_1(k_0) = \Vec{G}(k)$.
We adopt Deligne-Lusztig theory \cite{DL} (cf. \cite[Chapter 7]{Ca}) to the groups $\Vec{G}_1 \supset \Vec{T}_1$ defined over $k_0$.
Let $g \in G = \Vec{G}_1(k_0)$ and $g = us = su$ be the Jordan decomposition of $g$, where $u$ is the unipotent part of $g$ and $s$ is the semisimple part of $g$. Then we have the character formula \cite[Theorem 4.2]{DL} (cf. \cite[Theorem 7.2.8]{Ca}) as follows:
\begin{equation*}
R_{T,\theta}(g) = \displaystyle \frac{1}{\vert \Vec{C}^0(s)^{F_1} \vert} \displaystyle \sum_{x \in G, x^{-1}sx \in T_1} \theta(x^{-1}sx)Q_{xT_1x^{-1}}^{C^0(s)}(u)
\end{equation*}
where $\Vec{C}^0(s)$ denotes the connected centralizer of $s$ in $\Vec{G}_1$, and $Q_{T_1}^{G_1}(u) = R_{T_1,1}(u)$.

For the decomposition $g = us$, $\tau(g) = \tau(u)\tau(s)$ is also the Jordan decomposition with $\tau(u)$ unipotent and $\tau(s)$ semisimple. Thus we obtain
\begin{equation}
R_{T,\theta}(\tau(g)) = \displaystyle \frac{1}{\vert \Vec{C}^0(\tau(s))^{F_1} \vert} \displaystyle \sum_{x \in G, x^{-1}\tau(s)x \in T_1} \theta(x^{-1}\tau(s)x)Q_{xT_1x^{-1}}^{C^0(\tau(s))}(\tau(u))
\end{equation}
as well.

(a) From the properties of $\tau$, we have $\tau(\Vec{C}^0(\tau(s))^{F_1}) = \Vec{C}^0(s)^{F_1}$ and
\begin{equation*}
\vert \Vec{C}^0(\tau(s))^{F_1} \vert = \vert \Vec{C}^0(s)^{F_1} \vert.
\end{equation*}

(b) Similarly, from $\tau(x^{-1}\tau(s)x) = \tau(x)^{-1}s\tau(x)$, we obtain
\begin{equation*}
\theta(x^{-1}\tau(s)x) = \theta \circ \tau(\tau(x)^{-1}s\tau(x)),
\end{equation*}
and if \ $x \in G = \Vec{G}_1(k_0),\ x^{-1}\tau(s)x \in T = \Vec{T}_1(k_0)$, we have
\begin{equation*}
\tau(x) \in G,\ \tau(x)^{-1}s\tau(x) \in \tau(T).
\end{equation*}

(c) We again have $\tau(\Vec{C}^0(\tau(s))) = \Vec{C}^0(s),\ \tau(\tau(x)\tau(\Vec{T}_1)\tau(x)^{-1}) = x\Vec{T}_1x^{-1}$. The Lang variety $\widetilde{X} = L^{-1}(\Vec{U})$ is associated with a Borel subgroup $\Vec{B} = \Vec{T}_1\Vec{U}$ of $\Vec{G}_1$, where $\Vec{U}$ is the unipotent radical of $\Vec{B}$. Thus 
\begin{equation*}
\tau(\widetilde{X}) = \tau(L^{-1}(\Vec{U})) = L^{-1}(\tau(\Vec{U}))
\end{equation*}
is associated with $\tau(\Vec{B}) = \tau(\Vec{T}_1)\tau(\Vec{U})$. Hence we get
\begin{equation*}
Q_{xT_1x^{-1}}^{C^0(\tau(s))}(\tau(u)) =Q_{\tau(x)\tau(T_1)\tau(x)^{-1}}^{C^0(s)}(u).
\end{equation*}
From (a), (b) and (c), it follows that the right-hand side of $R_{T,\theta}(\tau(g))$ in (5.4.1) is equal to $R_{\tau(T),\theta \circ \tau}(g)$. The proof is completed.
\end{proof}

We further study the right-hand side of the equality in Proposition 5.4.1 and obtain the following.

\begin{prop}                % Prop. 5.4.2
We have $\chi_{\sigma_0^*} = \pm R_{T,\overline{\theta}_1}$ for the unique character $\theta_1$ of $T = \Vec{T}(k)$ with $\theta_1^q = \theta$.
\end{prop}
\begin{proof}
From proposition 5.4.1, we have
\begin{equation*}
\chi_{\sigma_0^*} = \pm R_{\tau(T),\theta \circ \tau}.
\end{equation*}

We can represent $T = \Vec{T}(k)$ in $G = \Vec{G}(k)$ as follows: We choose an element $\alpha \in k_f^\times = k_f - \{0\}$ satisfying
\begin{enumerate}
 \item $\{1,\alpha,\cdots,\alpha^{f-1}\}$ is a basis of $k_f$ as a $k$-vector space,
 \item for the regular representation $\rho: k_f^\times \to G = \Vec{GL}_f(k)$ with respect to the basis, we may set
\begin{equation*}
T = \{\rho(x) \vert x \in k_f^\times\}.
\end{equation*}
\end{enumerate}
Write $\overline{\rho(x)} = F_0(\rho(x))$ for simplicity. We have $\overline{\alpha} = F_0(\alpha) \in k_f^\times$ and $\{1,\overline{\alpha},\cdots,\overline{\alpha}^{f-1}\}$ is also a $k$-basis of $k_f$. Let $\rho':k_f^\times \to G$ be the regular representation of $k_f^\times$ with respect to this new basis. Then, for $x \in k_f^\times$, we can check that
\begin{equation*}
\overline{\rho(x)} = \rho'(x^q) = \rho'(x)^q
\end{equation*}
and that there is an element $g_0 \in G$ such that $\rho'(x) = g_0\rho(x)g_0^{-1},\ x \in k_f^\times$. Hence we have
\begin{equation*}
\overline{\rho(x)} = g_0\rho(x)^qg_0^{-1},\ x \in k_f^\times
\end{equation*}
and
\begin{equation*}
\overline{T} = \{\overline{\rho(x)} \vert x \in k_f^\times\} = g_0Tg_0^{-1}. 
\end{equation*}
However, for $g \in \tau(T) =\ ^t\overline{T}$, we have
\begin{equation*}
\theta \circ \tau(g) = \theta((\overline{^tg})^{-1}) = \overline{\theta}(\overline{^tg}).
\end{equation*}
Since the Pontrjagin dual $\hat{T}$ of $T$ is (non-canonically) isomorphic to $k_f^{\times} = (\mathbb{F}_{q^{2f}})^\times$, it is a cyclic group of order $q^{2f} - 1$. It follows that there is a chracter $\theta_1$ of $T$ with $\theta_1^q = \theta$ as in the assertion. Ths we have $\theta \circ \tau(g) = \overline{\theta}_1^q(\overline{^tg})$. We can write $\overline{^tg} = \rho(x)$ for some $x \in k_f^\times$, so that 
\begin{equation*}
^tg = \overline{\rho(x)} = g_0\rho(x)^qg_0^{-1}.
\end{equation*}
From $\overline{T} = g_0Tg_0^{-1}$ above, it follows that $^{g_0}\overline{\theta}_1$ is a unique character of $\overline{T}$. Thus
\begin{equation*}
(^{g_0}\overline{\theta}_1)(^tg) = \overline{\theta}_1(g_0^{-1}(^tg)g_0) = \overline{\theta}_1(\rho(x)^q) = \overline{\theta}_1^q(\overline{^tg}) = \overline{\theta}(\overline{^tg}).
\end{equation*}
Hence, for $g \in \tau(T) =\ ^t\overline{T}$, we have
\begin{equation*}
\theta \circ \tau (g) =\ ^{g_0}\overline{\theta}_1(^tg).
\end{equation*}

Let $h$ be a generator of the group $\tau(T) =\ ^t\overline{T}$. Then the elements $h \in \tau(T)$ and $^th \in \overline{T}$ are both regular semisimple, and have the same characteristic polynomial. Thus there is an element $g_1 \in G$ such that $h = g_1(^th)g_1^{-1}$, and it does not depend on the choice of $h$. So we have $\tau(T) =\ ^t\overline{T} = g_1(\overline{T})g_1^{-1}$. Hence, since $^tg = g_1^{-1}gg_1$ for $g \in \tau(T)$, we have
\begin{equation*}
^{g_0}\overline{\theta}_1(^tg) =\ ^{g_0}\overline{\theta}_1(g_1^{-1}gg_1) =\ ^{g_1g_0}(\overline{\theta}_1)(g).
\end{equation*}
Consequently, it follows that
\begin{equation*}
\theta \circ \tau(g) =\ ^{g_1g_0}(\overline{\theta}_1)(g),\ g \in \tau(T)
\end{equation*}
and that
\begin{equation*}
(g_1g_0)^{-1}\tau(T)(g_1g_0) = T.
\end{equation*}
By the orthonality relation \cite[Theorem 6.8]{DL}for $R_{T,\theta}$, we obtain
\begin{equation*}
R_{\tau(T),\theta \circ \tau} = R_{\tau(T),^{g_1g_0}\overline{\theta}_1} = R_{T,\overline{\theta}_1},
\end{equation*}
which complete the proof.
\end{proof}

\begin{cor}         % Cor. 5.4.3
If the integer $f$ is odd, there is an irreducible cuspidal representation $\sigma_0$ of $G = GL(f,k_E)$ such that $\sigma_0$ is equivalent to $\sigma_0^*$.
\end{cor}
\begin{proof}
Let $\Vec{T}$ be a minisotropic maximal $k$-torus of $\Vec{G}$, and $\theta$ be a regular character of $T = \Vec{T}(k)$ such that $\chi_{\sigma} = \pm R_{T,\theta}$.
We have $\sigma_0 \simeq \sigma_0^*$ if and only if $\chi_{\sigma_0} = \chi_{\sigma_0^*}$. Thus it follows from Proposition 5.4.2 that $\sigma_0 \simeq \sigma_0^*$ is equivalent to $R_{T,\theta} = R_{T,\overline{\theta}_1}$, where $\theta_1^q = \theta$. By the orthogonality relations for $R_{T,\theta}$, the last condition is equivalent to the condition that there is a non-negative integer $\ell$ such that $\theta^{q^{2\ell}} = \overline{\theta}_1$, that is, $\theta^{q^{2\ell+1}} = \theta^{-1}$.

Let $\xi$ be a generator of $\hat{T} \simeq k_f^\times$. Take $\theta = \xi^{q^f-1}$ in $\hat{T}$. Then we have $\theta^{q^f+1} = (\xi^{q^f-1})^{q^f+1} = \xi^{q^{2f}-1} =1$. Further we can show directly that $\theta^{q^{2i}} \ne \theta$ for any integer $i$, $1 \le i \le f-1$, that is, $\theta$ is regular. Hence the proof is completed.
\end{proof}

\subsection{}            % Sec. 5.5

We moreover study the $G$-intertwining of a simple type $(J_-(\beta, \mathfrak{A}),\lambda_-)$ in $G$.

\begin{prop}[cf. (5.5.11) of \cite{BK1}]         % Prop. 5.5.1
Let $[\mathfrak{A},n,0,\beta]$ be a good skew simple stratum in $A$, with $\mathfrak{A} = \mathfrak{A}(\mathcal{L})$ principal, and $\lambda_- = \kappa_- \otimes \sigma_-$ a simple type in $G$ attached to $[\mathfrak{A},n,0,\beta]$. Then we have
\begin{equation*}
I_G(\lambda_-) \subset J_-(\beta,\mathfrak{A})\Vec{W}(\mathfrak{B})J_-(\beta,\mathfrak{A}).
\end{equation*}
\end{prop}
\begin{proof}
If $g \in G$ intertwines $\lambda_-$, from Proposition 5.3.2, $g \in J_-yJ_-$ for some $y \in B^\times \cap G$ and $y$ intertwines $\sigma_- \vert \Vec{U}(\mathfrak{B})$. 
Since $J_-$ contains the Iwahori subgroup $\Vec{U}(\mathfrak{B}_m)$ of $B^\times \cap G$, by the Iwahori-Bruhat decomposition of (5.2.1), we may take $y \in \Vec{W}$. Thus the result follows from the following lemma, which is an analogue of \cite[(5.5.5)]{BK1}.
\end{proof}

\begin{lem}              % Lem. 5.5.2
If $w \in \Vec{W}$ intertwines $\sigma_- \vert \Vec{U}(\mathfrak{B})$, then $w \in \Vec{W}(\mathfrak{B})$.
\end{lem}
\begin{proof}
It is hard to prove this lemma (cf. \cite[(5.5.5)]{BK1}).

It follows from the argument in 5.2 that the $\Vec{W}$-normalizer of $\widetilde{\mathfrak{M}}(\mathfrak{B})^\times$ is equal to $\Vec{W}(\mathfrak{B}) = \Vec{D}(\mathfrak{B}) \rtimes \Vec{W}_0(\mathfrak{B})$. Thus, if $w \in \Vec{W}$ intertwines $\sigma_- \vert \Vec{U}(\mathfrak{B})$, it is enough to prove that $w$ normalizes $\widetilde{\Vec{\mathfrak{M}}}(\mathfrak{B})^\times$.

We now assume that $w \in \Vec{W}$ does not normalize $\widetilde{\Vec{\mathfrak{M}}}(\mathfrak{B})^\times$. Put $\mathcal{L}_{\mathfrak{o}_E} = \{L_k \mid k \in \mathbb{Z}\}$ with $L_0^\natural = L_0$. 
Let $V = \bigoplus_{i=1}^e V^i$ be the $E$-decomposition of $V$ subordinated to $\mathcal{L}_{\mathfrak{o}_E}$, $L_k = \coprod_{i=1}^e L_k^i$, $L_k^i = L_k \cap V^i$, for $k \in \mathbb{Z}$, $\mathcal{V} = \{v_1, v_2, \cdots, v_R\}$ and $\mathcal{V} = \coprod_{i=1}^e \mathcal{V}^i$ be as in Lemma 2.2.3. Let $L_k \in \mathcal{L}$. Then for each integer $i$, $1 \le i \le e$, there is an integer $m(i,k)$ such that
\begin{equation*}
L_k \cap V^i = L_k^i = \mathfrak{p}_E^{m(i,k)}<\mathcal{V}^i>.
\end{equation*}
We denote this lattice by $<\mathfrak{p}_E^{m(i,k)}>^i$. Thus we have
\begin{equation}
L_k = \displaystyle \bigoplus_{i=1}^e L_k^i = \displaystyle \bigoplus_{i=1}^e <\mathfrak{p}_E^{m(i,k)}>^i
\end{equation}

We prepare the following three lemmas.

\begin{lem}     % Lem. 5.5.3
The function $m(i,k)$ on $\{1, \cdots, e\} \times \mathbb{Z}$ satisfies the following conditions:
\begin{enumerate}
 \item $m(1,0) = m(2,0) = \cdots = m(e,0) = 0$, \
 \item $m(1,k) \le m(2,k) \le \cdots \le m(e,k) \le m(1,k) + 1$, for $k \in \mathbb{Z}$, and precisely one of these inequalities is strict,\
 \item for each $i$, $m(i,k)$ jumps at $k$, with $k \equiv -i\ (\mathrm{mod}\ e)$, that is, $m(i,k+1) = m(i,k) + 1$.
\end{enumerate}
\end{lem}
\begin{proof}
Straightforward.
\end{proof}

\begin{lem}     % Lem. 5.5.4
Let $w \in \Vec{W}$. Then, for each integer $j$, $1 \le j \le R/2$, there are integers $d_j$ and $k = k(j)$, determined uniquely by $j$, such that
\begin{equation*}
w(\mathfrak{o}_Ev_j) = \mathfrak{p}_E^{d_j}v_k,\ w(\mathfrak{o}_Ev_{R-j+1}) = \mathfrak{p}_E^{-d_j}v_{R-k+1}.
\end{equation*}
\end{lem}
\begin{proof}
This follows straightforward by the definition of $\Vec{W}$ in 5.1.
\end{proof}

We recall $\overline{i} = e -i + 1$, for $i \in \{1, 2, \cdots, e\}$, defined by (5.2.2).

\begin{lem}     % Lem. 5.5.5
Let $w \in \Vec{W}$. The element $w$ permutes $\{L_k^i \vert i \in \{1, 2, \cdots, e\}, k \in \mathbb{Z}\}$ if and only if for each $L_k^i = <\mathfrak{p}_E^{m(i,k)}>^i,\ L_k^{\overline{i}} = <\mathfrak{p}_E^{m(\overline{i},k)}>^{\overline{i}}$, there are integers $\delta_i, j, k', k^{\prime\prime}$ such that
\begin{equation*}
w(L_k^i) = L_{k'}^j = <\mathfrak{p}_E^{m(i,k)+\delta_i}>^j,\ w(L_k^{\overline{i}}) = L_{k^{\prime\prime}}^{\overline{j}} = <\mathfrak{p}_E^{m(\overline{i},k)-\delta_i}>^{\overline{j}}.
\end{equation*}
\end{lem}
\begin{proof}
This follows directly from Lemma 5.5.4.
\end{proof}

By Lemma 5.5.5, we may assume that the element $w$ does not permute $\{L_k^i\}$ as in the proof of \cite[(5.5.5)]{BK1}.

For $i \in \{1, \cdots, e\}$ and $j \in \{1, \cdots, f\}$, let the basis $\mathcal{V}^i = \{v_j^i\}$ to be as in (5.2.3), and define an integer $\nu(i,j)$ in $\{1, \cdots, e\}$ by
\begin{equation*}
w^{-1}(v_j^i) \in V^{\nu(i,j)}.
\end{equation*}
Let $k$ be any integer, and $L_k$ be the lattice in $\mathcal{L}$ of the form (5.5.1). Then $wL_k \cap Ev_j^i \subset w(L_k \cap V^{\nu(i,j)})$, and from Lemma 5.5.4, there is an integer $d_j^i$ such that
\begin{equation*}
wL_k \cap Ev_j^i = \mathfrak{p}_E^{m(\nu(i,j),k)+d_j^i}v_j^i.
\end{equation*}
We remark that the integers $\nu(i,j)$ and $d_j^i$ depend on the element $w$ of $\Vec{W}$, but they do not depend on $k$ of $L_k$. 

Let $i$ be an integer with $1 \le i \le [(e+1)/2]$. Then, for each integer $k$, we have
\begin{equation*}
wL_k \cap (V^i + V^{\overline{i}}) = (wL_k \cap V^i) + (wL_k \cap V^{\overline{i}}).
\end{equation*}
If $i \ne \overline{i}$, then, again by Lemma 5.5.4, we have $w^{-1}(v_j^{\overline{i}}) \in V^{\overline{\nu(i,j)}}$, so that $\nu(\overline{i},j) = \overline{\nu(i,j)}$, and similarly $d_j^{\overline{i}} = - d_j^i$. If $i = \overline{i}$, then we have $\nu(i,f-j+1) = \overline{\nu(i,j)}$ and $d_{f-j+1}^i = - d_j^i$ as well. We put
\begin{equation*}
f' = 
 \begin{cases}
 f \hspace{6mm} \text{if}\ i \ne \overline{i},\\
 f/2 \hspace{3mm} \text{if}\ i = \overline{i},
 \end{cases}
\end{equation*}
and for each $j \in \{1, \cdots, f'\}$, rewrite
\begin{equation*}
v_{-j}^i =
 \begin{cases}
 v_j^{\overline{i}} \hspace{13mm} \text{if}\ i \ne \overline{i},\\
 v_{2f'-j+1}^i \hspace{3mm} \text{if}\ i = \overline{i}.
 \end{cases}
\end{equation*}
Then $\{v_j^i,\ v_{-j}^i\ \vert\ j \in \{1, \cdots, f'\}\}$ form a basis of $V^i + V^{\overline{i}}$, and for each integer $k$, we have
\begin{equation}
wL_k \cap (V^i + V^{\overline{i}}) = \displaystyle \sum_{j=1}^{f'} \mathfrak{p}_E^{m(\nu(i,j),k)+d^i_j}v_j^i + \displaystyle \sum_{j=1}^{f'} \mathfrak{p}_E^{m(\overline{\nu(i,j)},k)-d^i_j}v_{-j}^i.
\end{equation}

\begin{lem}    % Lem. 5.5.6
There is an integer $i$, $1 \le i \le [(e+1)/2]$, which satisfies the condition: ``not $\nu(i,1) = \cdots = \nu(i,f)$ or not $d_1^i = \cdots = d_f^i"$.
\end{lem}
\begin{proof}
Suppose that there is no integer $i$ as in the assertion. Then, for $i = \overline{i} = (e+1)/2$, we have $\nu(i,1) = \cdots = \nu(i,f') = (e+1)/2$ and $d_1^i = \cdots = d_{f'}^i = 0$, so that $w(L_k^i) = L_k^i$, for $k \in \mathbb{Z}$.
For $i$, with $i \ne \overline{i}$, put $\nu = \nu(i,1) = \cdots = \nu(i,f')$ and $d = d_1^i = \cdots = d_{f'}^i$. For each integer $k$, it follows from the above argument that
\begin{equation*}
wL_k^\nu = wL_k \cap V^i = <\mathfrak{p}_E^{m(\nu,k)+d}>^i,
\end{equation*}
whence, by Lemma 5.5.3, we have $wL_k^\nu = <\mathfrak{p}_E^{m(i,\ell)}>^i = L_{\ell}^i$ for some integer $\ell$. Hence the element $w$ permutes $\{L_k^i\}$, which contradicts the assumption on $w$. The proof is completed.
\end{proof}

We fix such an integer $i$ as in Lemma 5.5.6, and for each $j \in \{1, \cdots, f'\}$, write $\mu(j), d_j$, and $v_j$ for $\nu(i,j), d_j^i$, and $v_j^i$, respectively. Put $W = V^i + V^{\overline{i}}$, and
\begin{equation*}
W_+ = \displaystyle \sum_{j=1}^{f'} Ev_j,\ W_- = \displaystyle \sum_{j=1}^{f'} Ev_{-j}.
\end{equation*}
Then we have $W = W_+ \oplus W_-$, and $W_+$ and $W_-$ are both maximal totally isotropic subspaces of $W$ with respect to $\widetilde{h}_\beta \vert W$.

\begin{rems}   % Remarks 5.5.7
(i) In case $i = \overline{i}$, the condition in Lemma 5.5.6 is divided into the following two cases:
\begin{enumerate}
 \item not $\nu(1) = \cdots = \nu(f')$ or not $d_1 = \cdots = d_{f'}$,\
 \item $\nu(1) = \cdots = \nu(f'), d_1 = \cdots = d_{f'}$, and $``\nu(f') \ne \overline{\nu(1)}$ or $d_1 \ne 0"$.
\end{enumerate}
(ii) In case $i \ne \overline{i}$, it is nothing but (1) above, since $f' = f$.
\end{rems}

For $wL_k \cap W$ of (5.5.2), put
\begin{equation*}
M = \{(\nu(j),d_j),\ (\overline{\nu(j)},-d_j)\ \vert\ j \in \{1, \cdots, f'\}\},
\end{equation*}
where the $(\nu(j),d_j)$ do not depend on $k$ of $L_k$ as remarked above.
We define a linear order, $\prec$, on the set $M$ by $(\nu',d') \prec (\nu,d)$ if and only if either ``\ $d' < d$ \ " or ``\ $d' = d$ and $\nu' < \nu$ \ ".

\begin{lem}   % Lem. 5.5.8
If elements $(\nu,d)$ and $(\nu',d')$ in $M_i$ satisfy $(\nu',d') \prec (\nu,d)$, then $m(\nu',k)+d' \le m(\nu,k)+d$ and $m(\overline{\nu},k)-d \le  m(\overline{\nu'},k)-d'$, for any integer $k$.
\end{lem}
\begin{proof}
This follows directly from Lemma 5.5.3 (2).
\end{proof}

Denote by $\tau_{j\ell}$ the product of the transposition of $v_j$ and $v_\ell$ in $\mathcal{V}^i$ with that of $v_{-j}$ and $v_{-\ell}$ in $\mathcal{V}^{\overline{i}}$. 
By Lemma 5.5.8, multiplying an element $u$ which is a product of appropriate $\tau_{j\ell}$'s, we can permute $\{v_1, \cdots, v_{f'}\}$ (so $\{v_{-1}, \cdots, v_{-f'}\}$) so as to have
\begin{equation*}
uwL_k \cap W = \displaystyle \sum_{j=1}^{f'} \mathfrak{p}_E^{\mu(j,k)}v_j + \displaystyle \sum_{j=1}^{f'} \mathfrak{p}_E^{\mu'(j,k)}v_{-j},
\end{equation*}
with
\begin{equation*}
\mu(1,k) \le \cdots \le \mu(f',k),\ \mu'(f',k) \le \cdots \le \mu'(1,k).\,
\end{equation*}
for each $k$.

Let $(\nu_0,d_0)$ be the maximal element in the set $M$ with respect to the order $\prec$. Then we have $d_0 \ge 0$, and $\mu(f',k) = m(\nu_0,k) + d_0$ or $\mu'(1,k) = m(\nu_0,k) + d_0$. We may assume
\begin{equation*}
\mu(f',k) = m(\nu_0,k) + d_0,
\end{equation*}
up to the transposition of $W_+$ and $W_-$. Put
\begin{equation*}
\kappa = e - \nu_0,
\end{equation*}
and for $uwL_\kappa \cap W$ and $uwL_{\kappa+1} \cap W$, write
\begin{equation*}
a_j = \mu(j,\kappa), a_j' = \mu'(j,\kappa); b_j = \mu(j,\kappa+1), b_j' = \mu'(j,\kappa+1)
\end{equation*}
for $j \in \{1, \cdots, f'\}$. Then, from the choice of $\kappa$, we have
\begin{equation}
m(1,\kappa) = \cdots = m(\nu_0,\kappa) = 0, m(\nu_0+1,\kappa) = \cdots = m(e,\kappa) =1
\end{equation}
and $m(\nu_0,\kappa+1) = 1.$ Thus, by definition, we have
\begin{center}
$a_{f'} = \mu(f',\kappa) = m(\nu_0,\kappa) + d_0 = d_0$,\
$b_{f'} = \mu(f',\kappa+1) = m(\nu_0,\kappa+1) + d_0 = 1 + d_0 = a_{f'} + 1.$
\end{center}
This implies $uwL_\kappa \cap W \supsetneq uwL_{\kappa+1} \cap W$.

\begin{lem}   % Lem. 5.5.9
(1) In case $i \ne \overline{i}$, there is an integer $s$, $1 \le s \le f'$, such that $b_1 \le \cdots \le b_s < b_{s+1} = \cdots = b_{f'}$.

(2) In case $i = \overline{i}$, we can replace the element $u$ of $\Vec{W}$ so that there is an integer $s$, $0 \le s \le f'$, such that $b_1 \le \cdots \le b_s < b_{s+1} = \cdots = b_{f'}$ and $b_1' < b_{s+1}$. In particular, if $s = 0$, then $b_1 = \cdots = b_{f'} > b_{f'}' = \cdots = b_1'$.
\end{lem}
\begin{proof}
We first assume (1) not $\nu(1) = \cdots = \nu(f')$ or not $d_1 = \cdots = d_{f'}$ in Remarks 5.5.7. Then there is an integer $s$, $1 \le s \le f'$, which satisfies $b_1 \le \cdots \le b_s < b_{s+1} = \cdots = b_{f'}$. For, if not all the $\nu(j)$ are equal, then there is some $s$ such that $a_s = b_s$. 
Thus the maximal one of these is the desired. If all the $\nu(j)$ are equal, not all the $d_j$ are equal. Thus, if $a_s < a_{f'}$, then $b_s \le a_s + 1 < a_{f'} + 1 = b_{f'}$. Hence, similarly, we get $s$ as claimed. If $i \ne \overline{i}$, then, since the assumption (1) is satisfied, the assertion (1) is proved. 

So, let $i = \overline{i}$. Denote by $\tau_j$ the transposition of $v_j$ and $v_{-j}$. If we have $b_{s+1} = b_{f'} = b_1'$, we can replace $u$ by the product of appropriate $\tau_{j\ell}$'s and $\tau_m$'s 
so that $b_{f'}' \le \cdots \le b_1' < b_{s+1}$. Then we have $0 \le s \le f'$ and $b_1 \le \cdots \le b_s < b_{s+1} = \cdots = b_{f'}$ as the assertion says.

We next assume (2) $\nu(1) = \cdots = \nu(f'), d_1 = \cdots = d_{f'}$, and ``$\nu(f') \ne \overline{\nu(1)}$ or $d_1 \ne 0$" in Remarks 5.5.7. Then similarly we can replace $u$ 
so that $\mu(1,k) = \cdots = \mu(f',k) > \mu'(f',k) = \cdots = \mu'(1,k)$, for any integer $k$. In particular, for $k = \kappa + 1$, $b_1 = \cdots = b_{f'} > b_{f'}' = \cdots = b_1'$. The proof is completed.
\end{proof}

Via the integer $s$ in Lemma 5.5.9, we decompose the spaces $W_+$ and $W_-$ into
\begin{equation*}
W_+ = W_1 \oplus W_2,\ W_- = W_2^{\natural} \oplus W_1^{\natural}
\end{equation*}
by setting
\begin{equation*}
W_1 = \displaystyle \sum_{j=1}^s Ev_j,\ W_2 = \displaystyle \sum_{j=s+1}^{f'} Ev_j
,\ W_2^{\natural} = \displaystyle \sum_{j=s+1}^{f'} Ev_{-j},\ W_1^{\natural} = \displaystyle \sum_{j=1}^s Ev_{-j}.
\end{equation*}
Here, if $s = 0$, we understand $W_1 = W_1^{\natural} = (0)$. Then we have $W = W_2 \oplus (W_1^\natural \oplus W_1) \oplus W_2$. We produce a self-dual $\mathfrak{o}_E$-lattice chain in $W$ of $\mathfrak{o}_E$-period equal to 2 or 3. We first define $\mathfrak{o}_E$-lattices in $W_+$ by
\begin{equation*}
\overline{L}_0 = \displaystyle \sum_{j=1}^{f'} \mathfrak{o}_Ev_j \supsetneq \overline{L}_1 = \displaystyle \sum_{j=1}^{s}\mathfrak{o}_Ev_j + \displaystyle \sum_{j=s+1}^{f'}\mathfrak{p}_Ev_j \supsetneq \varpi_E\overline{L}_0,
\end{equation*}
and in $W_-$
\begin{equation*}
\overline{L}_0^\natural = \displaystyle \sum_{j=1}^{f'} \mathfrak{o}_Ev_{-j} \supsetneq \varpi_E\overline{L}_1^\natural = \displaystyle \sum_{j=s+1}^{f'} \mathfrak{o}_Ev_{-j} + \displaystyle \sum_{j=1}^s\mathfrak{p}_Ev_{-j} \supsetneq \varpi_E\overline{L}_0^\natural.
\end{equation*}
Multiplying these $\mathfrak{o}_E$-lattices by $\varpi_E^m$, $m \in \mathbb{Z}$, we obtain an $\mathfrak{o}_E$-lattice chain, $\overline{\mathcal{L}}$, in $V^i$. Further, in $W$, we define
\begin{equation*}
M_0 = \overline{L}_0^\natural \oplus \overline{L}_0,\ M_1 = \overline{L}_0^\natural \oplus \varpi_E\overline{L}_1,\ M_2 = \varpi_E\overline{L}_1^\natural \oplus \varpi_E\overline{L}_0.
\end{equation*}
Then we have $M_0 \supsetneq M_1 \supset M_2 \supsetneq \varpi_E M_0$, and these $\mathfrak{o}_E$-lattices generate a self-dual $\mathfrak{o}_E$-lattice chain, $\overline{\mathcal{M}}$, in $W$. The $\mathfrak{o}_E$-period of $\overline{\mathcal{M}}$ is equal to 3, if $s \ne 0$, and to 2, if $s = 0$.

Let $\overline{\mathfrak{B}} = \mathrm{End}_{\mathfrak{o}_E}^0(\overline{\mathcal{M}})$ be the hereditary $\mathfrak{o}_E$-order in $\mathrm{End}_E(W)$ defined by $\overline{\mathcal{M}}$, 
and $\overline{\mathfrak{Q}}$ its Jacobson radical. In $\mathrm{End}_E(W) \cap \mathfrak{G}$, put
\begin{equation*}
\mathfrak{n} = \{\mathrm{Hom}_E(W_1^\natural \oplus W_1 \oplus W_2, W_2^\natural) \coprod \mathrm{Hom}_E(W_2, W_1^\natural \oplus W_1)\} \cap \mathfrak{G},
\end{equation*}
if $i = \overline{i}$, and put
\begin{equation*}
\mathfrak{n} = \{\mathrm{Hom}_E(W_1^\natural, W_2^\natural) \coprod \mathrm{Hom}_E(W_2, W_1)\} \cap \mathfrak{G},
\end{equation*}
if $i \ne \overline{i}$. Take any element
\begin{equation*}
x \in \mathfrak{n} \cap \overline{\mathfrak{B}} = \mathfrak{n} \cap \overline{\mathfrak{Q}}.
\end{equation*}

\begin{lem}    % Lem, 5.5.10
There is an integer $\ell$, with $0 \le \ell < e$, such that
\begin{eqnarray}
x(uwL_{\kappa+1} \cap W) &\subset& uwL_{\kappa+\ell+1} \cap W,\\
x(uwL_{\kappa+\ell+1} \cap W) &\subset& \varpi_E(uwL_{\kappa+1} \cap W).
\end{eqnarray}
\end{lem}

Since we have chosen the element $u \in \Vec{W}$ so as to have $b_1' \le b_{s+1}$, we have $b_1 \ge b_{s+1}'$ by Lemma 5.5.8. Thus $b_1 \ge b_{s+1}' \le b_s'$. To prove Lemma 5.5.10, we consider the following two cases:
\begin{enumerate}[{Case} 1.] 
 \item $b_{s+1}' < b_s'$, if $i \ne \overline{i}$, and $b_1 > b_{s+1}' < b_s'$, if $i = \overline{i}$,\
 \item $b_{s+1}' = b_s'$, if $i \ne \overline{i}$, and $b_1 = b_{s+1}'$ or $b_{s+1}' = b_s'$, if $i = \overline{i}$.
\end{enumerate}

In case 1, by definition, we see that $x(uwL_{\kappa+1} \cap W)$ is contained in
\begin{equation}
\begin{cases}
\displaystyle \sum_{j=s+1}^{f'} \mathfrak{p}_E^{b_s'}v_{-j} + \displaystyle \sum_{j=1}^s \mathfrak{p}_E^{b_{s+1}}v_j,\ \text{if}\ i \ne \overline{i},\\
\displaystyle \sum_{j=s+1}^{f'} \mathfrak{p}_E^{\mathrm{min}\{b_s',b_1\}}v_{-j} + \displaystyle \sum_{j=1}^s (\mathfrak{p}_E^{b_{s+1}}v_{-j} + \mathfrak{p}_E^{b_{s+1}}v_j),\ \text{if}\ i = \overline{i}
\end{cases}
\end{equation}
By Lemma 5.5.9, we have
\begin{eqnarray*}
b_{f'}' + 1 &\le& \cdots \le b_{s+1}' + 1 \le \mathrm{min}\{b_s',b_1\} \le b_s',\\
b_s' + 1 &\le& \cdots \le b_1' + 1 \le b_{s+1},\ \text{if}\ i = \overline{i},\\
b_1 + 1 &\le& \cdots \le b_s + 1 \le b_{s+1}.
\end{eqnarray*}
Hence we obtain $x(uwL_{\kappa+1} \cap W) \subset \varpi_E(uwL_{\kappa+1} \cap W)$, which is (5.5.5) with $\ell = 0$ in Lemma 5.5.10.

We consider case 2. For an integer $\ell$, $0 \le \ell < e$, put
\begin{equation*}
c_j = \mu(j,\kappa+\ell+1), c_j' = \mu'(j,\kappa+\ell+1)
\end{equation*}
for $j \in \{1, \cdots, f'\}$. Then we see that $x(uwL_{\kappa+\ell+1} \cap W)$ is contained in (5.5.6) in which $b_s', b_1$, and $b_{s+1}$ are replaced by $c_s', c_1$, and $c_{s+1}$, respectively. To prove (5.5.4), we have to prove the following inequalities:
\begin{enumerate} [({I}-1)]
 \item $c_{s+1}' \le b_s'$, if $i \ne \overline{i}$, and $c_{s+1}' \le \mathrm{min}\{b_1,b_s'\}$, if $i = \overline{i}$\
 \item $c_1' \le b_{s+1}$, if $i = \overline{i}$,\
 \item $c_s \le b_{s+1}$,
\end{enumerate}
and for (5.5.5),
\begin{enumerate}[({II}-1)]
 \item $b_{s+1}' < c_s'$, if $i \ne \overline{i}$, and $b_{s+1}' < \mathrm{min}\{c_1,c_s'\}$, if $i = \overline{i}$,\
 \item $b_1' < c_{s+1}$, if $i = \overline{i}$,\
 \item $b_s < c_{s+1}$.
\end{enumerate}
By Lemma 5.5.9, we easily obtain (I-2), (I-3), (II-2), and (II-3), for any integer $\ell$, $0 \le \ell <e$, in case 2. Thus it remains for us to prove that there is an integer $\ell$, $0 \le \ell <e$, such that (I-1) and (II-1) hold.

\begin{lem}    % Lem. 5.5.11
If $b_{s+1}' = b_s'$, then there is an integer $\ell$, $0 \le \ell < e$, such that $c_{s+1}' = b_{s+1}'$ and $c_s' = b_s' + 1$.
\end{lem}
\begin{proof}
Put $b_s' = m(a,\kappa+1) + d$, for some integers $a$ and $d$. Then $b_s = m(\overline{a},\kappa+1) - d$. On the other hand, $b_{s+1} = b_{f'} = m(\nu_0,\kappa+1) + d_0 = 1 + d_0$ and $b_{s+1}' = m(\overline{\nu_0},\kappa+1) - d_0$. From $b_s < b_{s+1}$ and $b_{s+1}' = b_s'$, we easily get $\overline{\nu_0} < a.$ For, if $\overline{\nu_0} = a$, then $\overline{a} = \nu_0$. It follows that $b_s < b_{s+1}$ implies $-d < d_0$ and that $b_s' = b_{s+1}'$ implies $d = -d_0$. This is a contradiction. Thus, if $\nu_0 \le \overline{\nu_0}$, then $\nu_0 \le \overline{\nu_0} < a.$ On the other hand, if $\overline{\nu_0} < \nu_0$, then we have $a < \nu_0.$ For, suppose $\nu_0 \le a$. Then $\overline{a} \le \overline{\nu_0}$, so that $m(\overline{\nu_0},\kappa+1) = m(\overline{a},\kappa+1) = 0$ and $m(a,\kappa+1) = 1.$ 
Thus, again from the above condition, we obtain $-d < 1+d_0$ and $-d_0 = 1 + d.$ This is a contradiction. Hence we have obtained
\begin{equation*}
\begin{cases}
\nu_0 \le \overline{\nu_0} < a,\hspace{3mm} \text{if}\ \nu_0 \le \overline{\nu_0},\\
\overline{\nu_0} < a < \nu_0,\hspace{3mm} \text{if}\ \overline{\nu_0} < \nu_0.
\end{cases}
\end{equation*}
It follows from Lemma 5.5.3 that $m(a,k)$ jumps at $k = \kappa + \ell + 1$ for some integer $\ell$, $0 \le \ell < e$, and that $m(\overline{\nu_0},k)$ is constant for $\kappa + 1 \le k \le \kappa + \ell + 1$. Hence the assertion follows.
\end{proof}

If $i \ne \overline{i}$, for the integer $\ell$ of Lemma 5.5.11, we have
\begin{equation*}
c_{s+1}' = b_{s+1}' = b_s' < b_s' + 1 = c_s'.
\end{equation*}
Thus (I-1) and (II-1) hold. Hence, in case 2 with $i \ne \overline{i}$, the proof of Lemma 5.5.10 is completed.

We prove this lemma in case 2 with $i = \overline{i}$. Let $i = \overline{i}$, and $b_1 = b_{s+1}'$ or $b_{s+1}' = b_s'$.

\begin{lem}   % Lem. 5.5.12
If $b_1 = b_{s+1}'$, then there is an integer $\ell$, $0 \le \ell < e$, such that $c_{s+1}' = b_{s+1}'$ and $c_1 = b_1 + 1$.
\end{lem}
\begin{proof}
The proof is quite similar to that of Lemma 5.5.11. We sketch the outline. Put $b_1 = m(a,\kappa+1) + d.$ Then $b_1' = m(\overline{a},\kappa+1) - d$. We have $b_{s+1} = 1 + d_0$ and $b_{s+1}' = m(\overline{\nu_0},\kappa+1) - d_0$. By Lemma 5.5.9 (2), we have $b_1 < b_{s+1}$ and $b_1 = b_{s+1}'$. Similarly, it follows that
\begin{equation*}
\begin{cases}
a \le \nu_0\ \text{or}\ \overline{\nu_0} < a,\hspace{3mm} \text{if}\ \nu_0 \le \overline{\nu_0},\\
\overline{\nu_0} < a \le \nu_0,\hspace{3mm} \text{if}\ \overline{\nu_0} < \nu_0.
\end{cases}
\end{equation*}
This shows the assertion.
\end{proof}

Denote by $\ell_1$ (resp. $\ell_2$) the integer $\ell$ in Lemma 5.5.11 (resp. Lemma 5.5.12). Put $\ell = \mathrm{max}\{\ell_1,\ell_2\}$. Then, for this $\ell$, we have $c_{s+1}' = b_{s+1}', c_s' = b_s' + 1$, and $c_1 = b_1 + 1$. Since $b_1 \ge b_{s+1}' \le b_s'$, we obtain $c_{s+1}' = b_{s+1}' \le \mathrm{min}\{b_1,b_s'\}$ (I-1). Further, $c_1 > b_1 \ge b_{s+1}' \le b_s' < c_s'$, so that $b_{s+1}' = c_{s+1}' < \mathrm{min}\{c_1,c_s'\}$ (II-1). Hence the proof of Lemma 5.5.10 is completed.

By Lemma 5.5.10, we have
\begin{equation}
(uw)^{-1}x(uw) \in \mathfrak{Q} = \mathrm{rad}(\mathfrak{B}),
\end{equation}
and by definition
\begin{equation}
u^{-1}xu \in u^{-1}\overline{\mathfrak{Q}}u = \mathrm{End}_{\mathfrak{o}_E}^1(u^{-1}\overline{\mathcal{M}})
\end{equation}
in $\mathrm{End}_E(W)$ as well.

Let $i =~ \overline{i}$. Then $u^{-1}\overline{\mathcal{M}}$ is a self-dual $\mathfrak{o}_E$-lattice chain in $W = V^i = V^{\overline{i}}$ of $\mathfrak{o}_E$-period equal to 2 or 3. Let $h = C(x) = (1 - \frac{1}{2}x)(1 + \frac{1}{2}x)^{-1}$ in $G$. Then from (5.5.7), we have $w^{-1}u^{-1}huw \in \Vec{U}^1(\mathfrak{B})$. Take an operator $T$ in $I_w(\sigma_- \vert \Vec{U}(\mathfrak{B}))$. Then it follows that
\begin{equation*}
\sigma_-(u^{-1}hu) \circ T = \sigma_-^w(w^{-1}u^{-1}huw) \circ T = T \circ \sigma_-(w^{-1}u^{-1}huw) = T.
\end{equation*}
In $B^i = \mathrm{End}_E(V^i)$, let $\mathfrak{B}^i = \mathrm{End}_{\mathfrak{o}_E}(\{L_k^i \vert k \in \mathbb{Z}\})$ and $\mathfrak{Q}^i$ its Jacobson radical. 
By the choice of the element $u$ of $\Vec{W}$, it follows from (5.5.8) that the set of $\{u^{-1}hu \vert h = C(x), x \in \mathfrak{n} \cap \overline{\mathfrak{Q}}\}$ projects onto the unipotent radical of a proper parabolic subgroup of $\Vec{U}(\mathfrak{B}^i)/\Vec{U}^1(\mathfrak{B}^i)$. 
Thus $\sigma_-(u^{-1}hu) \circ T = T$ above contradicts the cuspidality of $\sigma_1$. Hence the element $w$ never intertwines $\sigma_- \vert \Vec{U}(\mathfrak{B})$.

Let $i \ne \overline{i}$. Then $u^{-1}\overline{\mathcal{M}}$ is a self-dual $\mathfrak{o}_E$-chain in $W = V^i \oplus V^{\overline{i}}$ of $\mathfrak{o}_E$-period equal to 3. 
For the $\mathfrak{o}_E$-lattice chain $\overline{\mathcal{L}}$ in $V^i$ defined above, let $\overline{\mathfrak{B}}^i = \mathrm{End}_{\mathfrak{o}_E}^0(\overline{\mathcal{L}})$ and $\overline{\mathfrak{Q}}^i$ its Jacobson radical, in $B^i = \mathrm{End}_E(V^i)$. 
As an element $x \in \mathfrak{n} \cap \overline{\mathfrak{B}} = \mathfrak{n} \cap \overline{\mathfrak{Q}}$ above, we take $x = (x_1,x_1^\natural) \in (B^i)^\times \times (B^{\overline{i}})^\times$ and let $h = C(x)$. 
Then this is written in the form $(y,y')$, with $y = C(x_1) = 1 - x_1 \in \Vec{U}^1(\overline{\mathfrak{B}}^i)$. 
If $x_1$ varies,  the set of the $y = C(x_1)$'s projects onto $\Vec{U}^1(\overline{\mathfrak{B}}^i)/\Vec{U}^1(\mathfrak{B}^i)$. The quotient $\Vec{U}(\overline{\mathfrak{B}}^i)/\Vec{U}^1(\mathfrak{B}^i)$ is a proper parabolic subgroup of $\Vec{U}(\mathfrak{B}^i)/\Vec{U}^1(\mathfrak{B}^i)$, and $\Vec{U}^1(\overline{\mathfrak{B}}^i)/\Vec{U}^1(\mathfrak{B}^i)$ is its unipotent radical, as in the proof of \cite[5.5.7]{BK1}. 
Hence, similarly, we have $\sigma_-(u^{-1}hu) \circ T = T$ for $T \in I_w(\sigma_- \vert \Vec{U}(\mathfrak{B}))$, and this contradicts the cuspidality of $\sigma_0$. The proof of Lemma 5.5.2 is completed.
\end{proof}

\subsection{}                       % Sec. 5.6

From Proposition 5.5.1, we obtain an analogue of a maximal simple type for $GL(N,F)$ of \cite[(6.1)]{BK1} as follows:

\begin{thm}            % Thm. 5.6.1
Let $[\mathfrak{A},n,0,\beta]$ be a good skew simple stratum in $A$, with $\mathfrak{A} = \mathfrak{A}(\mathcal{L})$ principal, and $(J_-,\lambda_-)$ be a simple type in $G$ attached to $[\mathfrak{A},n,0,\beta]$. 
Let $\mathfrak{B}$ be the $\mathfrak{A}$-centralizer of $\beta$. Suppose that $\mathfrak{B}$ is maximal, i.e., $e(\mathcal{L}_{\mathfrak{o}_E}) = 1$. Then $(J_-, \lambda_-)$ is a $[G,\pi]_G$-type in $G$ for some irreducible supercuspidal representation $\pi$ of $G$, and $\pi$ is given by $\mathrm{Ind}(\lambda_-: J_-, G)$.
\end{thm}
\begin{proof}
From Proposition 5.1.1, we have $\Vec{W}(\mathfrak{B}) = \{1\}$, and from Proposition 5.5.1, $I_G(\lambda_-) \subset J_-$. Thus $\mathrm{Ind}(\lambda_-:J_-,G)$ is an irreducible supercuspidal representation of $G$ (cf. \cite[(1.5)]{Cy}). 
If an irreducible representation $\pi$ of $G$ contains $\lambda_-$, from Frobenius reciprocity (cf. \cite[(1.6)]{Cy}), $\pi$ is equivalent to $\mathrm{Ind}(\lambda_-:J_-,G)$. Hence the assertion follows from \cite[Section 2]{BK2} (cf. \cite[Definition 7.3]{Ro}).
\end{proof}

Such a simple type $(J_-,\lambda_-)$ in $G$ as in Theorem 5.6.1 is called a {\it supercuspidal type} in $G$.

Suppose that $[\mathfrak{A},n,0,\beta]$ is a good simple stratum in $A$, with $\mathfrak{A} = \mathfrak{A}(\mathcal{L})$ principal, and $\theta_- \in \mathcal{C}_-(\mathfrak{A},0,\beta)$. Let $E = F[\beta]$ and $e = e(\mathcal{L}_{\mathfrak{o}_E})$.

\begin{defn}              % Def. 5.6.2
Let $P = MN_u$ be a parabolic subgroup of $G$ associated with $[\mathfrak{A},n,0,\beta]$.
Let $(J_-,\lambda_-)$ be a simple type in $G$ attached to $[\mathfrak{A},n,0,\beta]$.
We write
\begin{equation*}
J_{P,-} = (J_- \cap P)H_-^1
\end{equation*}
as in 4.2, and define $\lambda_{P,-}$ to be the natural representation on the subspace of $(J_- \cap N_u)$-fixed vectors in the representation space of $\lambda_-$.
Moreover, we define a representation $(J_{P,-} \cap M,\lambda_{M,-})$ by
\begin{equation*}
\lambda_{M,-} = \lambda_{P,-} \vert (J_{P,-} \cap M).
\end{equation*}
\end{defn}

We note $J_{P,-} \cap M = J_- \cap M$. Put $t = [(e+1)/2]$. We have seen in 4.2 that
\begin{equation}
J_- \cap M = \displaystyle \prod_{i=1}^t J(\beta,\mathfrak{A}^{(i)}),
\end{equation}
where if $t = (e+1)/2 \in \mathbb{Z}$, we understand $J(\beta,\mathfrak{A}^{(t)}) = J_-(\beta,\mathfrak{A}^{(t)})$ in $U(V^t,h_t)$ (cf. Lemma 4.1.2). According to this decomposition, the representation $\lambda_{M,-}$ will be decomposed.

From Proposition 4.1.3, under the identification $H_-^1(\beta,\mathfrak{A}) = \prod_i H^1(\beta,\mathfrak{A}^{(i)})$, we have
\begin{equation*}
\theta_- = \theta^{(1)} \otimes \cdots \otimes \theta^{(t)},
\end{equation*}
where $\theta^{(i)} \in \mathcal{C}(\mathfrak{A}^{(i)},0,2\beta)$, $1 \le i \le t$, (cf. Proposition 4.2.1). From Proposition 3.1.2, there is a unique irreducible representation $\eta_-$ which contains $\theta_-$, 
and from Theorem 3.2.6, we have an irreducible representation $\kappa_-$ of $J_-$, which is a $\beta$-extension of $\eta_-$. From Proposition 4.2.2, we obtain $\eta_{P,-}$ of $J_{P,-}^1 = (J_-^1 \cap P)H_-^1$ such that
\begin{equation*}
\eta_{P,-} \vert (J_-^1 \cap M) \simeq \eta^{(1)} \otimes \cdots \otimes \eta^{(t)},
\end{equation*}
where $\eta^{(i)}$ is the unique irreducible representation of $J^1(\beta,\mathfrak{A}^{(i)})$ which contains $\theta^{(i)}$, and if $t = (e+1)/2 \in \mathbb{Z}$, we understand $J^1(\beta,\mathfrak{A}^{(t)}) = J_-^1(\beta,\mathfrak{A}^{(t)}),\ \eta^{(t)} = \eta_-^{(t)}$.

Let $\kappa_{P,-}$ be the natural representation on the subspace of $(J_-^1 \cap N_u)$-fixed vectors in the representation space of $\kappa_-$. Then, as in \cite[(7.2)]{BK1}, we obtain the results for $\kappa_{P,-}$ as follows: $\kappa_{P,-}$ is irreducible and $\kappa_{P,-} \vert J_{P,-}^1 = \eta_{P,-}$. We have
\begin{equation*}
\kappa_{P,-} \vert (J_- \cap M) \simeq \kappa^{(1)} \otimes \cdots \otimes \kappa^{(t)},
\end{equation*}
where $\kappa^{(i)}$ is an irreducible representation of $J(\beta,\mathfrak{A}^{(i)})$ and a $\beta$-extension of $\eta^{(i)}$, and if $t = (e+1)/2 \in \mathbb{Z}$, we understand $J(\beta,\mathfrak{A}^{(t)}) = J_-(\beta,\mathfrak{A}^{(t)}),\ \kappa^{(t)} = \kappa_-^{(t)}$.
Moreover, we have
\begin{equation*}
\kappa_- = \mathrm{Ind}(\kappa_{P,-}:(J_- \cap P)H_-^1,J_-).
\end{equation*}
By definition, elements of $\Vec{W}(\mathfrak{B})$ normalize the Levi subgroup $M$ of $G$ (cf. 4.1 and 5.1). We can easily show that the analogues of \cite[(7.2.10), (7.1.15)]{BK1} hold for $G$. Thus it follows from \cite[(7.2.16)]{BK1} that some element of $\Vec{W}(\mathfrak{B})$ may induce an equivalence $\kappa^{(i)} \simeq \kappa^{(j)}$. Hence we have
\begin{equation*}
\kappa^{(i)} \simeq \kappa^{(j)},\ \text{for}\ 1 \le i, j \le [e/2].
\end{equation*}
We note that the involution $^-$ on $A$ induces an involution on $J(\beta,\mathfrak{A}^{(i)})$, for $1 \le i \le t$, by (5.6.1). Furthermore, we have
\begin{equation*}
\kappa^{(i)} \simeq (\kappa^{(i)})^*,\ \text{for}\ 1 \le i \le t,
\end{equation*}
where $(\kappa^{(i)})^*(x) = \kappa^{(i)}(\overline{x}\ ^{-1}),$ for $x \in J(\beta,\mathfrak{A}^{(i)})$. This leads to
\begin{equation*}
\theta^{(i)} \simeq (\theta^{(i)})^*,\ \text{for}\ 1 \le i \le t.
\end{equation*}
In particular, if $t = (e+1)/2 \in \mathbb{Z}$, $\kappa^{(t)} = \kappa_-^{(t)}$, and automatically, $\kappa_-^{(t)} = (\kappa_-^{(t)})^*$, and  $\theta_-^{(t)} = (\theta_-^{(t)})^*$.

\begin{thm}[(7.2.17) of \cite{BK1}]          % Thm. 5.6.3
Let $[\mathfrak{A},n,0,\beta]$ be a good skew simple stratum in $A$, with $\mathfrak{A} = \mathfrak{A}(\mathcal{L})$ principal, and $(J_-,\lambda_-)$ be a simple type in $G$ attached to $[\mathfrak{A},n,0,\beta]$. 
Let $P = MN_u$ be a parabolic subgroup of $G$ associated with $[\mathfrak{A},n,0,\beta]$, and $(J_{P,-}, \lambda_{P,-})$, $(J_{P,-} \cap M, \lambda_{M,-})$ the representations in Definition 5.6.2. Then
\begin{enumerate}
  \item $\lambda_{P,-}$ and $\lambda_{M,-}$ are irreducible, and $\lambda_- \simeq \mathrm{Ind}(\lambda_{P,-}:J_{P,-},J_-)$;\
  \item Under the identification $J_{P,-} \cap M = \prod_i J(\beta,\mathfrak{A}^{(i)})$, for $1 \le i \le [e/2]$, there is a supercuspidal type $(J(\beta,\mathfrak{A}^{(i)}), \lambda^{(i)})$ in $\mathrm{Aut}_F(V^i)$, and if $t = (e+1)/2 \in \mathbb{Z}$,  there is a supercuspidal type $(J_-(\beta,\mathfrak{A}^{(t)}), \lambda_-^{(t)})$ in $U(V^t,h_t)$ such that\\
   \begin{equation*}
   \lambda_{M,-} \simeq \lambda^{(1)} \otimes \cdots \otimes \lambda^{(t)},
   \end{equation*}
   where we understand that $\lambda^{(t)}$ means $\lambda_-^{(t)}$ if $e$ is odd,
  \item For $1 \le i, j \le [e/2]$, $\lambda^{(i)} \simeq \lambda^{(j)}$. If the component $\sigma_0$ of $\sigma_-$ is self-dual, then $\lambda^{(i)} \simeq (\lambda^{(i)})^*$, for $1 \le i \le t$.
\end{enumerate}
\end{thm}
\begin{proof}
By the above argument, we can prove the theorem in the same way as the proof of \cite[(7.2.17)]{BK1}. In particular, for (3), we can similarly translate properties of $\kappa_-$ directly to $\lambda_-$, if the component $\sigma_0$ of $\sigma_-$ is self-dual.
\end{proof}

\begin{cor}         % Cor. 5.6.4
Let notation and assumptions be as in Theorem 5.6.3. Let $\pi_i$ be an irreducible supercuspidal representation of $\mathrm{Aut}_F(V^i)$ which contains $\lambda^{(i)}$, for $1 \le i \le [e/2]$, and when $t = (e+1)/2 \in\mathbb{Z}$, 
let $\pi_t$ be an irreducible supercuspidal representation of $U(V^t, h_t)$ which contains $\lambda_-^{(t)}$. We define an irreducible supercuspidal representation $\pi$ of the Levi subgroup $M$ of $G$ by
\begin{equation*}
\pi = \displaystyle \bigotimes^{[(e+1)/2]} \pi_i
\end{equation*}
Then $(J_{P,-} \cap M, \lambda_{M,-})$ is an $[M,\pi]_M$-type in $M$. 
\end{cor}
\begin{proof}
This follows directly from \cite[(6.2.2)]{BK1} and Theorem 5.6.3 (cf. \cite[Proposition 1.3]{BK3}).
\end{proof}

\begin{rem}        % Rem. 5.6.5
Let $\pi$ be an irreducible supercuspidal representation of $M$ as in Corollary 5.6.4. If the component $\sigma_0$ of $\sigma_-$, with $\lambda_- = \kappa_- \otimes \sigma_-$, is self-dual, the contragradient representation of $\pi$ belongs to $[M,\pi]_M$, and this inertial class contains a self-contragradient representation of $M$. This follows from Theorem 5.6.3 and statements in \cite[2.2 and Introduction]{Bl}.
\end{rem}

\section{Hecke algebras and Types}                        % Sec. 6

\subsection{}                     % Sec. 6.1

In this section, we prove that $(J_{P,-},\lambda_{P,-})$ is a type in $G$. To do so, we study the Hecke algebras $\mathcal{H}(G,\lambda_{P,-})$ of $(J_{P,-},\lambda_{P,-})$.

Suppose that $[\mathfrak{A},n,0,\beta]$ is a good simple stratum in $A$, with $\mathfrak{A} = \mathfrak{A}(\mathcal{L})$ principal, and $(J_-,\lambda_-)$ a simple type in $G$ attached to $[\mathfrak{A},n,0,\beta]$, with $\lambda_- = \kappa_- \otimes \sigma_-$. 
Let $E = F[\beta]$, $B = B_\beta$ the $A$-centralizer of $\beta$, and $\mathfrak{B} = \mathfrak{A} \cap B$.

\begin{prop}[(7.2.19) of \cite{BK1}]        % Prop. 6.1.1
Let $\lambda_{M,-}$ be the representation of $J_{P,-} \cap M$ which is the restriction of $\lambda_{P,-}$ as in Definition 5.6.2, and $\Vec{W}(\sigma_-)$ be the subgroup of $\Vec{W}(\mathfrak{B})$ defined by (5.3.1). Let $w$ be an element of $\Vec{W}(\mathfrak{B})$. Then $I_w(\lambda_{P,-}) = I_w(\lambda_{M,-})$, and if $w \in \Vec{W}(\sigma_-)$, its dimension is equal to one.
\end{prop}
\begin{proof}
As stated in 5.6, $\Vec{W}(\mathfrak{B})$ normalizes $J_- \cap M$. Take a representative, $y \in \Vec{N}(E_0) \subset B^\times \cap G$, of $w$ (see 5.1).
Clearly $I_y(\lambda_{P,-}) = I_w(\lambda_{P,-}) \subset I_y(\lambda_{M,-})$. We show the converse inclusion. For $GL(N,D)$ with $D$ a central division $F$-algebra, we have an Iwahori decomposition of $J_P$ in the proof of \cite[Theorem 2.19]{Sc2}. Similarly we obtain 
\begin{equation}
J_{P,-} = (J_{P,-} \cap\ ^yN_\ell)(J_{P,-} \cap M)(J_{P,-} \cap\ ^yN_u). 
\end{equation}
The subgroups $\widetilde{N}_\ell$ and $\widetilde{N}_u$ of $\widetilde{G}$, defined in 4.1, are denoted by $U^-$ and $U$ respectively in the proof. We have
\begin{equation}
(\widetilde{N}_\ell\widetilde{M}\widetilde{N}_u)^\Gamma = \widetilde{N}_\ell^\Gamma\widetilde{M}^\Gamma\widetilde{N}_u^\Gamma = N_\ell MN_u.
\end{equation}
In the proof of \cite[Theorem 2.19]{Sc2}, replacing $J_P, \kappa_M$ and $\kappa_P$ by $J_{P,-}, \lambda_{M,-}$ and $\lambda_{P,-}$ respectively, we imitate the proof to prove $I_y(\lambda_{M,-}) \subset I_y(\lambda_{P,-})$ by using (6.1.1) and (6.1.2). Hence the first assertion follows.

Suppose that $w \in \Vec{W}(\sigma_-)$. Then, since by definition $(\sigma_-)^y \simeq \sigma_-$, it follows from Theorem 5.6.3 (3) that the element $y$ stabilizes $\lambda_{M,-}$ (cf. the proof of \cite[(7.2.19)]{BK1}). Thus the space $I_y(\lambda_{M,-}) = I_w(\lambda_{M,-})$ has dimension one.  The proof is completed.
\end{proof}

Let $P = MN_u$ be a parabolic subgroup of $G$ associated with $[\mathfrak{A},n,0,\beta]$, and $(J_{P,-}, \lambda_{P,-})$ the representation obtained from $(J_-, \lambda_-)$ in Definition 5.6.2. Let $\mathcal{H}(G,\lambda_-)$ be the Hecke algebra of $(J_-,\lambda_-)$ (cf. \cite[4.1]{BK1}). 
From Theorem 5.6.3 (1) and \cite[(4.1.3)]{BK1}, there is a canonical algebra isomorphism
\begin{equation}
\mathcal{H}(G, \lambda_-) \simeq \mathcal{H}(G, \lambda_{P,-}).
\end{equation}

\begin{prop}            % Prop. 6.1.2
The Hecke algebra $\mathcal{H}(G,\lambda_-)$ is spanned by functions with support $J_-wJ_-$, $w \in \Vec{W}(\sigma_-)$, as a $\mathbb{C}$-vector space, and the isomorphism of (6.1.3) is support-preserving.
\end{prop}
\begin{proof}
From Proposition 5.5.1, the Hecke algebra $\mathcal{H}(G,\lambda_-)$ is spanned by functions with support $J_-wJ_-$, $w \in \Vec{W}(\mathfrak{B})$, as a $\mathbb{C}$-vector space.
For $w \in \Vec{W}(\mathfrak{B})$, we can show that the dimension of $I_w(\lambda_-)$ is at most one, in a quite similar way to the proof of \cite[(5.6.15)]{BK1}. If $w$ intertwines $\lambda_-$, the space $I_w(\lambda_-)$ has one dimension. Thus it follows from \cite[(4.1.5)]{BK1} that $w$ intertwines $\lambda_{P,-}$. Since $I_w(\lambda_{P,-}) = I_w(\lambda_{M,-})$ by Proposition 6.1.1, it intertwines $\lambda_{M,-}$ as well. Hence, from Theorem 5.6.3 (3), we see that $w \in \Vec{W}(\sigma_-)$ and that $\mathcal{H}(G,\lambda_-)$ is spanned by functions with support $J_-wJ_-$, $w \in \Vec{W}(\sigma_-)$. For $w \in \Vec{W}(\sigma_-)$, again from \cite[(4.1.5)]{BK1} and Proposition 6.1.1, we see that the spaces $I_w(\lambda_-)$ and $I_w(\lambda_{P,-})$ are both of one dimensional. Thus the algebra isomorphism (6.1.3) is support-preserving. The proof is completed.
\end{proof}

We may identify $\mathcal{H}(G, \lambda_{P,-})$ with $\mathcal{H}(G, \lambda_-)$ via the isomorphism (6.1.3). 
Let $E = F[\beta]$, $B = B_\beta$ the $A$-centralizer of $\beta$, and $\mathfrak{B} = \mathfrak{A} \cap B.$ Let $\Vec{D}(\mathfrak{B})$ be the subgroup of $B^\times \cap G$ defined in 5.1. Let $e = e(\mathcal{L}_{\mathfrak{o}_E})$ and $e' = [e/2]$.
We define $\Vec{D}^-(\mathfrak{B})$ to be a submonoid of $\Vec{D}(\mathfrak{B})$ which consists of elements whose eigenvalues are $\varpi_E^{n_1}, \cdots, \varpi_E^{n_{e'}}, \varpi_E^{-n_{e'}}, \cdots, \varpi_E^{-n_1}$ with 
\begin{equation*}
n_1 \ge \cdots \ge n_{e'}
\end{equation*}
if $e$ is even, and whose eigenvalues are those, together with $1$, if $e$ is odd.

\begin{lem}             % Lem. 6.1.3
Let $\lambda_{M,-}$ be the representation of $J_{P,-} \cap M$ as above. Then the Hecke algebra $\mathcal{H}(M, \lambda_{M,-})$ is isomorphic to the Laurent polynomial ring 
\begin{equation*}
\mathbb{C}[X_1, \cdots, X_{[e/2]};X_1^{-1}, \cdots, X_{[e/2]}^{-1}].
\end{equation*}
\end{lem}
\begin{proof}
From Theorem 5.6.3, $\lambda_{M,-} \simeq \lambda^{(1)} \otimes \cdots \otimes \lambda^{(t)}$, where $t = [(e+1)/2]$. If $t = (e+1)/2 \in \mathbb{Z}$, $\lambda^{(t)} = \lambda_-^{(t)}$ is a supercuspidal type in $U(V^t, h_t)$. Thus from Theorem 5.6.1, we have
\begin{equation*}
\mathcal{H}(U(V^t,h_t), \lambda_-^{(t)}) \simeq \mathbb{C}.
\end{equation*}
However, since $\lambda^{(i)}$, $1 \le i \le [e/2]$, is a maximal simple type in $\mathrm{Aut}_F(V^i)$, from \cite[(7.6.3)]{BK1}, we have
\begin{equation*}
\mathcal{H}(\mathrm{Aut}_F(V^i), \lambda^{(i)}) \simeq \mathbb{C}[X,X^{-1}].
\end{equation*}
Put $e' = [e/2]$. Hence we obtain
\begin{eqnarray*}
\mathcal{H}(M, \lambda_{M,-}) &\simeq& \mathcal{H}(\mathrm{Aut}_F(V^1), \lambda^{(1)}) \otimes \cdots \otimes \mathcal{H}(\mathrm{Aut}_F(V^{e'}), \lambda^{(e')})\\
                              &\simeq& \mathbb{C}[X_1,X_1^{-1}] \otimes \cdots \otimes \mathbb{C}[X_{e'},X_{e'}^{-1}]\\
                              &\simeq& \mathbb{C}[X_1, \cdots, X_{e'};X_1^{-1}, \cdots, X_{e'}^{-1}].
\end{eqnarray*}
The proof is completed.
\end{proof}

\begin{prop}         % Prop. 6.1.4
There is an injective homomorphism
\begin{equation*}
j_P: \mathcal{H}(M,\lambda_{M,-}) \to \mathcal{H}(G,\lambda_{P,-})
\end{equation*}
such that for $z \in \Vec{D}^-(\mathfrak{B})$ and $\phi \in \mathcal{H}(M,\lambda_{M,-})$ with support $(J_- \cap M)z$, the support of $j_P(\phi)$ is $J_{P,-}zJ_{P,-}$, and $j_P(\phi)(z) = \phi(z)$.
\end{prop}
\begin{proof}
Identify $\mathcal{H}(G,\lambda_-) = \mathcal{H}(G,\lambda_{P,-})$ as above. Since $\Vec{D}^-(\mathfrak{B}) \subset \Vec{W}(\sigma_-)$, it follows from Proposition 6.1.1 that for each $z \in \Vec{D}^-(\mathfrak{B})$, there is a function of $\mathcal{H}(G,\lambda_{P,-})$ supported on $J_{P,-}zJ_{P,-}$. Hence the proposition is proved in a quite similar way to the proof of \cite[(7.6.2)]{BK1}.
\end{proof}

\subsection{}                 % Sec. 6.2

Suppose that $(J_-,\lambda_-)$, with $\lambda_- = \kappa_- \otimes \sigma_-$, is a simple type in $G$ attached to a good skew simple stratum $[\mathfrak{A},n,0,\beta]$, with $\mathfrak{A} = \mathfrak{A}(\mathcal{L})$ principal. Let $P = MN_u$ be a parabolic subgroup $G$ associated with $[\mathfrak{A},n,0,\beta]$, and $(J_{P,-},\lambda_{P,-})$ the natural representation defined by $(J_-,\lambda_-)$. 
Then, from Corollary 5.6.4, there is an irreducible supercuspidal representation $\pi$ of $M$, which is of the form $\bigotimes^{e/2} \pi_0$, $\bigotimes^{(e-1)/2} \pi_0 \otimes \pi_1$, according to $e = e(\mathcal{L}_{\mathfrak{o}_E}) \equiv 0,\ 1\ (\mathrm{mod}\ 2)$, such that $(J_{P,-} \cap M,\lambda_{M,-})$ is an $[M,\pi]_M$-type in $M$.
Moreover, the representation satisfies the following conditions:
\begin{enumerate}
  \item $(J_{P,-},\lambda_{P,-})$ is a decomposed pair with respect to $(M,P)$, i.e., 
  \begin{equation*}
  J_{P,-} = (J_{P,-} \cap N_\ell)(J_- \cap M)(J_{P,-} \cap N_u),
  \end{equation*}
  and $\lambda_{P,-}$ is trivial on both $J_{P,-} \cap N_\ell$ and $J_{P,-} \cap N_u$.
  \item $\lambda_{M,-} = \lambda_{P,-} \vert (J_{P,-} \cap M)$.
\end{enumerate}

\begin{lem}        % Lem. 6.2.1
Let notation and assumptions be as above. Then there is an invertible element, $\xi$, of $\mathcal{H}(G,\lambda_{P,-})$ supported on the double coset $J_{P,-}z_PJ_{P,-}$, where $z_P$ is an element of the center, $Z(M)$, of $M$, and $\xi$ is a strongly $(P,J_{P,-})$-positive element.
\end{lem}
\begin{proof}
For an integer $j$, $1 \le j \le [e/2]$, we put
\begin{equation*}
\Vec{a}_j = \mathrm{Diag}(\varpi_EI, \cdots, \varpi_EI, I, \cdots, I, \varpi_E^{-1}I, \cdots, \varpi_E^{-1}I),
\end{equation*}
where $\varpi_EI$ (resp. $\varpi_E^{-1}I$) appears $j$ times. Then these are elements of $\Vec{D}^-(\mathfrak{B})$, and for each an integer $i$, $1 \le i \le [e/2]$, there is a non-zero function $X_i$ in $\mathcal{H}(M,\lambda_{M,-})$ supported on $(J_{P,-} \cap M)\Vec{a}_i$, as in the proof of \cite[(7.6.2)]{BK1}. This element $X_i$ is the same as that of Lemma 6.1.3 (cf. \cite[p. 245]{BK1}) and is invertible in $\mathcal{H}(M,\lambda_{M,-})$. Put $e_0 = e(E \vert F)$ and
\begin{equation*}
Z_P = X_1^{e_0}X_2^{e_0} \cdots X_{[e/2]}^{e_0}
\end{equation*}
in $\mathcal{H}(M,\lambda_{M,-})$. Then the function $Z_P$ is supported on $(J_{P,-} \cap M)z_P$, with $z_P = \Vec{a}_1^{e_0}\Vec{a}_2^{e_0} \cdots \Vec{a}_{[e/2]}^{e_0}$, and it is invertible in $\mathcal{H}(M,\lambda_{M,-})$. It is easy to see $z_P \in Z(M)$.  Put $\xi = j_P(Z_P) \in \mathcal{H}(G,\lambda_{P,-})$. Then it follows from Proposition 6.1.4 that the function $\xi$ is supported on $J_{P,-}z_PJ_{P,-}$ and is invertible. Hence the proof is completed.
\end{proof}

\begin{thm}        % Thm. 6.2.2
Let $[\mathfrak{A},n,0,\beta]$ be a good skew simple stratum in $A$, with $\mathfrak{A}$ principal, and $(J_-,\lambda_-)$ a simple type in $G$ attached to $[\mathfrak{A},n,0,\beta]$. 
Let $(J_{P,-},\lambda_{P,-})$ be the representation defined in Definition 5.6.2 from $(J_-,\lambda_-)$, and $\pi$ an irreducible supercuspidal representation of $M$ as in Corollary 5.6.4. Then $(J_{P,-},\lambda_{P,-})$ is an $[M,\pi]_G$-type in $G$, and so is $(J_-,\lambda_-)$.
\end{thm}
\begin{proof}
From the conditions (1), (2) and Lemma 6.2.1, $(J_{P,-},\lambda_{P,-})$ satisfy the hypotheses of \cite[(7.9)]{BK2}. Thus, (3) for any smooth irreducible representation $(\mu, \mathcal{V})$ of $G$, the restriction to $\mathcal{V}^{\lambda_{P,-}}$ of the Jacquet functor $r_u$ is injective. 
The definition of $G$-cover, given in \cite[(8.1)]{BK2}, is modified so that if the conditions (1), (2) and (3) are satisfied for one parabolic subgroup $P$, then $(J_{P.-},\lambda_{P,-})$ is a $G$-cover of $(J_{P,-} \cap M,\lambda_{M,-})$ (see \cite[Introduction]{Bl}). This modification follows from \cite{Bu}. 
Since $(J_{P,-} \cap M,\lambda_{M,-})$ is an $[M,\pi]_M$-type in $M$, the theorem follows from \cite[(8.3)]{BK2}. Moreover, since 
\begin{equation*}
\lambda_- \simeq  \mathrm{Ind}(\lambda_{P,-}:J_{P,-}, J_-)
\end{equation*}
by Theorem 5.6.3 (1), it is easy to see that $(J_-,\lambda_-)$ is also an $[M,\pi]_G$-type in $G$ (cf. \cite[5.3]{Sc3}).
\end{proof}

{\bf Acknowledgements}

I would like to thank the referee for pointing out many mistakes of the previous draft of the manuscript and giving me useful suggestions for the proofs of Lemma 5.2.2 and Propositions 6.1.1 and 6.2.2, and Dr Shin-ichi Kato for his advice of the arguments of 5.4.

\renewcommand{\labelenumi}{\theenumi}

\vspace{4mm}

{\small {\it Department of Economics, Management

and Information Science

Onomichi University

Onomichi 722-8506, JAPAN

e-mail: kariyama@onomichi-u.ac.jp}


\begin{thebibliography}{99}


\bibitem{Ad1} Adler J.D.: {\it Self-contragradient supercuspidal representation of $GL_n$}, Proc. Amer. Math. Soc.(8) {\bf 125}(1997) 2471-2479.
\bibitem{Ad2} Adler J.D.: {\it Refined anisotropic $K$-types and supercuspidal representations}, Pacific J. Math. {\bf 185} no. 1 (1998) 1-32.
\bibitem{Bl} Blondel C.: {\it $Sp(2N)$-Covers for self-contragradient supercuspidal representations of $GL(N)$}, Ann. Sci. \'Ecole. Norm. Sup. (4) {\bf 37}(2004) 533-558.
\bibitem{Bu} Bushnell C. J.: {\it Representations of reductive $p$-adic groups: localization of Hecke algebras and applications}, J. London Math. Soc. (2) {\bf 63}(2001) 364-386.
\bibitem{BK1} Bushnell C. J. and Kutzko P.:  {\it The Admissible Dual of GL($N$) Via Compact Open Subgroups}, Ann. Math. Stud. 129, Princeton Univ. Press 1993.
\bibitem{BK2} Bushnell C. J. and Kutzko P.: {\it Smooth representations of reductive $p$-adic groups: structure theory via types}, Proc. London Math. Soc. {\bf 77} (1998) 582-634.
\bibitem{BK3} Bushnell C. J. and Kutzko P.: {\it Semisimple types in $GL_n$}, Compositio Math. {\bf 119} (1999) 53-97.
\bibitem{BT} Bruhat F. and Tits J.: {\it Groupes reductifs sur un corps local, I:Donnees radicielles valuees,} Publ. Math. I.H.E.S. {\bf 41} (1972) 5-252.

\bibitem{Cy} Carayol H.: {\it Repr\'esentations cuspidales du groupe lin\'eaire}, Ann. Sci. \'Ecole Norm. Sup. (4) {\bf 17} (1984) 191-225.
\bibitem{Ca} Carter R. W.: {\it Finite groups of Lie type, Conjugacy Classes and Complex Characters}, Wiley-Interscience, 1985.
\bibitem{Ca2} Carter R. W.: {\it Simple groups of Lie type,} John Wiley and Sons Ltd, 1972.
\bibitem{Cs} Casselmann. W.: {\it Introduction to the theory of admissible representations of $\mathfrak{p}$-adic reductive groups}, preprint, University of British Columbus, 1974.
\bibitem{DL} Deligne P. and Lusztig G.: {\it Representations of reductive groups over finite fields,} Ann. Math. {\bf 103} (1976) 103-161.
\bibitem{GK} Gelfand I.M. and Kazhdan D.A.: {\it Representations of the group $GL(n,K)$ where $K$ is a local field, in: Lie Groups and the Representations (Proc. Summer School of the Bolya-Janos Math. Soc., Budapest, 1971),} Halsted, New York, 1975.
\bibitem{HL} Howellet R.B. and Lehrer G.I.: {\it Induced cuspidal representations and generalized Hecke rings}, Invent. Math. {\bf 58} (1980) 37-64.
\bibitem{IM} Iwahori N. and Matsumoto H.: {\it On some Bruhat decomposition and the structure of the Hecke rings of $p$-adic Chevalley groups,} Publ. Math. I.H.E.S. {\bf 25} (1965) 5-48.
\bibitem{Ka} Kariyama K.: {\it Very cuspidal representations of $p$-adic symplectic groups}, J. Algebra {\bf 207} (1998) 205-255.
\bibitem{Mo1} Morris L.: {\it Tamely ramified supercuspidal representations of symplectic groups}, Proc. London Math. Soc. (3) {\bf 63} (1991) 519-551.
\bibitem{Mo2} Morris L.: {\it Tamely ramified superecuspidal representations of classical groups. I. Filtrations}, Ann. Sci. \'Ecole. Norm. Sup. (4) {\bf 24}(1991) 705-738.
\bibitem{MP} Moy A. and Prasad. P.: {\it Unramified minimal $K$-types for $p$-adic groups}, Invent. Math. {\bf 116} (1994), 393-408.
\bibitem{Ro} Roche A.: {\it Types and Hecke algebras for principal series representations of split reductive $p$-adic groups}, Ann. Sci. \'Ecole Norm. Sup. (4) {\bf 34} (1998) 361-423.
\bibitem{Se} Serre J.-P.: {\it Cohomologie Galoisienne}, Lecture Notes in Mathematics, {\bf 5}, Springer, Berlin, 1964.
\bibitem{Sc1} S\'echerre V.: {\it Repr\'esentations lisse de $GL(m,D)$, I : caracter\`es simples}, Bull. Soc. Math. France {\bf 132} (3) (2004), 327-396.
\bibitem{Sc2} S\'echerre V.: {\it Repr\'esentations lisse de $GL(m,D)$, II : $\beta$-extensions}, Compositio Math. {\bf 141} (2005) 1531-1550.
\bibitem{Sc3} S\'echerre V.: {\it Repr\'esentations lisse de $GL(m,D)$, III : types simples}, Ann. Sci. \'Ecole Norm. Sup. (4) {\bf 38} (2005) 951-977.
\bibitem{SS} Springer T. A. and Steinberg R.: {\it Conjugacy classes, in Seminar on algebraic groups and related finite groups, eds A. Borel and et al}, Lecture Notes in Mathematics, {\bf 131}, Springer, Berlin, 1970.
\bibitem{St} Steinberg R.: {\it Lectures on Chevalley groups}, Yale University, 1967.
\bibitem{St1} Stevens S.: {\it Double coset decompositions and intertwining},  Manuscripta Math. {\bf 106}(3) (2001) 349-364.
\bibitem{St2} Stevens S.: {\it Intertwining and supercuspidal types for $p$-adic classical groups}, Proc. London Math. Soc. (3) {\bf 83} (2001) 120-140.
\bibitem{St3} Stevens S.: {\it Semisimple characters for $p$-adic classical groups},  Duke Math. J. {\bf 127} no.1 (2005) 123-173.
\bibitem{Ti} Tits J.: {\it Reductive groups over local fields,} in Proc. Pure Math. {\bf 33}, Vol. 1, AMS, Providence, 1979, 29-69
\bibitem{Yu} Yu J.-K.: {\it Construction of tame supercuspidal representations}, J. Amer. Math. Soc. {\bf14} (3) (2001) 579-622.
\bibitem{We} Weil A.: {\it Basic Number Theory}, Grundlehren Math. Wiss., Bd. {\bf 144}, Springer-Verlag, Berlin and New York, 1997.
\end{thebibliography}
\end{document}